\documentclass[12pt]{article}
\usepackage{color}
\usepackage{graphicx}
\usepackage{amsmath}
\usepackage{cases}
\usepackage{amsthm}
\usepackage{amssymb}
\usepackage{bm}
\usepackage{cases}
\usepackage{amsfonts}
\usepackage{latexsym,enumerate,amscd,verbatim}
\usepackage{mathrsfs}
\usepackage[pdftex,colorlinks=true,bookmarks=false,citecolor=blue]{hyperref}

\allowdisplaybreaks
\numberwithin{equation}{section}

\newtheorem{theorem}{Theorem}[section]
\newtheorem{lemma}{Lemma}[section]

\newtheorem{corollary}{Corollary}[section]
\newtheorem{definition}{Definition}[section]

\newtheorem{example}{Example}[section]
\newtheorem{remark}{Remark}[section]

\makeatletter
\def\mr{\mathbb{R}}
\def\mrn{\mathbb{R}^n}
\def\mrm{\mathbb{R}^m}
\def\mn{\mathbb{N}}
\def\mrd{\mathbb{R}^d}

\def\mmf{\mathbb{F}}
\def\mf{\mathcal{F}}
\def\mg{\mathcal{G}}
\def\mb{\mathcal{B}}

\def\ms{\mathbb{S}}
\def\mg{\mathscr{G}}

\def\mt{\mathcal{T}}
\def\md{\mathbb{D}}
\def\ml{\mathbb{L}}

\def\me{\mathbb{E}}
\def\mk{\mathcal{K}}
\def\mmu{\mathcal{U}_{ad}}

\def\t{\tau}
\def\ups{\Upsilon_{\bar{u}}}
\def\ma{\mathcal{A}_{\bar{u}}}

\def\dd{\mathcal{D}}
\def\bu{\bar{u}}
\def\bx{\bar{x}}
\def\N{N^{b}}
\def\NN{N^{b(2)}}
\def\T{T^{b}}
\def\TT{T^{b(2)}}
\def\eps{\varepsilon}
\def\t{\tau}

\newcommand{\inner}[2]{\left\langle#1,#2\right\rangle}

\renewcommand{\@seccntformat}[1]{\csname the#1\endcsname.\hspace{0.5em}}
\makeatother

\title{First and second order necessary conditions for stochastic optimal controls}

\author{H\'el\`ene Frankowska\thanks{CNRS, IMJ-PRG, UMR 7586, Sorbonne Universit\'es,  UPMC Univ Paris 06, Univ Paris Diderot,
case 247, 4 place Jussieu,
75252 Paris, France.  The research of this
author is partially supported by the Gaspard Monge Program for Optimisation and Operational Research, Jacques Hadamard Mathematical Foundation (FMJH). {\small\it E-mail:} {\small\tt helene.frankowska@imj-prg.fr}.},~~~
Haisen Zhang\thanks{School of Mathematics and Statistics, Southwest University, Chongqing 400715,  China. The research of this
author is partially supported by NSF of China under grants 11401404 and 11471231, the fundamental research funds for the central universities under grants SWU114074 and XDJK2015C142. {\small\it E-mail:} {\small\tt haisenzhang@yeah.net}.}~~~and~~~
Xu Zhang\thanks{School of Mathematics, Sichuan University, Chengdu 610064, China. The research of this
author is partially supported by NSF of China under grant 11231007 and the Chang Jiang
Scholars Program from the Chinese Education Ministry. {\small\it E-mail:} {\small\tt
zhang$\_$xu@scu.edu.cn}.}}

\date{}

\begin{document}
\maketitle

\begin{abstract}
The main purpose of this paper is to establish the first and  second order necessary optimality conditions for stochastic optimal controls using the classical variational analysis approach. The control system is governed by a stochastic differential equation, in which both drift and diffusion terms may contain the control variable and the set of controls  is allowed to be nonconvex. Only one adjoint equation is introduced to derive the first order necessary condition; while only two adjoint equations are needed to state the second order necessary conditions for stochastic optimal controls.
\end{abstract}

\vspace{+0.3em}

\noindent {\bf Key words:}
Stochastic optimal control, Malliavin calculus,
necessary conditions, adjacent cone, variational equation, adjoint equation.

\vspace{+0.3em}

\noindent {\bf AMS subject classifications:}
Primary 93E20; Secondary  49J53, 60H07, 60H10.

\section{Introduction}
Let $T>0$ and $(\Omega,\mf, \mmf,$ $P)$  be a complete filtered
probability space (satisfying the usual conditions),
on which a $1$-dimensional standard Wiener
process $W(\cdot)$ is defined such that $\mmf=\{\mf_{t} \}_{0\le t\le T}$ is the natural filtration generated by $W(\cdot)$ (augmented by all the $P$-null sets).

Let us consider the following controlled stochastic differential equation
\begin{equation}\label{controlsys}
\left\{
\begin{array}{l}
dx(t)=b(t,x(t),u(t))dt+\sigma(t,x(t),u(t))dW(t),\ \ \ t\in[0,T],\\
x(0)=x_0\in K,
\end{array}\right.
\end{equation}
with the cost functional
\begin{equation}\label{costfunction}
J(u(\cdot), x_0)=\me\Big[\int_{0}^{T}f(t,x(t),u(t))dt+g(x(T))\Big].
\end{equation}
Here $u(\cdot)$ is the control variable with values in a closed nonempty subset $U$ of $ \mrm$ (for some fixed $m\in \mn$), $x(\cdot)$ is the state variable with values in $\mr^n$ (for some given $n\in \mn$), $K$ is a closed nonempty subset in $\mrn$, and $b,\sigma:[0,T]\times \mrn\times \mrm\times\Omega\to \mrn$, $f:[0,T]\times \mrn\times \mrm\times\Omega\to \mr$ and $g:\mrn\times\Omega\to \mr$ are given functions (satisfying suitable conditions to be stated later). As usual, when the context is clear, we omit the $\omega$ ($\in \Omega$) argument in the defined functions.

Denote by $\inner{\cdot}{\cdot}$ and $|\cdot|$ respectively the inner product and norm in $\mrn$ or $\mrm$, which can be identified from the contexts, by $\mb(X)$ the Borel $\sigma$-field of a metric space $X$, and by $\mmu$ the set of $\mb([0,T])\otimes\mf$-measurable and $\mmf$-adapted stochastic processes with values in $U$ such that $\me\int_{0}^{T}|u(t,\omega)|^2dt<\infty$. Any $u(\cdot)\in \mmu$ is called an admissible control, the corresponding state $x(\cdot;x_{0})$ of (\ref{controlsys}) with initial datum $x_{0}\in K$ is called an admissible state,  and $(x,u,x_{0})$ is called an admissible triple.
An admissible triple  $(\bar{x},\bar u,\bar x_{0})$  is called  optimal if
\begin{equation}\label{minimum J}
J(\bar{u}(\cdot),\bar x_0)=\inf_{\substack{u(\cdot)\in \mmu\\ x_{0}\in K}}J(u(\cdot),x_{0}).
\end{equation}

The purpose of this paper is to establish first and second order necessary optimality conditions  for  problem (\ref{minimum J}). We refer to  \cite{Bensoussan81, Bismut78, Haussmann76, Kushner72} and references cited therein for some early works on this subject.
Although the stochastic optimal control theory was developing almost simultaneously with the deterministic one, its results are much less fruitful than those obtained for the deterministic control systems. The main reasons are due to some essential difficulties (or new phenomena) when the diffusion term of the stochastic control system depends on  the control variable and the control region lacks convexity. In contrast with the deterministic case, for stochastic optimal control problems when spike variations  are used as  perturbations, the cost functional needs to be expanded up to the {\em second order} and {\em two adjoint equations} have to be introduced  to derive  the {\em first order necessary optimality conditions}. A stochastic maximum principle for this general case was established in \cite{Peng90}. On the other hand, to derive the second order necessary optimality conditions, the cost functional needs to be expanded up to the forth order and four adjoint equations have to be introduced, see \cite{zhangH14b}. Consequently, these necessary conditions narrow the field of applications, since they require so many adjoint equations and considerably strong smoothness assumptions (with respect to the state variable $x$) on the coefficients of the control system and the cost functional.

Can we use just one  adjoint equation (resp. two adjoint equations) to derive a first (resp. second) order necessary condition for the above general stochastic optimal control problem? To answer this question, let us first turn back to the special case of convex control constraint.
When the control region is convex, the usual convex variation can be used to construct a control perturbation. Only one adjoint equation is needed to establish the first order necessary condition (see \cite{Bensoussan81}) and two adjoint equations are needed to establish the second order necessary condition (see \cite{zhangH14a}) for stochastic optimal controls. The main advantage of using the convex variations instead of the
spike ones, is the fact  that, it avoids efficiently the
difficulties brought by perturbations with respect to the measure.
However, when the control region is nonconvex, the traditional
convex variations cannot be used, since there may exist a control
$u(\cdot)$ in the set of admissible controls $\mmu$ such that
$v:=u-\bu$ is not an  admissible direction to construct a control
perturbation (of the optimal control $\bu$). Nevertheless, if the
perturbation direction $v$  is chosen  so that for any
$\varepsilon>0$ one can find a $v^{\varepsilon}$ converging to $v$
(in a suitable sense) when $\varepsilon \to 0^+$ and  satisfying
$\bar{u}+\varepsilon v^{\varepsilon}\in \mathcal{U}_{ad}$, then
the variational approach can be adopted to deal with some optimal
control problems having nonconvex control regions (we call it the
{\em classical variational analysis approach}). Indeed, this
method has been used extensively in optimization and optimal
control theory in the deterministic setting. Using this method, in
\cite{Hoehener12,Frankowska15}, some second order integral type
necessary conditions for deterministic optimal controls were
established. It was shown in
\cite{Frankowska13,FrankowskaHoehener15} that these necessary
conditions imply pointwise ones.

In this paper, we shall use the classical variational analysis approach to establish the first and second order necessary optimality conditions for stochastic optimal controls in the general setting, that is, when the control region is allowed to be nonconvex and the control variable enters also into the diffusion term of the control system. Let us recall that, when the diffusion term does NOT depend on  the control variable, cf. \cite{Agayeva11, {Mahmudov97}, Tang10}, the situation  is more or less similar to the deterministic setting like the one  in \cite{Frankowska15,Lou10}.  Compared to the existing results for the case of general control constraints obtained by the spike variations  (\cite{Peng90, zhangH14b}), the main advantage of the classical variational analysis approach is due to weaker  smoothness requirements imposed on the coefficients of the control system and the cost functional (with respect to the state variable $x$)  and to fewer adjoint equations needed to state these conditions. Previously  the first and second order integral type necessary conditions for stochastic optimal controls with convex control constraints were derived in \cite{Bonnans12} using the convex (first order) variations of optimal control.  In the difference with \cite{Bonnans12},  our variational approach  is also valid when the control region is nonconvex  and, since the second order variations of the control region are used in this paper, the corresponding second order necessary condition is more effective than the one of \cite{Bonnans12} even in the case of convex control constraints (see Example \ref{comparion example with Bonnans} below).

In a sense, our work can be viewed as a refinement of known
 optimality conditions for stochastic control problems. To see
it, let us return,  for a moment, to the deterministic optimal
control problem, i.e., when  the functions $\sigma(\cdot)\equiv
0$, $b(\cdot)$, $f(\cdot)$, $g(\cdot)$, $x(\cdot)$ and $u(\cdot)$
in (\ref{controlsys})--(\ref{costfunction}) are independent from
the sample point $\omega$, and also, for the sake of simplicity,
let $K=\{x_0\}$ for some fixed $x_0\in \mrn$.  Consider an
optimal pair $(\bar{x},\bar{u})$ and  the solution  $\psi(\cdot)$
to the following ordinary differential equation,
\begin{equation}\label{firstadjoint for ode}
\left\{
\begin{array}{l}
\dot{\psi}(t)=-b_{x}(t,\bar{x}(t),\bar{u}(t))^{\top}\psi(t)+f_{x}(t,\bar{x}(t),\bar{u}(t)),\quad t\in[0,T],\\
\psi(T)=- \nabla g (\bar{x}(T)).
\end{array}\right.
\end{equation}
Define the (deterministic) Hamiltonian
$$H(t,x,u,\psi):=\inner{\psi}{b(t,x,u)}-f(t,x,u), \qquad\forall\;(t,x,u,\psi)\in [0,T]\times\mrn\times \mrm\times\mrn.
 $$
Then the following Pontryagin maximum principle (\cite{Pontryagin62}) holds
\begin{equation}\label{pontryagin max for ode}
H(t,\bar{x}(t),\bar{u}(t),\psi(t))=\max_{v\in U}H(t,\bar{x}(t),v,\psi(t)), \;\ a.e.\;\ t\in[0,T].
\end{equation}
Clearly, when $U$ is a finite set,  condition (\ref{pontryagin max for ode}) provides an effective way to compute ``$\bar u(\cdot)$"; while when $U$ is convex,
 condition (\ref{pontryagin max for ode}) yields
\begin{equation}\label{firsple}
\big\langle{H}_{u}(t,\bar{x}(t),\bar{u}(t),\psi(t)),v-\bar{u}(t)\big\rangle
\le 0,\qquad
\forall\ v\in U, \ a.e.\  t\in  [0,T].
\end{equation}
What  about  other types of  $U$? Are there other necessary conditions for optimal pairs? The classical monograph \cite{Pontryagin62} was followed by
  numerous  works addressing the above issues
and  refinements of known results on  optimal control problems in
the deterministic finite dimensional setting. In this respect, we
refer to \cite{BellJa75, CA78, Frankowska13, Gabasov73, Goh66,
Hoehener12, Knobloch81, Krener77, Osmolovskii} for high order
necessary conditions when the first-order necessary conditions
turn out to be trivial  and  to   \cite{Osmolovskii} for a
discussion  on ``bang-bang" controls which are very useful in
applications. A very natural question concerns the stochastic
counterpart of the above results. Surprisingly, very little is
known about high order conditions in the stochastic framework!
Indeed, as an interesting comparison, we mention that, there
exists at least five  research monographs (\cite{BellJa75, CA78,
Gabasov73, Knobloch81, Osmolovskii}) devoted to  deterministic
high order necessary conditions but one can find only a very few
published articles (\cite{Agayeva11, Bonnans12, Mahmudov97,
Tang10, zhangH14a}) for their stochastic analogues.

The outline of the paper is as follows. In Section 2, we collect some notations and introduce some  spaces and preliminary results that will be used later. In Section 3, we derive the first order necessary conditions for stochastic optimal controls. Section 4 is devoted to establishing  second order necessary conditions. Finally, in the Appendix, we give the proofs of two technical results from Sections 3 and 4.

Some of preliminary results of this paper are announced (without proofs) in  \cite{Frankowska16}.

\section{Preliminaries}

This section is of preliminary nature, in which we shall introduce some useful notations and spaces, and recall some concepts and results from the set-valued analysis and the Malliavin calculus.

\subsection{Notations and spaces}
In this subsection, we introduce some notations and spaces which will be used in the sequel.

Denote by $C_{b}^{\infty}(\mrn; \mrm)$ the set of $C^{\infty}$-smooth functions from $\mrn$ to $\mrm$ with bounded partial derivatives. Let $\mr^{n\times m}$ be the space of all $n\times m$-real matrices. For any $A\in \mr^{n\times m}$, denote by $A^{\top}$ its transpose  and by
$|A|=\sqrt{tr\{AA^{\top}\}}$ the norm of $A$. Also, write $\mathbf{S}^n:=\big\{A\in \mr^{n\times n}\big|\ A^{\top}=A\big\}$.

Let $\varphi: [0,T]\times\mrn\times \mrm\times \Omega\to \mr^{d}$ ($d\in\mn$) be a given function. For a.e. $(t,\omega)\in [0,T]\times\Omega$, we denote by $\varphi_{x} (t,x,u,\omega)$ and $\varphi_{u} (t,x,u,\omega)$ respectively the first order partial derivatives of $\varphi$ with respect to $x$ and $u$ at $(t,x,u,\omega)$, by $\varphi_{(x,u)^2}(t,x,u,\omega)$ the Hessian of $\varphi$ with respect to $(x,u)$ at $(t,x,u,\omega)$, and by $\varphi_{xx} (t,x,u,\omega)$, $\varphi_{xu} (t,x,u,\omega)$ and $\varphi_{uu} (t,x,u,\omega)$ respectively the second order partial derivatives of $\varphi$ with respect to $x$ and $u$ at $(t,x,u,\omega)$.

For any $\alpha,\beta\in [1,+\infty)$ and $t\in[0,T]$, we denote by $L_{\mf_{t}}^{\beta}(\Omega; \mrn)$ the space of $\mrn$-valued, $\mf_{t}$ measurable random variables $\xi$ such that $\me~|\xi|^{\beta}<+\infty$;
by $L^{\beta}([0,T]\times\Omega; \mrn)$ the space of $\mrn$-valued, $\mb([0,T])\otimes \mf$-measurable processes $\varphi$ such that $\|\varphi\|_{\beta}:=\big[\me\int_{0}^{T}|\varphi(t,\omega)|^{\beta}dt
\big]^{\frac{1}{\beta}} <+\infty$;
by $L_{\mmf}^{\beta}(\Omega; L^{\alpha}(0,T; \mrn))$ the space of $\mrn$-valued, $\mb([0,T])\otimes \mf$-measurable, $\mmf$-adapted processes $\varphi$ such that $\|\varphi\|_{\alpha,\beta}:=\big[\me~\big(\int_{0}^{T}|\varphi(t,\omega)|^{\alpha}dt\big)
^{\frac{\beta}{\alpha}}\big]^{\frac{1}{\beta}} <+\infty$;
by $L_{\mmf}^{\beta}(\Omega; C([0,T]; \mrn))$ the space of $\mrn$-valued, $\mb([0,T])\otimes \mf$-measurable and $\mmf$-adapted  continuous processes $\varphi$  such that $\|\varphi\|_{\infty,\beta}:=
\big[\me~\big(\sup_{t\in[0,T]}|\varphi(t,\omega)|^{\beta}\big)\big]^{\frac{1}{\beta}} <+\infty$;
by $L^{\infty}([0,T]\times\Omega; \mrn)$ the space of $\mrn$-valued, $\mb([0,T])\otimes \mf$-measurable processes $\varphi$ such that $\|\varphi\|_{\infty}:=\mbox{ess sup}_{(t,\omega)\in [0,T]\times\Omega}|\varphi(t,\omega)| <+\infty $
and by
$L^{\beta}(0,T;  L_{\mmf}^{\beta}([0,T]\times\Omega; \mrn))$ the $\mrn$-valued, $\mb([0,T])\otimes \mb([0,T])\otimes\mf$ measurable functions $\varphi$ such that $\|\varphi\|_{\beta}
:=\big[\me\int_{0}^{T}\int_{0}^{T}|\varphi(s,t,\omega)|^{\beta}dsdt
\big]^{\frac{1}{\beta}}$ $<+\infty$ and for any $t\in[0,T]$, the process $\varphi(\cdot,t,\cdot)$ is $\mmf$-adapted.

Let us recall that on a given filtered probability space,  any $\mmf$-progressively measurable process is $\mb([0; T])\otimes\mf$-measurable and $\mmf$-adapted, and every $\mb([0; T])\otimes\mf$-measurable and $\mmf$-adapted process has an $\mmf$-progressively measurable modification (see \cite[Proposition 2.8]{Yong99}).

\subsection{Some concepts and results from the set-valued analysis}
In this subsection, we recall some concepts and results from the set-valued analysis. We refer the reader to \cite{Aubin90} for more details.

Let $X$ be a Banach space with norm $\|\cdot \|_{X}$, and denote by $X^*$ the dual space of $X$.  For any subset $K\subset X$, denote by $\partial K$, $int K$ and $cl K$ its boundary, interior and closure, respectively. $K$ is called a cone if $\alpha x\in K$ for any $\alpha\ge0$ and $x\in K$. Define the distance between a point $x\in X$ and $K$  by $\displaystyle dist\,(x,K):= \inf_{y\in K} \|y-x\|_{X}$. Define the metric projection of $x$ onto $K$ by $\Pi_{K}(x):= \{y\in K\ |\ \|y-x\|_{X}=dist\,(x,K)\}$.
\begin{definition}
For $x\in K$, the adjacent cone $\T_{K}(x)$ to $K$ at $x$ is defined by
$$\T_{K}(x):=\Big\{v\in X\ \Big|\ \lim_{\eps\to 0^+} \frac{dist\,(x+\eps v,K)}{\eps}=0         \Big\}.$$
\end{definition}

If in the above $\lim_{\varepsilon\to 0^+}$ is replaced by  $\liminf_{\varepsilon\to 0^+}$, then we obtain a larger cone, the so called {\em contingent cone} $T_{K}^{B}(x)$  to $K$ at $x.$
When $K$ is convex, the adjacent cone and the contingent cone coincide with each other, and
$$\T_{K}(x)=cl\Big\{\alpha(y-x)\ \Big|\ \alpha\ge0,\ y\in K \Big\}.$$

It is not difficult to realize that
$v\in \T_{K}(x)$ if and only if for any $\eps>0$ there exists a $v_{\eps}\in X$ such that $v_{\eps}\to v$ (in $X$) as $\eps\to 0^+$, and $x+\eps v_{\eps}\in K$.

\begin{definition}
For any $x\in K$ and $v\in \T_{K}(x)$, the second order adjacent subset to $K$ at $(x,v)$ is defined by
$$\TT_{K}(x,v):=\Big\{h\in X\ \Big|\ \lim_{\eps\to 0^+} \frac{dist\,(x+\eps v+\eps^{2}h,K)}{\eps^2}=0 \Big\}.$$
\end{definition}

Similarly to the above, $h\in \TT_{K}(x,v)$ if and only if  for any $\eps>0$ there exists an $h_{\eps}\in X$ such that $h_{\eps}\to h$ (in $X$)  as $\eps\to 0^+$ and $x+\eps v+\eps^{2}h_{\eps}\in K$.

\begin{remark}\label{rem}
Clearly, $0\in \T_{K}(x)$ for any $x\in K$ and $\alpha v\in \T_{K}(x)$ for any $\alpha> 0$ and  $v\in \T_{K}(x)$. Therefore, $\T_{K}(x)$ is a nonempty closed  cone.  $\T_{K}(x)=X$ for any $x\in int K$. Also, $\TT_{K}(x,0)= \T_{K}(x)$. When $K$ is convex, $y-x\in \T_{K}(x)$  and $0 \in  \TT_{K}(x,y-x)$  for any $x\in K$ and $y\in K$.
When $v \neq 0$, the set $\TT_{K}(x,v)$, in general,  may not be a cone and it may be an empty set (some examples can be found in \cite[section 4.7]{Aubin90}).
\end{remark}

The dual cone of the tangent cone $\T_{K}(x)$, denoted by $\N_{K}(x)$, is called the normal cone of $K$ at $x$, i.e.,
$$\N_{K}(x):=\Big\{\xi\in X^*\ \Big|\ \inner{\xi}{v}\le 0,\ \forall\; v\in \T_{K}(x) \Big\}.$$
When $K$ is convex, $\N_{K}(x)$ reduces to the normal cone $N_{K}(x)$  of the convex analysis, where
$$N_{K}(x):=\Big\{\xi\in X^*\ \Big|\ \inner{\xi}{y-x}\le 0,\ \forall\; y\in K \Big\}.$$
When $X$ is a Hilbert space, for any $\xi\in \N_{K}(x)$ the second order normal cone  to $K$ at $(x,\xi)$ is defined by
$$\NN_{K}(x,\xi):=\Big\{\zeta\in \mathbf{S}(X)\ \Big|\ \inner{\xi}{h}+\frac{1}{2}\inner{\zeta v}{v}\le 0,\ \forall\; v\in \T_{K}(x)\cap \{\xi\}^{\bot}  ,\ \forall\; h\in\TT_{K}(x,v)\Big\},$$
where $\mathbf{S}(X)$ is the space of symmetric, continuous linear operators from $X$ to $X$ and  $\{\xi\}^{\bot}:=\{v\in X\ |\ \inner{\xi}{v}=0\}$.

In the following, we recall a classical example in which the closed set $K$ is defined by finitely many equalities and inequalities.

\begin{example}\label{example for 1st and 2nd tangent set}
When  $K \subset \mr^n$ is given by inequality and equality constraints and a constraint qualification holds true, there are exact expressions for the first and second order tangent sets.  More precisely, consider  twice continuously  differentiable
functions  $\varphi_1,...,\varphi_p \colon \mr^n \to \mr$ and $\psi_1,\dots, \psi_r\colon \mr^n \to \mr$ (for some $p,r\in \mn$),  set $\varphi=(\varphi_1,...,\varphi_p)$ and define
\begin{equation*}
K = \big \{x \in \mr^n\,\big|\, \varphi(x)=0, \;\;  \psi_j(x) \le 0, \;\forall\, j=1,...,r\big\}.
\end{equation*}
  If there are no equality, resp. inequality,  constraints in the definition of $K$, then  the terms involving $\varphi, \,\varphi_i$, resp. $\psi_j$,  are absent in the  discussion below and $p$, resp. $r$, is equal to zero.

Let $x \in K$ and denote by $I(x)$ the set of all active indices,
i.e. $j \in I(x)$ if and only if $\psi_j (x)=0$. We assume that the Mangasarian-Fromowitz constraint qualification holds true: the Jacobian  $\varphi'(x)$  is surjective and there exists a
$v_0 \in \mr^n$  such that $$\varphi'(x)v_0=0,\;\; \langle \nabla \psi_j(x),v_0\rangle < 0, \; \;\; \forall \; j \in I(x). $$
In the absence of equality constraints this is equivalent to the assumption that
$\{\nabla \psi_j (x)\;|\; j \in I(x)\}$ are positively independent or, equivalently, $0 \notin {\rm co}\, \{\nabla \psi_j (x)\;|\; j \in I(x)\}$.
Then it is well known, see for instance \cite[pp. 150--151]{Aubin90} that
$$\T_{K}(x)= \big\{v \in \mr^n \;\big|\; \varphi'(x)v=0, \;  \langle \nabla
\psi_j(x),v \rangle \leq 0, \; \; \forall \, j \in I(x)\big\},$$
$$  \N_{K}(x)=  \sum_{i=1}^p \mr \nabla \varphi_i(x) + \sum_{j \in I(x)}\mr_+ \nabla \psi_j (x)  . $$
 If there are no equality constraints and $I(x)=\emptyset$, then  $\T_{K}(x)=\mr^n$ and therefore $ \N_{K}(x)=\{0\}$.

Fix any $v \in \T_{K}(x)$  and consider the set $I_{v}(x)=\{j \in I(x)\,|\, \langle \nabla \psi_j(x), v \rangle=0\}$.
  Then the same proof as in \cite[p.177]{Aubin90} (given there only for the second  order  contingent set)
implies that
\begin{equation*}\label{jan13a}
\begin{array}{ll}\displaystyle   \TT_{K}(x,v) = \left\{h\in \mr^n \;\left|\;
 \langle \nabla \varphi_i(x),h\rangle+\frac{1}{2}\langle \varphi_i '' (x)v,v\rangle = 0, \;\forall \, i=1,\cdots,p
 \right.\right.\\
\displaystyle\qquad\qquad\qquad\qquad\quad\;\;\;\;\hbox{and }\left.
\langle \nabla \psi_j(x),
h\rangle+\frac{1}{2}\langle \psi_j '' (x)v,v\rangle\le0, \;
\forall \, j \in I_{v}(x)
\right \}.
\end{array}\end{equation*}  Thus,  under our assumptions,  $\TT_{K}(x,v)\neq
\emptyset$ for all $v \in \T_{K}(x)$.

Observe that $\NN_{K}(x,0)$ is equal to the set of all symmetric $(n\times n)$-matrices that are seminegative on $\T_{K}(x) $.

 If $I(x)\neq \emptyset$,  denote by $i_1,...,i_k$ all the active indices (for some $k \leq r$). In  the expressions below the terms involving $\varphi_i$, resp. $\psi_{i_j}$,  are absent when there are no equality constraints, resp. when $I(x)=\emptyset$.

 Fix any $0\neq q \in \N_{K}(x)$.  Then for some
reals  $\{\mu_i \}_{i=1}^p$,   $\lambda_j \geq 0, j=1,...,k$
$$q= \sum_{i=1}^{p}\mu_i \nabla \varphi_i(x) +
\sum_{j=1}^{k}\lambda_j \nabla \psi_{i_j}(x). $$

 To express  $\NN_{K}(x,q)$  we could apply the same method as in \cite{Frankowska15}.  In order to
simplify the discussion,  we assume that  $\{ \nabla \varphi_1(x),\cdots,\nabla \varphi_p(x)\}\bigcup \{ \nabla \psi_j(x)\,|\, j \in I_{v}(x)\}$ are linearly
 independent for every $ v \in  \T_{K}(x) \cap \{q\}^\perp$  different from zero.

Let $v \in  \T_{K}(x) \cap \{q\}^\perp$. If $I(x)\neq \emptyset$, then
$0=\langle q,v\rangle = \big\langle \sum_{j=1}^{k}\lambda_j \nabla \psi_{i_j}(x),v \big\rangle, $
which yields  $\lambda_j \langle
\nabla \psi_{i_j}(x),v\rangle =0$ for every $j=1,...,k$.
 Hence,  $\lambda_j=0$ whenever $
i_j \notin I_{v}(x)$. Furthermore,  if the equality constraints are absent, then $I_{v}(x)\neq\emptyset $ for every $v \in  \T_{K}(x) \cap \{q\}^\perp$.
Consequently,
$$\langle q,h\rangle + \frac{1}{2}\sum_{i=1}^{p}\mu_i \langle \varphi_i''(x)v,v \rangle + \frac{1}{2} \sum_{j=1}^{k}\lambda_j\langle  \psi_{i_j}''(x)v,v \rangle \leq 0, \qquad \forall\;h \in
T^{\flat (2)}_{K}(x,v) .$$
Therefore, by arbitrariness of $v \in  \T_{K}(x) \cap \{q\}^\perp$,
$$ \overline Q:=\sum_{i=1}^{p}\mu_i \varphi_i''(x) + \sum_{j=1}^{k}\lambda_j \psi_{i_j}''(x) \in \NN_{K}(x,q).$$
Observe that if a symmetric  $(n \times n) $-matrix  $Q $ is so
that $\langle Qv,v \rangle \leq \langle\overline Qv,v\rangle$  for
every $v \in  \T_{K}(x) \cap \{q\}^\perp$, denoted by $Q \leq
\overline Q$, then $Q \in \NN_{K}(x,q).$

We show next that $ \overline Q$ is the largest second order normal in the above sense. Fix any $Q \in \NN_{K}(x,q)$.
Let $v \in  \T_{K}(x) \cap \{q\}^\perp$. If  $v=0$, then $\langle Qv,v \rangle \leq \langle\overline Qv,v\rangle$. Assume next that  $v\neq 0$.   If  $I_{v}(x)\neq \emptyset$,  consider the set  $\{j_1,..., j_m\}$  of all the indices  that belong to $I_{v}(x)$.  Define  the $(n \times (p+m)) $-matrix $A  $ such that its
s-th column is $\nabla \varphi_s(x)$  for $1 \leq s \leq p$ and
$\nabla \psi_{j_{s-p}} (x)$
 for  $p+1 \leq s \leq p +m$ (we set $m=0$ if $I_{v}(x)=\emptyset$).
By the linear independence assumption,  we show that for any $0\neq v \in  \T_{K}(x) \cap \{q\}^\perp$ there exists $z_v \in \mr ^n$ satisfying
$$z_v^{\top} A = -\frac{1}{2}\left (\langle \varphi_1''(x)v,v\rangle ,...,\langle \varphi_p''(x)v,v\rangle ,\langle \psi_{j_1}''(x)v,v\rangle ,...,
\langle \psi_{j_m}''(x)v,v  \rangle  \right).$$ Hence $z_v \in \TT_{K}(x,v)$  and
$\langle q,z_v \rangle =  - \frac{1}{2}\sum_{i=1}^{p}\mu_i\langle  \varphi_i''(x)v,v \rangle - \frac{1}{2} \sum_{j=1}^{k}\lambda_j \langle \psi_{i_j}''(x)v,v \rangle . $ Thus
 $$\langle q,z_v \rangle + \frac{1}{2} \langle Qv,v \rangle \leq 0= \langle q,z_v \rangle + \frac{1}{2}\sum_{i=1}^{p}\mu_i \langle \varphi_{i}''(x)v,v\rangle  +\frac{1}{2} \sum_{j=1}^{r}\lambda_j \langle \psi_j''(x)v,v \rangle . $$
Consequently
$
Q \leq  \overline Q$  in the above sense.

However, in general, closed sets  do not have the above representation.  We refer to \cite{Frankowska15} for a very simple example of a set $K$ given by
union of two intervals in $\mr^2$, where the first and second order tangents can be easily computed, but, at the same time, $K$ does not satisfy
the constraint qualification assumption.

We would like to underline here that to prove the celebrated
Pontryagin maximum principle  in optimal control just a particular
subset of tangents to the set of controlled trajectories  was
used.  The computation of the whole tangent cone is, in
general, not possible.
Similarly, we do not need to know the whole set of the second
order tangents to eliminate some candidates for optimality.
\end{example}

Let $(\Xi, \mg)$ be a measurable space, and $F:\Xi\leadsto 2^X$ be a set-valued map.
For any $\xi\in \Xi$, $F(\xi)$ is called the value of $F$ at $\xi$. The domain of $F$ is the subset of all $\xi\in \Xi$ such that $F(\xi)$ is nonempty, i.e.,
$Dom\,(F):=\{\xi\in \Xi\ |\ F(\xi)\neq \emptyset\}$.
$F$ is called measurable if $F^{-1}(A):=\{\xi\in \Xi\ |\ F(\xi)\cap A\neq \emptyset\}\in \mg$ for any $A\in \mb(X)$.
Clearly, the domain of a measurable set-valued map is measurable.

The following result is a special case of \cite[Theorem 8.5.1]{Aubin90}.
\begin{lemma}\label{mt subset tbmk}
Suppose $(\Xi,\mg,\mu)$ is a complete $\sigma$-finite measure
space,  $X$ is a separable Banach space, $p\ge 1$ and $K$ is a
closed nonempty subset in $X$. Define
 $$
 \mk:= \big\{\varphi(\cdot)\in L^p(\Xi,\mg,\mu; X)\ \big| \ \varphi(\xi)\in K, \ \mu\hbox{--a.e.}\ \xi\in \Xi\big\}.
 $$
Then for any $\varphi(\cdot)\in\mk$,
the set-valued map $\T_{K}(\varphi(\cdot))$: $\xi\leadsto \T_{K}(\varphi(\xi))$ is
$\mg$-measurable, and
$$\mt := \big\{\psi(\cdot)\in L^p(\Xi,\mg,\mu; X)\ \big|\
\psi(\xi)\in\T_{K}(\varphi(\xi)), \ \mu\hbox{--a.e.}\ \xi\in \Xi\big\} \subset\T_{\mk}(\varphi(\cdot)).
$$
\end{lemma}

The following result is a special case of \cite[Corollary 8.2.13]{Aubin90}.
\begin{lemma}\label{metric projection}
Suppose $(\Xi,\mg,\mu)$ is a complete $\sigma$-finite measure
space,  $X$ is a separable Banach space, $K$ is a
closed nonempty subset in $X$ and $\varphi(\cdot)$ is a $\mg$-measurable single-valued mapping. Then the projection mapping
$\xi\rightsquigarrow \Pi_{K}(\varphi(\xi))$
is $\mg$-measurable, and there exists a $\mg$-measurable, $X$-valued selection $\psi(\cdot)$ such that $\|\psi(\xi)-\varphi(\xi)\|_{X}=dist\,(\varphi(\xi), K)$, $\mu$-a.e.
\end{lemma}

As in \cite{Kisielewicz13}, we call a measurable set-valued map $\zeta:(\Omega,\mf)\leadsto 2^{\mrm}$ a set-valued random variable, and, we call a map $\Gamma:[0,T]\times\Omega\leadsto 2^{\mrm}$ a  measurable set-valued stochastic process if $\Gamma$ is $\mathcal{B}([0,T])\otimes \mathcal{F}$-measurable. We say that $\Gamma$ is $\mmf$-adapted if $\Gamma(t)$ is $\mf_{t}$-measurable for any $t\in [0,T]$. Define
\begin{equation}\label{adapted sigma field}
\mg:=\big\{A\in \mb([0,T])\otimes\mf\ \big|\ A_{t}\in \mf_{t},\ \forall\; t\in [0,T]\big\},
\end{equation}
where $A_{t}:=\{\omega\in \Omega\ |\ (t,\omega)\in A\}$ is the section of $A$. Obviously, $\mg$ is a sub-$\sigma$-algebra of $\mb([0,T])\otimes\mf$. As pointed in \cite[p. 96]{Kisielewicz13}, the following result holds.

\begin{lemma}\label{lemma adapted sigma field}
A set-valued stochastic process $\Gamma:[0,T]\times\Omega\leadsto 2^{\mrm}$ is $\mb([0,T])\otimes\mf$-measurable and $\mmf$-adapted if and only if $\Gamma$ is $\mg$-measurable.
\end{lemma}

Obviously, $\mmu$ is a nonempty closed subset of the Banach space $L^{2}_{\mmf}(\Omega;L^{2}(0,T); \mrm)$. Using Lemmas \ref{mt subset tbmk} and
\ref{lemma adapted sigma field}, the following result was derived in   \cite{WangZhang15}.  It is useful later in getting the desired pointwise first order necessary condition.

\begin{lemma}(\cite[Lemma 4.6]{WangZhang15})\label{Lemma integral to pointwise} Let $U$ be closed,  $\tilde{u}(\cdot)\in \mmu$, and
$F:[0,T]\times\Omega\to\mrm$ be a
$\mb([0,T])\times\mf$-measurable and $\mmf$-adapted process such
that
$$
\me\int_{0}^{T}\inner{ F(t )}{v(t )}dt\le 0,\quad \forall\;
v(\cdot)\in \T_{\mmu}(\tilde{u}(\cdot)).
$$
Then,
$$
 \inner{F(t,\omega)}{v} \le0,\quad \forall\; v\in T^{b}_{U}(\tilde
{u}(t,\omega)),\ a.e.\ (t,\omega)\in [0,T]\times\Omega.
$$
\end{lemma}

\subsection{Some concepts and results from the Malliavin calculus}
In this subsection, we recall some concepts and results from the
Malliavin calculus (see \cite{Nualart06} for a detailed discussion
on this topic).

For any $\eta\in L^2(0,T)$, write $\mathcal{W}(\eta)=\int_{0}^{T}\eta(t)dW(t)$. Define
\begin{equation}\label{mS}
\begin{array}{ll}
\mathcal{S}:=\Big\{\zeta=\varphi( \mathcal{W}(\eta_{1}),\ \mathcal{W}(\eta_{2}),\ \cdots,\  \mathcal{W}(\eta_{d}))\ \Big|\
\varphi\in C_{b}^{\infty}(\mrd; \mrn),\\
 \qquad\qquad\qquad\qquad\qquad\qquad\qquad\qquad\qquad \quad\;\;
 \eta_1,\eta_2,\cdots,\eta_d\in L^2(0,T), d\in \mn\Big\}.
\end{array}
\end{equation}
Clearly, $\mathcal{S}$ is a linear subspace of $L^{2}_{\mf_{T}}(\Omega; \mrn)$. For any $\zeta\in \mathcal{S}$ (as in (\ref{mS})), its Malliavin derivative is defined as follows:
$$\dd_{s}\zeta:=\sum_{i=1}^{d}\eta_{i}(s)\frac{\partial\varphi}{\partial x_{i}}(\mathcal{W}(\eta_{1}),\ \mathcal{W}(\eta_{2}),\ \cdots,\  \mathcal{W}(\eta_{d})), \ \ \ \mbox{a.e.}\ s\in[0,T],\ a.s.$$
Write
$$|||\zeta|||_{2}:=\Big[\me ~|\zeta|^2+\me\int_{0}^{T}|\dd_{s}\zeta|^2ds\Big]^{\frac{1}{2}}.$$
Obviously, $|||\cdot|||_{2}$ is a norm on $\mathcal{S}$. It is shown in \cite {Nualart06} that the operator $\dd$ has a closed extension to the space $\md^{1,2}(\mrn)$, the completion of $\mathcal{S}$ with respect to the norm $|||\cdot|||_{2}$. When $\zeta\in \md^{1,2}(\mrn)$, the following Clark--Ocone representation formula holds:
\begin{equation}\label{clark-ocone formula}
\zeta=\me ~\zeta+\int_{0}^{T}\me~(\dd_{s} \zeta\ |\ \mf_{s})dW(s).
\end{equation}
Furthermore, if $\zeta$ is $\mf_{t}$-measurable, then $\dd_{s}\zeta=0$ for any $s\in(t,T]$.

Let  $\ml^{1,2}(\mrn)$ denote the space of processes $\varphi\in L^{2}([0,T]\times\Omega; \mrn)$
such that
\begin{enumerate}[{\rm (i)}]
  \item For $a.e.$ $t\in[0,T]$, $\varphi(t,\cdot)\in \md^{1,2}(\mrn)$;
  \item the function $\dd_{\cdot}\varphi(\cdot, \cdot):\ [0,T]\times[0,T]\times\Omega\to\mrn $ admits a $\mb([0,T]\times[0,T])\otimes\mf$-measurable version;
  \item $\displaystyle |||\varphi|||_{1,2}:=\big[\me\int_{0}^{T}|\varphi(t,\omega)|^2dt
      +\me\int_{0}^{T}\int_{0}^{T}|\dd_{s}\varphi(t,\omega)|^2dsdt\big]^{\frac{1}{2}}<+\infty.$
\end{enumerate}
Denote by $\ml_{\mmf}^{1,2}(\mrn)$ the set of all $\mmf$-adapted processes in $\ml^{1,2}(\mrn)$.

In addition, write
\begin{eqnarray*}
& &\ml_{2^+}^{1,2}(\mrn):=\Bigg\{\varphi\in\ml^{1,2}(\mrn)\Big|\ \exists\ \dd^{+}\varphi\in L^2([0,T]\times\Omega;\mrn)\ \mbox{s. t. for any small } \varepsilon>0,\\
& &\qquad\quad\quad f_{\eps}(s):=\sup_{s<t<(s+\varepsilon)\wedge T}
\me~\big|\dd_{s}\varphi(t,\omega)-\dd^{+}\varphi(s,\omega)\big|^2<\infty,\ \mbox{a.e.}\ s\in [0,T],\\
& &\quad\qquad\quad f_{\eps}(\cdot)\ \mbox{is measurable on }\ [0,T],\  \mbox{and}\ \lim_{\varepsilon\to 0^+}\int_{0}^{T}f_{\eps}(s)ds=0\Bigg\};
\end{eqnarray*}
\begin{eqnarray*}
& &\ml_{2^-}^{1,2}(\mrn):=\Bigg\{\varphi\in\ml^{1,2}(\mrn)\Big|\ \exists\ \dd^{-}\varphi\in L^2([0,T]\times\Omega;\mrn)\ \mbox{s. t.  for any small } \varepsilon>0,\\
& &\qquad\quad\quad g_{\eps}(s):=\sup_{(s-\varepsilon)\vee 0<t<s}
\me~\big|\dd_{s}\varphi(t,\omega)-\dd^{-}\varphi(s,\omega)\big|^2<\infty,\ \mbox{a.e.}\ s\in [0,T],\\
& &\qquad\quad\quad g_{\eps}(\cdot)\ \mbox{is measurable on}\ [0,T],\ \mbox{and}\ \lim_{\varepsilon\to 0^+}\int_{0}^{T}g_{\eps}(s)ds=0\Bigg\}.
\end{eqnarray*}
Set
$\mathbb{L}_{2}^{1,2}(\mathbb{R}^n)=\mathbb{L}_{2^+}^{1,2}(\mathbb{R}^n)\cap\mathbb{L}_{2^-}^{1,2}(\mathbb{R}^n)$
and   define
 $$
 \nabla\varphi=\mathcal{D}^{+}\varphi+\mathcal{D}^{-}\varphi,\quad \forall\;\varphi\in \mathbb{L}_{2}^{1,2}(\mathbb{R}^n).
 $$

When $\varphi$ is $\mathbb{F}$-adapted, $\mathcal{D}_{s}\varphi(t,\omega)=0$ a.s. for any $t<s$. In this case, $\mathcal{D}^{-}\varphi=0$  and $\nabla\varphi=\mathcal{D}^{+}\varphi$ a.e. $t\in[0,T]$, a.s. Denote by $\mathbb{L}_{2,\mathbb{F}}^{1,2}(\mathbb{R}^n)$ the set of all $\mathbb{F}$-adapted processes in $\mathbb{L}_{2}^{1,2}(\mathbb{R}^n)$.

Roughly speaking, an element $\varphi\in\ml_{2}^{1,2}(\mrn)$ is a stochastic process whose  Malliavin derivative has suitable continuity on some neighborhood of $\{(t,t)\ |\  t\in [0,T]\}$. Examples of such processes can be found in \cite{Nualart06}. Especially, if $(s,t)\mapsto \dd_{s}\varphi(t,\omega)$ is continuous from $V_{\delta}:=\{(s,t)\big|\ |s-t|<\delta,\ s,t\in [0,T]\}$ (for some $\delta>0$) to $L_{\mf_{T}}^{2}(\Omega;\mrn)$, then $\varphi\in\ml_{2}^{1,2}(\mrn)$ and, $\dd^{+}\varphi(t,\omega)=\dd^{-}\varphi(t,\omega)=\dd_{t}\varphi(t,\omega)$ a.e. $t\in[0,T]$, a.s.

\section{First order necessary conditions}
In this section, we study the first order necessary optimality conditions for the optimal control problem (\ref{minimum J}).
Firstly, we introduce the notion of local minimizer for the problem (\ref{minimum J}).

\begin{definition}
An admissible triple  $(\bx,\bu,\bx_{0})\in L^{2}_{\mmf}(\Omega;C([0,T];\mrn))\times\mmu\times K$ is called a local minimizer for the problem (\ref{minimum J}) if there exists a $\delta>0$ such that $J(u,x_{0})\ge J(\bu,\bx_{0})$ for any admissible  triple $(x,u,x_{0})\in L^{2}_{\mmf}(\Omega;C([0,T];\mrn))\times\mmu\times K$ satisfying $\|u-\bu\|_{2}<\delta$ and $|\bx_{0}-x_{0}|< \delta$.
\end{definition}

In this section, we need the following assumptions:

\begin{enumerate}
  \item [{\bf (C1)}] {\em The control region $U $ is nonempty and closed.}
  \item [{\bf (C2)}] {\em The functions $b$, $\sigma$, $f$ and $g$ satisfy the following:}
  \begin{enumerate}[{\rm (i)}]
      \item{\em For any $(x, u)\in \mrn\times \mrm$, the stochastic processes
      $b(\cdot, x, u,\cdot):\ [0,T]\times\Omega\to \mrn$ and $\sigma(\cdot, x, u,\cdot):\ [0,T]\times\Omega\to \mrn$
      are $\mb([0,T])\otimes\mf$-measurable and $\mmf$-adapted. For a.e. $(t,\omega)\in [0, T]\times\Omega$, the functions
      $b(t, \cdot, \cdot,\omega):\ \mrn\times \mrm\to \mrn$ and $\sigma(t, \cdot, \cdot,\omega):\ \mrn\times \mrm\to \mrn$
      are differentiable and
      $$(x,u)\mapsto (b_{x}(t,x,u,\omega),b_{u}(t,x,u,\omega),
      \sigma_{x}(t,x,u,\omega),\sigma_{u}(t,x,u,\omega))$$
      is uniformly continuous in $x\in \mrn$ and $u\in \mrm$.
      There exist a constant $L > 0$ and a nonnegative $\eta\in L^{\beta}_{\mmf}(\Omega;L^{2}(0,T;\mr))$ with $\eta(T,\cdot)\in L_{\mf_{T}}^{\beta}(\Omega; \mr)$ and $\beta\ge 1$ such that for a.e. $(t,\omega)\in [0, T]\times\Omega$ and for any  $x\in\mrn$ and $u\in \mrm$,}
      $$
      \left\{
      \begin{array}{l}
      |b(t,0, u,\omega)|+|\sigma(t,0, u,\omega)|\le L(\eta(t,\omega)+|u|),\\[+0.3em]
      |b_{x}(t,x,u,\omega)|+|b_{u}(t,x,u,\omega)|\le L,\\[+0.3em]
      |\sigma_{x}(t,x,u,\omega)|+|\sigma_{u}(t,x,u,\omega)|\le L;
      \end{array}\right.
      $$

       \item {\em For any $(x, u)\in \mrn\times \mrm$, the stochastic process
             $f(\cdot, x, u,\cdot):\ [0,T]\times\Omega\to \mr$ is $\mb([0,T])\otimes\mf$-measurable and $\mmf$-adapted, and
             the random variable $g(x,\cdot)$ is $\mf_{T}$-measurable. For a.e. $(t,\omega)\in [0, T]\times\Omega$, the functions
             $f(t, \cdot, \cdot,\omega):\ \mrn\times \mrm\to \mr$ and $g(\cdot,\omega):\ \mrn \to \mr$ are  differentiable, and for any  $x,\ \tilde{x}\in\mrn$ and $u,\ \tilde{u}\in \mrm$,}
             $$
             \left\{
             \begin{array}{l}
              |f(t,x,u,\omega)|\le L(\eta(t,\omega)^2+|x|^{2}+|u|^{2}),\\[+0.3em]
              |f_{x}(t,0,u,\omega)|+|f_{u}(t,0,u,\omega)|\le L(\eta(t,\omega)+|u|),\\[+0.3em]
              |f_{x}(t,x,u,\omega)-f_{x}(t,\tilde{x},\tilde{u},\omega)|
              +|f_{u}(t,x,u,\omega)-f_{u}(t,\tilde{x},\tilde{u},\omega)|\\[+0.3em]
              \quad\le L(|x-\tilde{x}|+|u-\tilde{u}|),\\[+0.3em]
              |g(x,\omega)|\le L(\eta(T,\omega)^2+|x|^{2}),\ |g_{x}(0,\omega)| \le L\eta(T,\omega),\\[+0.3em]
              |g_{x}(x,\omega)-g_{x}(\tilde{x},\omega)|\le L|x-\tilde{x}|.
             \end{array}\right.
             $$
  \end{enumerate}
\end{enumerate}

When the condition (C2) is satisfied, the state $x $ (of (\ref{controlsys})) is uniquely defined by any given initial datum $x_{0}\in\mr^n$ and admissible control $u \in \mmu$, and the cost functional (\ref{costfunction}) is well-defined on $\mmu$. In what follows, $C$ represents a generic positive constant (depending only on $T$, $\beta$, $\eta(\cdot)$ and $L$), which may be different from one place to another.

The following known result (\cite{Yong07}) is useful in the sequel.

\begin{lemma}\label{estimatelinearsde}
Assume  (C2). Then,  for any $x_{0}\in \mrn$,  $\beta\ge 1$  and $u\in L_{\mmf}^{\beta}(\Omega; L^{2}(0, T; \mrm))$, the state equation (\ref{controlsys}) admits a unique solution
$x \in L_{\mmf}^{\beta}(\Omega; C([0, T];$ $ \mrn))$, and for any $t\in [0,T]$ the following estimate holds:
\begin{equation}\label{estimateofx}
\me\Big(\sup_{s\in[0,t]}|x(s,\omega)|^{\beta}\Big)
\le C\me~\Big[|x_{0}|^{\beta}
+\Big(\int_{0}^{t}|b(s,0,u(s),\omega)|ds\Big)^{\beta}
+\Big(\int_{0}^{t}|\sigma(s,0,u(s),\omega)|^{2}ds\Big)^{\frac{\beta}{2}}\Big].
\end{equation}
Moreover, if  $\tilde{x}$ is the  solution to (\ref{controlsys}) corresponding to
$(\tilde x_{0}, \tilde{u})\in \mrn\times L_{\mmf}^{\beta}(\Omega; L^{2}(0, T;$ $\mrm))$, then, for any $t\in [0,T]$,
\begin{equation}\label{estimateof delta x}
\me\Big(\sup_{s\in[0,t]}|x(s,\omega)-\tilde{x}(s,\omega)|^{\beta}\Big)
\le C\me~\Big[|x_{0}-\tilde x_{0}|^{\beta}+\Big(\int_{0}^{t}|u(s,\omega)
-\tilde{u}(s,\omega)|^{2}ds
\Big)^{\frac{\beta}{2}}\Big].
\end{equation}
\end{lemma}

Now, let us introduce the classical first order variational control system. Let $\bu, v, v_{\eps}\in  L^{\beta}_{\mmf}(\Omega;L^{2}(0,T;\mrm))$ ($\beta\ge 1$) and $\nu_{0},\nu_{0}^{\eps}\in \mrn$ satisfying $v_{\eps}\to v$ in $L^{\beta}_{\mmf}(\Omega;L^{2}(0,T;\mrm))$ and $\nu_{0}^{\eps}\to \nu_{0}$ in $\mrn$ as $\eps\to 0^+$. For $u^{\eps}:=\bu+\eps v_{\eps}$ and $x^{\eps}_{0}:=x_{0}+\eps \nu_{0}^{\eps}$, let $x^{\eps}$ be the state  of (\ref{controlsys}) corresponding  to the control $u^{\eps}$ and the initial datum $x_{0}^{\eps}$, and put $\delta x^{\eps}=x^{\eps}-\bx$. For $\varphi=b,\sigma, f$, denote
$$
\varphi_{x}(t)=\varphi_{x}(t,\bar{x}(t),\bar{u}(t)),\quad
\varphi_{u}(t)=\varphi_{u}(t,\bar{x}(t),\bar{u}(t)).$$

Consider the following linearized stochastic control system:
\begin{equation}\label{first vari equ}
\left\{
\begin{array}{l}
dy_{1}(t)= \big(b_{x}(t) y_{1}(t)+b_{u}(t)v(t)\big)dt+\big(\sigma_{x}(t) y_{1}(t)+ \sigma_{u}(t)v(t)\big)dW(t),\quad  t\in[0,T], \\
y_{1}(0)=\nu_{0}.
\end{array}\right.
\end{equation}
We first establish the following estimates.

\begin{lemma}\label{estimate one of varie qu}
Let (C2) hold and  $\beta\ge 1$. Then, for any $\bu, v,v_{\eps}, \nu_{0},\nu_{0}^{\eps}$ and $\delta x^\eps$ as above
\begin{equation*}
 \|y_{1}\|_{\infty,\beta}^{\beta}\le C\big(|\nu_{0}|^{\beta}+\|v\|_{2,\beta}^\beta\big),
\quad
\|\delta x^{\eps}\|_{\infty,\beta}^{\beta}= O(\eps^{\beta}).
\end{equation*}
Furthermore,
\begin{equation}\label{r1 to 0}
\|r_{1}^{\eps}\|_{\infty,\beta}^{\beta}\to 0,\quad \hbox{as }\eps\to 0^+,
\end{equation}
where $r_{1}^{\eps}(t,\omega):= \frac{\delta x^{\eps}(t,\omega)}{\eps}- y_{1}(t,\omega)$.
\end{lemma}
\begin{proof}
See Appendix A.
\end{proof}

\vspace{0.5mm}

Next, define the Hamiltonian
\begin{equation}\label{Hamiltonianconvex}
H(t,x,u, p,q,\omega)
:=\inner{p}{b(t,x,u,\omega)}+\inner{q}{\sigma(t,x,u,\omega)}-f(t,x,u,\omega),
\end{equation}
where $(t,x,u,p,q,\omega)\in [0,T]\times\mrn\times \mrm\times\mrn\times\mrn\times\Omega.$
We introduce the first order adjoint equation for (\ref{first vari equ}):
\begin{equation}\label{first ajoint equ}
 \left\{
\begin{array}{l}
dP_{1}(t)=-\big(b_{x}(t)^{\top}P_{1}(t)
          +\sigma_{x}(t)^{\top}Q_{1}(t)
          -f_{x}(t)\big)dt+Q_{1}(t)dW(t), \quad  t\in[0,T], \\
P_{1}(T)=-g_{x}(\bar{x}(T)).
\end{array}\right.
\end{equation}
By \cite{Peng97} and (C2), for any $\beta\ge 1$, if $\bar u\in L_{\mmf}^{\beta}(\Omega; L^{2}(0,T; \mrm))$, the equation (\ref{first ajoint equ}) admits a unique strong solution
$(P_{1},Q_{1})\in L_{\mmf}^{\beta}(\Omega; $ $C([0,T]; \mrn))\times
L_{\mmf}^{\beta}(\Omega; L^{2}(0,T; \mrn))$.

We have the following result.

\begin{theorem}\label{TH first order integraltype condition}
Let (C1)--(C2) hold. If $(\bx,\bu,\bx_{0})$ is a local minimizer for the problem (\ref{minimum J}), then
\begin{equation}\label{first order integraltype condition}
\me\int_{0}^{T}\inner{H_{u}(t)}{v(t)}dt\le 0,\quad \forall\; v\in \T_{\mmu}(\bu),
\end{equation}
and
\begin{equation}\label{trans condition}
P_{1}(0) \in \N_{K}(\bar x_0),
\end{equation}
where $(P_{1},Q_{1})$ is the solution to the first order adjoint equation (\ref{first ajoint equ}) corresponding to $(\bx,\bu,\bx_{0})$ and $H_{u}(t)=H_{u}(t,\bx(t),\bu(t),P_{1}(t),Q_{1}(t))$.
\end{theorem}

\begin{proof}
Let $v\in \T_{\mmu}(\bu)$ and $\nu_{0}\in\T_{K}(\bx_{0})$. Then, for any $\eps>0$, there exist  $v_{\eps}\in L_{\mmf}^{2}(\Omega;L^{2}(0,T;$ $\mrm))$ and $\nu_{0}^{\eps}\in\mrn$ such that $\bu+\eps v_{\eps}\in \mmu$, $\bx_{0}+\eps\nu_{0}^{\eps}\in K$  and
$$\me\int^{T}_{0}|v(t)-v_{\eps}(t)|^2dt\to 0,\quad |\nu_{0}^{\eps}-\nu_{0}|\to 0,\ \hbox{ as }\eps\to 0^+.$$
Expanding the cost functional $J(\cdot)$ at $\bu$, we have  for all small $\varepsilon > 0$,
\begin{eqnarray}\label{3.14}
0&\le& \frac{J(u^{\eps},x^{\eps}_{0})-J(\bu,\bx_{0})}{\eps}\nonumber\\
&=& \me \int_{0}^{T}\Big(\int_{0}^{1}\inner{ f_{x}(t,\bx(t)+\theta\delta x^{\eps}(t),\bu(t)+\eps v_{\eps}(t))}{\frac{\delta x^{\eps}(t)}{\eps}}d\theta\nonumber\\
& &\qquad\qquad
+\int_{0}^{1}\inner{f_{u}(t,\bx(t),\bu(t)+\theta \eps v_{\eps}(t))}{ v_{\eps}(t)}d\theta\Big)dt\nonumber\\
& &+\me \int_{0}^{1}\inner{g_{x}(\bx(T)+\theta\delta x^{\eps}(T))}{\frac{\delta x^{\eps}(T)}{\eps}}d\theta\nonumber\\
&=&\me \int_{0}^{T}\big(\inner{f_{x}(t)}{y_{1}(t)}
+\inner{f_{u}(t)}{v(t)}\big)dt+\me\inner{g_{x}(\bx(T))}{y_{1}(T)}+\rho_{1}^{\eps},
\end{eqnarray}
where
\begin{eqnarray}\label{2016302e1}
\rho_{1}^{\eps}&=&\me \int_{0}^{T}\Big(\int_{0}^{1}\inner{ f_{x}(t,\bx(t)+\theta\delta x^{\eps}(t),\bu(t)+\eps v_{\eps}(t))
- f_{x}(t)}{\frac{\delta x^{\eps}(t)}{\eps}}d\theta\nonumber\\
& &\qquad\quad
+\int_{0}^{1}\inner{f_{u}(t,\bx(t),\bu(t)+\theta \eps v_{\eps}(t))-f_{u}(t)}{ v_{\eps}(t)}d\theta\nonumber\\
& &\qquad\quad
+\inner{f_{x}(t)}{\frac{\delta x^{\eps}(t)}{\eps}-y_{1}(t)}
+\inner{f_{u}(t)}{v_{\eps}(t)-v(t)}
\Big)dt\nonumber\\
& &+\me \int_{0}^{1}\inner{g_{x}(\bx(T)+\theta\delta x^{\eps}(T))-g_{x}(\bx(T))}{\frac{\delta x^{\eps}(T)}{\eps}}d\theta\nonumber\\
& &\qquad\quad
+\me \inner{g_{x}(\bx(T))}{\frac{\delta x^{\eps}(T)}{\eps}-y_{1}(T)}.
\end{eqnarray}
By Lemma \ref{estimate one of varie qu} (with $\beta =2$) and (C2), it follows that
\begin{eqnarray*}
&&\Big|\me \int_{0}^{T}\int_{0}^{1}\inner{ f_{x}(t,\bx(t)+\theta\delta x^{\eps}(t),\bu(t)+\eps v_{\eps}(t))
- f_{x}(t)}{\frac{\delta x^{\eps}(t)}{\eps}}d\theta dt\Big|\\
&\le& \Big(\me \int_{0}^{T}\int_{0}^{1}\big| f_{x}(t,\bx(t)+\theta\delta x^{\eps}(t),\bu(t)+\eps v_{\eps}(t))
- f_{x}(t)\big|^{2}d\theta dt\Big)^{\frac{1}{2}}\Big(\me \int_{0}^{T}\big| \frac{\delta x^{\eps}(t)}{\eps}  \big|^{2}dt\Big)^{\frac{1}{2}}\\
&\le&C \Big[\me  \int_{0}^{T}\big(\big| \delta x^{\eps}(t)\big| + \big|\eps v_{\eps}(t)\big|\big)^{2}dt\Big]^{\frac{1}{2}}\cdot \Big(\me  \int_{0}^{T}\big| \frac{\delta x^{\eps}(t)}{\eps}  \big|^{2}dt\Big)^{\frac{1}{2}}\\
& &\to 0,\quad \hbox{ as }\eps\to 0^+.
\end{eqnarray*}
Similarly, we have
\begin{eqnarray*}
&&\Big|\me \int_{0}^{T}\int_{0}^{1}\inner{f_{u}(t,\bx(t),\bu(t)+\theta \eps v_{\eps}(t))-f_{u}(t)}{ v_{\eps}(t)}d\theta dt\Big|\\
&\le&C \Big(\me \int_{0}^{T} \big|\eps v_{\eps}(t)\big|^{2}dt\Big)^{\frac{1}{2}}\cdot\Big(\me \int_{0}^{T}\big| v_{\eps}(t)\big|^{2}dt\Big)^{\frac{1}{2}}\to 0,\quad \eps\to 0^+.
\end{eqnarray*}
and
\begin{eqnarray*}
&&\Big|\me \int_{0}^{1}\inner{g_{x}(\bx(T)+\theta\delta x^{\eps}(T))-g_{x}(\bx(T))}{\frac{\delta x^{\eps}(T)}{\eps}}d\theta\Big|\\
&\le&C\Big(\me\big|\delta x^{\eps}(T)\big|^{2}\Big)^{\frac{1}{2}}\cdot\Big(\me \big|\frac{\delta x^{\eps}(T)}{\eps}\big|^{2}\Big)^{\frac{1}{2}}\to 0,\quad \eps\to 0^+.
\end{eqnarray*}
Then, by (C2) and  Lemma \ref{estimate one of varie qu},
we obtain that
\begin{eqnarray}\label{2016302e2}
\lim_{\eps\to 0^+}\big|\rho_{1}^{\eps}\big|
&\le&\limsup_{\eps\to 0^+}\Big|\me \int_{0}^{T}\inner{f_{x}(t)}{\frac{\delta x^{\eps}(t)}{\eps}-y_{1}(t)}dt\Big|\nonumber\\
& &
+\limsup_{\eps\to 0^+}\Big|\me \int_{0}^{T}\inner{f_{u}(t)}{v_{\eps}(t)-v(t)}
dt\Big|\nonumber\\
& &+\limsup_{\eps\to 0^+}\Big|\me \inner{g_{x}(\bx(T))}{\frac{\delta x^{\eps}(T)}{\eps}-y_{1}(T)}\Big|= 0.
\end{eqnarray}
Therefore, from (\ref{3.14}) and (\ref{2016302e2}), we conclude that
\begin{eqnarray}\label{first order taylor exp}
0&\le& \me \int_{0}^{T}\big(\inner{f_{x}(t)}{y_{1}(t)}
+\inner{f_{u}(t)}{v(t)}\big)dt +\me\inner{g_{x}(\bx(T))}{y_{1}(T)}.
\end{eqnarray}

By the duality between (\ref{first vari equ}) and (\ref{first ajoint equ}), we have
\begin{eqnarray}\label{duality between y1 p1}
& &\me \inner{g_{x}(\bar{x}(T))}{y_{1}(T)}
=-\me\inner{P_{1}(T)}{y_{1}(T)}\nonumber\\
&=&-\inner{P_{1}(0)}{\nu_{0}}-\me\int_{0}^{T}\big(\inner{P_{1}(t)}{b_{x}(t)y_{1}(t)}
+\inner{P_{1}(t)}{b_{u}(t)v(t)}\nonumber\\
& &\qquad\quad+\inner{Q_{1}(t)}{\sigma_{x}(t)y_{1}(t)}
+\inner{Q_{1}(t)}{\sigma_{u}(t)v(t)}\nonumber\\
& &\qquad\quad- \inner{b_{x}(t)^{\top}P_{1}(t)}{y_{1}(t)}
-\inner{\sigma_{x}(t)^{\top}Q_{1}(t)}{y_{1}(t)}
+\inner{f_{x}(t)}{y_{1}(t)}\big)dt\nonumber\\
&=&-\inner{P_{1}(0)}{\nu_{0}}-\me\int_{0}^{T}\big( \inner{P_{1}(t)}{b_{u}(t)v(t)}
+\inner{Q_{1}(t)}{\sigma_{u}(t)v(t)}+\inner{f_{x}(t)}{y_{1}(t)}
\big)dt.\nonumber\\
\end{eqnarray}

Substituting (\ref{duality between y1 p1}) in (\ref{first order taylor exp}), we obtain that
\begin{eqnarray}\label{first order vari inequ}
0&\le& -\inner{P_{1}(0)}{\nu_{0}} -\me \int_{0}^{T}\big(\inner{P_{1}(t)}{b_{u}(t)v(t)}
+\inner{Q_{1}(t)}{\sigma_{u}(t)v(t)}
-\inner{f_{u}(t)}{v(t)}\big)dt\nonumber\\
&=&-\inner{P_{1}(0)}{\nu_{0}}-\me\int_{0}^{T}\inner{H_{u}(t)}{v(t)}dt.
\end{eqnarray}

For $v(\cdot)=0$, (\ref{first order vari inequ}) implies (\ref{trans condition}).
On the other hand, for $\nu_{0}=0$ in (\ref{first order vari inequ}), we have (\ref{first order integraltype condition}).
This completes the proof of Theorem \ref{TH first order integraltype condition}.
\end{proof}

From Theorem \ref{TH first order integraltype condition} and Lemma \ref{Lemma integral to pointwise}, it is easy to deduce  the following pointwise first order necessary condition.

\begin{theorem}\label{TH first order pointwise condition}
Let (C1)--(C2) hold.  If $(\bx,\bu,\bx_{0})$ is a local minimizer for the problem (\ref{minimum J}), then,
\begin{equation}\label{first order pointwise condition}
 H_{u}(t,\omega) \in \N_{U}(\bu(t,\omega)),\ a.e.\ t\in [0,T],\ a.s.\ \mbox{and}\ P_{1}(0)\in \N_{K}(\bx_{0}).
\end{equation}
\end{theorem}

\begin{remark}
When the control set $U$ and the initial state constraint set $K$ are also convex,
$\N_{U}(\bu)$ and $\N_{K}(\bx_{0})$ coincide with the normal cones of convex analysis. In this case, the condition (\ref{first order pointwise condition}) becomes
$$H_{u}(t,\omega)\in N_{U}(\bar{u}(t,\omega))\quad a.e.\ t\in [0,T],\ a.s. \ \mbox{and}\ P_{1}(0)\in N_{K}(\bx_{0}).$$
\end{remark}

\begin{remark}
If $T^{b}_{U}(\bar u(t,\omega))=\{0\}$ for a.e. $(t,\omega)\in
[0,T]\times\Omega$, then $\N_{U}(\bu(t,\omega))=\mrm$, for a.e.
$(t,\omega)\in [0,T]\times\Omega$, and the first condition in
(\ref{first order pointwise condition}) turns out to be trivial.
It is the case, for instance,  when the control set $U$ is a
finite union of singletons. Therefore, to have  the first
condition in (\ref{first order pointwise condition}) meaningful,
$U$ should have nontrivial tangent cones. It is not difficult to
verify that for every $v\in T^{b}_{\mathcal{U}_{ad}}(\bar{u})$,
and for a.e. $(t,\omega) \in [0,T]\times \Omega$,  the vector
$v(t,\omega)$ belongs to the {\em contingent cone}
$T_U^{B}(\bar{u}(t,\omega))$ to $U$  at $\bar{u}(t,\omega)$. Under
some suitable assumptions on $U$, we have
$T_U^{B}(\bar{u}(t,\omega))=T^{b}_{U}(\bar{u}(t,\omega))$ a.e. in
$[0,T]\times\Omega$, see \cite[Chapter 4]{Aubin90} for more
details. Consequently, under some convenient structural
assumptions on $U$, if $T^{b}_{\mathcal{U}_{ad}}(\bar{u}) \neq
\{0\}$, then $T^{b}_{U}(\bar u(t,\omega))\neq\{0\}$ on a set of
positive measure.
\end{remark}

\begin{remark}
Define
\begin{eqnarray*}
 {\cal H}(t,x,u,\omega)
\!\!&:=&\!\!H(t,x,u, P_{1}(t),Q_{1}(t),\omega)
\!-\!\frac{1}{2}\inner{ P_{2}(t)\sigma(t,\bx(t),\bu(t),\omega)}{\sigma(t,\bx(t),\bu(t),\omega)}\\
\!\!& + &\!\!
\frac{1}{2}\big\langle P_{2}(t)\big(\sigma(t,x,u,\omega)
\!-\!\sigma(t,\bx(t),\bu(t),\omega)\big),
\sigma(t,x,u,\omega)-\sigma(t,\bx(t),\bu(t),\omega)\big\rangle,
\end{eqnarray*}
where  $(P_{2},Q_{2})$ is the second order adjoint process with respect to $(\bx,\bu)$ (defined by (\ref{second ajoint equ}) in Section 4).
The stochastic maximum principle (e.g. \cite{Peng90}) says that, if $(\bx,\bu)$ is an optimal pair, then
\begin{equation}\label{maximum principle}
{\cal H}(t,\bx(t),\bu(t),\omega)
=\max_{v\in U} {\cal H}(t,\bx(t),v,\omega),\quad \ a.e.\ t\in [0,T],\ a.s.
\end{equation}
When $b$, $\sigma$ and $f$ are differentiable with respect to the variable $u$,
 (\ref{maximum principle}) implies that
$$\inner{H_{u}(t,\omega)}{v}\le 0,\quad \forall\; v\in \T_{U}(\bu(t,\omega)),\ a.e.\ t\in [0,T],\ a.s,$$
i.e., the first condition in (\ref{first order pointwise condition}) holds (when $U$ is  convex, this also coincides with the corresponding result in \cite{Bensoussan81}).
However, to derive the maximum principle (\ref{maximum principle}) one has to assume that $b$, $\sigma$, $f$ and $g$ are differentiable up to the second order with respect to the variable $x$, and  the second order adjoint process $(P_{2},Q_{2})$ should be introduced (even it does not appear in the condition (\ref{first order pointwise condition})). Therefore, in practice,  under the usual structural assumptions on $U$, it is more convenient to use the condition (\ref{first order pointwise condition}) directly.
\end{remark}

 In what follows we  give a simple example to demonstrate how to use
  the condition (\ref{first order pointwise condition}) to check if  a given admissible control is not optimal.

\begin{example}\label{example1}
Let $n=m=2$, $T=1$, $U=\{(u_{1},u_{2})\in \mr^2\ |\ u_{1}u_{2}=0, u_{1}\in [-1,1], u_{2}\in [-1,1]\}$. Clearly, this $U$ is neither a finite set nor convex in $\mr^2$. Consider the control system
\begin{equation}\label{controlsys example1}
\left\{
\begin{array}{l}
dx_{1}(t)=(x_{2}(t)-\frac{1}{2})dt+dW(t),\ \ \ t\in[0,1],\\
dx_{2}(t)=u_{1}(t)dt+u_{2}(t)dW(t),\ \ \ t\in[0,1],\\
x_{1}(0)=0, x_{2}(0)=0
\end{array}\right.
\end{equation}
with the cost functional
\begin{equation}\label{20160712e1}
J(u)=\frac{1}{2}\me |x_{1}(1)-W(1)|^2.
\end{equation}
Define the Hamiltonian of this optimal control problem
\begin{equation}\label{20160712e2}
H(t,(x_{1},x_{2}),(u_{1},u_{2}), (p_{1}^{1},p_{1}^{2}),(q_{1}^{1},q_{1}^{2}),\omega)
=p_{1}^{1}(x_{2}-\frac{1}{2})+p_{1}^{2}u_{1}+q_{1}^{1}+q_{1}^{2}u_{2},
\end{equation}
for all $(t,(x_{1},x_{2}),(u_{1},u_{2}),
(p_{1}^{1},p_{1}^{2}),(q_{1}^{1},q_{1}^{2}),\omega)\in
[0,1]\times\mr^2\times \mr^2\times\mr^2\times\mr^2\times\Omega$. In
what follows, we show that the control
$(u_{1}(t),u_{2}(t))\equiv(0,0)$ is not a local minimizer.

Obviously, the corresponding solution to the control system (\ref{controlsys example1}) is
 \begin{equation}\label{nding solution1}
 (x_{1}(t),x_{2}(t))=(W(t)-\frac{t}{2}, 0),
 \end{equation}
and the first order adjoint equation is
\begin{equation}\label{first order adjequ example1}
\left\{
\begin{array}{l}
dP_{1}^{1}(t)=Q_{1}^{1}(t)dW(t),\ \ \ t\in[0,1],\\
dP_{1}^{2}(t)=-P_{1}^{1}(t)dt+Q_{1}^{2}(t)dW(t),\ \ \ t\in[0,1],\\
P_{1}^{1}(1)=\frac{1}{2}, \quad P_{1}^{2}(1)=0
\end{array}\right.
\end{equation}
It is easy to verify that the solution to (\ref{first order adjequ example1}) is
\begin{equation}\label{20160712e3}
(P_{1}^{1}(t),Q_{1}^{1}(t))= (\frac{1}{2}, 0), \
(P_{1}^{2}(t),Q_{1}^{2}(t))=(\frac{1-t}{2}, 0),\quad a.e.\ (t,\omega)\in [0,1]\times\Omega.
\end{equation}
Note that even though the Mangasarian-Fromowitz constraint qualification does not hold at $(0,0)$, we can easily obtain that
$$\T_{U}((0,0))=\{(v_{1},v_{2})\in \mr^2\ |\ v_{1}v_{2}=0\}.$$
By the first order condition in (\ref{first order pointwise condition}),
$$\inner{H_{u}(t)}{v}=P_{1}^{2}(t)v_{1}\le 0,\quad \forall \ v=(v_{1},v_{2})\in \T_{U}((0,0)).$$
Since $P_{1}^{2}(t)=\frac{1}{2}(1-t)>0$ for any  $t\in[0,1)$, a.s., chose $(v_{1},v_{2})=(1,0)$ we have
$$P_{1}^{2}(t)v_{1}=\frac{1}{2}(1-t)>0, \quad a.e.\ (t,\omega)\in [0,1]\times\Omega,$$
which is a contradiction. Therefore, $(u_{1}(t),u_{2}(t))\equiv(0,0)$ is not an local minimizer.

Actually, choosing $(\bar u_{1}(t),\bar u_{2}(t))\equiv(1,0)$, we find that the corresponding state is
\begin{equation}\label{20160710e3}
(\bar x_{1}(t),\bar x_{2}(t))=\Big(\frac{t^2}{2}-\frac{t}{2}+W(t), t\Big),\quad \forall\ (t,\omega)\in [0,1]\times\Omega,
\end{equation}
and hence $\bar x_{1}(1)=W(1)$, i.e., the cost functional attains its minimum $0$ and $(\bar u_{1}(t),\bar u_{2}(t))\equiv(1,0)$ is the global minimizer. In addition, a simple calculation shows that the corresponding first order adjoint process is
\begin{equation}\label{20160710e4}
(P_{1}^{1}(t),Q_{1}^{1}(t))= (0, 0), \
(P_{1}^{2}(t),Q_{1}^{2}(t))=(0, 0),\quad \forall\ (t,\omega)\in [0,1]\times\Omega,
\end{equation}
which implies that the condition (\ref{first order pointwise condition}) is trivially satisfied.
\end{example}

\begin{remark}
The approach proposed in Theorems \ref{TH first order integraltype condition}--\ref{TH first order pointwise condition} can be applied to more general control problems. We refer the reader to \cite{WangZhang15} for the optimal control problems involving stochastic Volterra integral equations.
\end{remark}

\section{Second order necessary conditions}\label{s4}
In this section, we investigate the second order necessary conditions for the local minimizers $(\bx,\bu,\bx_{0})$ of (\ref{minimum J}).
In addition to the assumptions (C1) and (C2), we suppose that

\begin{enumerate}
  \item [{\bf (C3)}] {\em The functions $b$, $\sigma$, $f$ and $g$ satisfy the following:}
  \begin{enumerate}[{\rm (i)}]
      \item {\em For a.e. $(t,\omega)\in [0, T]\times\Omega$, the functions
             $b(t, \cdot, \cdot,\omega):\ \mrn\times \mrm\to \mrn$ and $\sigma(t, \cdot, \cdot,\omega):\ \mrn\times \mrm\to \mrn$
             are twice differentiable
             and
             $$(x,u)\mapsto (b_{(x,u)^2}(t,x,u,\omega),\sigma_{(x,u)^2}(t,x,u,\omega))$$
             is uniformly continuous in $x\in \mrn$ and $u\in \mrm$,
             and,}
             $$|b_{(x,u)^2}(t,x,u,\omega)| +  |\sigma_{(x,u)^2}(t,x,u,\omega)|\le L,\qquad\forall\; (x,u)\in \mrn\times \mrm;$$

       \item {\em For a.e. $(t,\omega)\in [0, T]\times\Omega$, the functions $f(t, \cdot, \cdot,\omega):\ \mrn\times \mrm\to \mr$ and $g(\cdot,\omega):\ \mrn \to \mr$ are  twice continuously differentiable, and for any $x,\ \tilde{x}\in\mrn$ and $u,\ \tilde{u}\in \mrm$,}
             $$
             \left\{
             \begin{array}{l}
              |f_{(x,u)^2}(t,x,u,\omega)|\le L,\\
              |f_{(x,u)^2}(t,x,u,\omega)-f_{(x,u)^2}(t,\tilde{x},\tilde{u},\omega)|\le L(|x-\tilde{x}|+|u-\tilde{u}|),\\
              |g_{xx}(x,\omega)| \le L, \ |g_{xx}(x,\omega)-g_{xx}(\tilde{x},\omega)|\le L|x-\tilde{x}|.
             \end{array}\right.
             $$
  \end{enumerate}
\end{enumerate}

For $\varphi=b,\; \sigma, \;  f$, denote
$$
\varphi_{xx}(t)=\varphi_{xx}(t,\bar{x}(t),\bar{u}(t)),\quad
\varphi_{xu}(t)=\varphi_{xu}(t,\bar{x}(t),\bar{u}(t)),\quad
\varphi_{uu}(t)=\varphi_{uu}(t,\bar{x}(t),\bar{u}(t)).
$$

\subsection{Integral-type second order necessary conditions}

In this subsection, we consider first the integral-type second order necessary conditions for the local minimizers of (\ref{minimum J}).

Let  $\bu, v, h,h_{\eps}\in L^{2\beta}_{\mmf}(\Omega; L^{4}(0,T;$ $\mrm))$ ($\beta \geq 1$) and $\nu_{0}, \varpi_{0}, \varpi_{0}^{\eps}\in \mrm$ be such that $h_{\eps}$ converges  to $h$ in $L^{2\beta}_{\mmf}(\Omega;L^{4}(0,T;\mrm))$ and $\varpi_{0}^{\eps} \to \varpi_{0}$ in $\mrm$ as $\eps\to 0^+$.
Set
$$u^{\eps}:=\bu+\eps v+\eps^2 h_{\eps}, \qquad x^{\eps}_{0}:=\bx_{0}+\eps \nu_{0}+\eps^2 \varpi_{0}^{\eps}.
$$

Denote by $x^{\eps}$ the solution  of (\ref{controlsys})
corresponding to the control $u^{\eps}$ and the initial datum
$x^{\eps}_{0}$. Put
$$\delta x^{\eps}=x^{\eps}-\bar{x}, \qquad\delta u^{\eps}=\eps v+\eps^2 h_{\eps}.
$$

Similarly to \cite{Hoehener12}, we introduce the following second-order variational equation:
\begin{equation}\label{second order vari equ}
\quad\left\{
\begin{array}{l}
dy_{2}(t)= \Big(b_{x}(t)y_{2}(t) + 2b_{u}(t)h(t) +y_{1}(t)^{\top}b_{xx}(t)y_{1}(t)+2v(t)^{\top}b_{xu}(t)y_{1}(t)\\
\qquad\qquad+v(t)^{\top}b_{uu}(t)v(t)\Big)dt
+\Big(\sigma_{x}(t)y_{2}(t)+2\sigma_{u}(t)h(t) +y_{1}(t)^{\top}\sigma_{xx}(t)y_{1}(t)\\
\qquad\qquad+2v(t)^{\top}\sigma_{xu}(t)y_{1}(t)
+v(t)^{\top}\sigma_{uu}(t)v(t)\Big)dW(t),\qquad t\in[0,T],\\
y_{2}(0)=2\varpi_{0},
\end{array}\right.
\end{equation}
where $y_{1}$ is the solution to the first variational equation (\ref{first vari equ}) (for $v(\cdot)$ and $\nu_{0}$ as above).
We have the following estimates.
 \begin{lemma}\label{estimate two of varie qu}
Let (C2)--(C3) hold and $\beta \geq 1$. Then, for  $\bu, v, h,h_{\eps}\in L^{2\beta}_{\mmf}(\Omega; L^{4}(0,T;\mrm))$ and $\nu_{0}, \varpi_{0}, \varpi_{0}^{\eps}\in \mrm$ as above, we have
\begin{equation*}
\|y_{2}\|_{\infty,\beta}^{\beta}\le C(|\varpi_{0}|^{\beta}+|\nu_{0}|^{2\beta}+\|v\|_{4,2\beta}^{2\beta}+\|h\|_{2,\beta}^{\beta}).
\end{equation*}
Furthermore,
\begin{equation}\label{r2 to 0}
\|r_{2}^{\eps}\|_{\infty,\beta}^{\beta}\to 0,\quad \eps\to 0^+,
\end{equation}
where,
$$r_{2}^{\eps}(t,\omega):=\frac{\delta x^{\eps }(t,\omega)-\eps  y_{1}(t,\omega)}{\eps ^2} -\frac{1}{2}y_{2}(t,\omega).$$

\end{lemma}
\begin{proof}
See Appendix B.
\end{proof}

\vspace{0.5mm}

We now introduce  the following adjoint equation for  (\ref{second order vari equ}):
\begin{equation}\label{second ajoint equ}
\quad\left\{
\begin{array}{l}
dP_{2}(t)=-\Big(b_{x}(t)^{\top}P_{2}(t)+P_{2}(t)b_{x}(t) +\sigma_{x}(t)^{\top}P_{2}(t)\sigma_{x}(t) +\sigma_{x}(t)^{\top}Q_{2}(t)\\
\qquad\qquad \qquad
+Q_{2}(t)\sigma_{x}(t)+H_{xx}(t)\Big)dt+Q_{2}(t)dW(t),\  t\in[0,T], \\
P_{2}(T)=-g_{xx}(\bar{x}(T)),
\end{array}\right.
\end{equation}
where $H_{xx}(t)=H_{xx}(t,\bar{x}(t),\bar{u}(t), P_{1}(t),Q_{1}(t))$ with $(P_1(\cdot),Q_1(\cdot))$ given by (\ref{first ajoint equ}).

By \cite{Peng97} and (C2)--(C3), it is easy to check that,  if $\bar u\in L_{\mmf}^{\beta}(\Omega; L^{2}(0,T; \mrm))$, (\ref{second ajoint equ}) admits a unique strong solution
$(P_{2}(\cdot),Q_{2}(\cdot))\in L_{\mmf}^{\beta}(\Omega; C([0,T]; \mathbf{S}^n))\times
L_{\mmf}^{\beta}(\Omega; L^{2}(0,T; \mathbf{S}^n))$  for any $\beta\ge1$.

To simplify the notation, we define
\begin{eqnarray}\label{S}
\ms(t,x,u,y_{1},z_{1},y_{2},z_{2},\omega)
&:=& H_{xu}(t,x,u,y_{1},z_{1},\omega)+b_{u}(t,x,u,\omega)^{\top}y_{2}\\
& &+\sigma_{u}(t,x,u,\omega)^{\top}z_{2}+\sigma_{u}(t,x,u,\omega)^{\top}
y_{2}\sigma_{x}(t,x,u,\omega),\nonumber
\end{eqnarray}
where $(t,x,u,y_{1},z_{1},y_{2},z_{2},\omega)\in [0,T]\times\mrn\times \mrm \times\mrn\times\mrn\times \mathbf{S}^n\times \mathbf{S}^n\times\Omega$,
and denote
\begin{equation}\label{S(t)}
\ms(t)= \ms(t,\bar{x}(t),\bar{u}(t),P_{1}(t),Q_{1}(t),P_{2}(t),Q_{2}(t)),
\quad t \in [0,T].
\end{equation}

Let $\bu\in \mmu\cap L_{\mmf}^{4}(\Omega;L^{4}(0,T;\mrm))$. Define
$$\ups :=\Big\{v \in L_{\mmf}^{2}(\Omega;L^{2}(0,T;\mrm)) \ \Big|\
\inner{H_{u}(t,\omega)}{v(t,\omega)}=0 \;\;\mbox{\rm a.e.} \;\; t \in [0,T],   \;\;\mbox{\rm a.s. }\Big\},$$
and the set of admissible second order variations by
\begin{eqnarray*}
& &\ma:=\Big\{(v, h)\in L_{\mmf}^{4}(\Omega;L^{4}(0,T;\mrm))\times L_{\mmf}^{4}(\Omega;L^{4}(0,T;\mrm))\ ~\Big|~\
\\
& &\qquad\qquad\qquad\qquad
 h(t,\omega)\in \TT_{U}(\bar u (t,\omega), v(t,\omega)), \
\mbox{a.e.}\ t\in[0,T],\ \mbox{a.s. }
\Big\}.
\end{eqnarray*}
Denote
$$ \ma^{1}:=\Big\{v\in L_{\mmf}^{4}(\Omega;L^{4}(0,T;\mrm))~\Big|~ \exists\ h\in L_{\mmf}^{4}(\Omega;L^{4}(0,T;\mrm)),\ \mbox{s.t.}\ (v, h)\in \ma \Big\}.$$

We have the following result.

\begin{theorem}\label{TH second order integral condition}
Let (C1)--(C3) hold and $(\bx,\bu,\bx_{0})$ be a local minimizer for the problem (\ref{minimum J}) with $\bu\in L_{\mmf}^{4}(\Omega;L^{4}(0,T;\mrm))$.   Then for the adjoint process  $P_{1}$   defined by (\ref{first ajoint equ}) (relative to $(\bx,\bu,\bx_{0})$) and
for all $(v,h)\in \ma$  satisfying $v\in \ups$,
\begin{eqnarray}\label{second order integral condition}
&&\me\int_{0}^{T}\Big(
2\inner{H_{u}(t)}{h(t)}
+\inner{H_{uu}(t)v(t)}{v(t)}\nonumber\\
&&\qquad\qquad+\inner{P_{2}(t)\sigma_{u}(t)v(t)}{\sigma_{u}(t)v(t)}
+2\inner{\ms(t)y_{1}(t)}{v(t)}\Big)dt\le 0,
\end{eqnarray}
and
\begin{equation}\label{second order trans}
P_{2}(0)  \in  \NN_{K}(x,P_{1}(0)).
\end{equation}
\end{theorem}
\begin{proof} We borrow some ideas from  \cite[proof of Theorem 2]{Frankowska15}.

From the definition of the second order adjacent set, we deduce
that, if $(v, h)\in \ma$, then $v(t,\omega)\in \T_{U}(\bar u
(t,\omega))$, a.e. $(t,\omega)\in [0,T]\times\Omega$, and for any
$\eps>0$, there exist an $r(\eps,t,\omega)\in \mrm$ such that
$$\bar u(t,\omega)+\eps v(t,\omega)+\eps^2 h(t,\omega) +r(\eps,t,\omega)\in U,\ r(\eps,t,\omega)=o(\eps^2),\ \text{a.e.}\ (t,\omega)\in [0,T]\times\Omega.$$
Furthermore, let $\ell(t,\omega)=|h(t,\omega)|+1$, then for a.e.
$(t,\omega)\in [0,T]\times\Omega$ there exists a
$\rho(t,\omega)>0$ such that
\begin{equation}\label{20160710e1}
\begin{array}{ll}
dist(\bar u(t,\omega)+\eps v(t,\omega), U)\\[2mm]
\le |\bar u(t,\omega)+\eps v(t,\omega)-(\bar u(t,\omega)+\eps v(t,\omega)+\eps^2 h(t,\omega) +r(\eps,t,\omega))|\\[2mm]
= |\eps^2 h(t,\omega)+ r(\eps,t,\omega)|\le \eps^2\ell(t,\omega),\  \ \forall\; \eps\in [0,\rho(t,\omega)].
\end{array}
\end{equation}

Motivated by the inequality (\ref{20160710e1}), we introduce the following subset of $\ma$:
 $$
 \begin{array}{ll}
 \ma^*=\big\{(v,h)\in\ma\;\big|\; \exists \hbox{ a }\rho_{0} > 0 \hbox{ (independent of }(t,\omega)) \hbox{ such that}\\[2mm]
\qquad\qquad\qquad dist(\bar u(t,\omega)+\eps v(t,\omega), U)\le \eps^2\ell(t,\omega),\ \forall\; \eps\in [0,\rho_{0}]\big\}.
 \end{array}
$$
We fist prove that (\ref{second order integral condition}) and
(\ref{second order trans}) hold for any $(v,h)\in\ma^*$ satisfying
$v\in \ups$.  Fix such a   $(v,h)\in\ma^*$  and a corresponding
$\rho_0>0.$

Using similar arguments as those  in the proof of
\cite[Proposition 4.2]{Hoehener12}, we now prove that $v\in
\T_{\mmu}(\bar u)$ and $h\in \TT_{\mmu}(\bar u, v)$.

Define
$$\alpha_{\eps}(t,\omega)=dist(\bar u(t,\omega)+\eps v(t,\omega), U).$$
The distance function being Lipschitz continuous, $\alpha_{\eps}$
is a $\mb([0,T])\otimes\mf$-measurable and $\mmf$-adapted process.
Furthermore, since, $v(t,\omega)\in \T_{U}(\bar u(t,\omega))$
a.s., we have $\alpha_{\eps}(t,\omega)/\eps\to 0$ a.e.
$(t,\omega)\in [0,T]\times\Omega$ as $\eps\to 0^+$.

On the other hand, $U$ being a closed set in $\mrm$, for a.e.
$(t,\omega)\in [0,T]\times\Omega$ there exists a
$u_{\eps}(t,\omega)\in U$ such that
$$\alpha_{\eps}(t,\omega)=|u_{\eps}(t,\omega)-\bar u(t,\omega)-\eps v(t,\omega)|\le \eps^2\ell(t,\omega)\;\; \forall \; \varepsilon \in [0,\rho_0].$$
Using Lemma \ref{metric projection}, we show that  $u_{\eps}$
admits a $\mb([0,T])\otimes\mf$-measurable and $\mmf$-adapted
version (Note that the metric projection mapping
$(t,\omega)\rightsquigarrow\Pi_{U}(\bar u(t,\omega)+\eps
v(t,\omega))$ may not be  $\mb([0; T])\otimes\mf$-measurable,
since $([0,T]\times\Omega,  \mb([0; T])\otimes\mf, dt\times dP)$
is not complete. Therefore, we can only obtain a measurable
selection of $(t,\omega)\rightsquigarrow\Pi_{U}(\bar
u(t,\omega)+\eps v(t,\omega))$ on the completion of this product
measure space and then modify this selection to be a
$\mb([0,T])\otimes\mf$-measurable process.)
 To simplify the notation, we still denote this version by $u_{\eps}$.

For $v_{\eps}=(u_{\eps}-\bar u)/\eps$, we have
$$
|v_{\eps}(t,\omega)-v(t,\omega)|=\Big|\frac{u_{\eps}(t,\omega)-\bar
u(t,\omega)}{\eps}-v(t,\omega)\Big|=
\Big|\frac{\alpha_{\eps}(t,\omega) }{\eps}\Big|\le
\eps\ell(t,\omega).
$$
Since $(v,h)\in\ma^*$, it follows that $v_{\eps}\in
L_{\mmf}^{4}(\Omega;L^{4}(0,T;\mrm))$ and, by the dominated
convergence theorem, $v_{\eps}\to v$ in
$L_{\mmf}^{4}(\Omega;L^{4}(0,T;\mrm))$ as $\eps\to 0^+$. By the
definition of $v_{\eps}$, we get  $\bar u(t,\omega)+\eps
v_{\eps}(t,\omega)=u_{\eps}(t,\omega)\in U$, a.e. $(t,\omega)\in
[0,T]\times\Omega$. This proves that $v\in \T_{\mmu}(\bar u)$.

Similarly, define
$$\gamma_{\eps}(t,\omega)=dist(\bar u(t,\omega)+\eps v(t,\omega)+\eps^2 h(t,\omega), U).$$
Then, $\gamma_{\eps}$ is $\mb([0,T])\otimes\mf$-measurable and
$\mmf$-adapted, and, because $h(t,\omega)\in \TT_{U}(\bar
u(t,\omega),$ $v(t,\omega))$, a.e. $(t,\omega)\in
[0,T]\times\Omega$, $\gamma_{\eps}(t,\omega)/\eps^2\to 0$  a.e.
$(t,\omega)\in [0,T]\times\Omega$ as $\eps\to 0^+$.

Choose a $\mb([0,T])\otimes\mf$-measurable and $\mmf$-adapted processes $w_{\eps}(t,\omega)\in U$, such that
$$\gamma_{\eps}(t,\omega)=|w_{\eps}(t,\omega)-\bar u(t,\omega)-\eps v(t,\omega)-\eps^2 h(t,\omega)|, \ \mbox{a.e.}
\ (t,\omega)\in [0,T]\times\Omega$$
and define
$$h_{\eps}=\frac{w_{\eps}-\bar u-\eps v}{\eps^2}.$$
Then,
\begin{eqnarray*}
|h_{\eps}(t,\omega)-h(t,\omega)|&=&\Big|\frac{w_{\eps}-\bar u(t,\omega)-\eps v(t,\omega)}{\eps^2}-h(t,\omega)\Big|\\
&\le&\Big|\frac{u_{\eps}-\bar u(t,\omega)-\eps v(t,\omega)-\eps^2 h(t,\omega)}{\eps^2}\Big|
\le \frac{\alpha_{\eps}(t,\omega)}{\eps^2}+|h(t,\omega)|\\
&\le& \ell(t,\omega)+|h(t,\omega)|,\quad \mbox{a.e.}\ (t,\omega)\in [0,T]\times\Omega,
\end{eqnarray*}
and hence $h_{\eps}\in L_{\mmf}^{4}(\Omega;L^{4}(0,T;\mrm))$.
Moreover, by the definition of $h_{\eps}$,
$$\bar u(t,\omega)+\eps v(t,\omega)+\eps^2 h_{\eps}(t,\omega)=w_{\eps}(t,\omega)\in U,\ \mbox{a.e.}\ (t,\omega)\in [0,T]\times\Omega,$$
and
$$|h_{\eps}(t,\omega)-h(t,\omega)|=\Big|\frac{\gamma_{\eps}(t,\omega) }{\eps^2}\Big|\to 0,\quad
\ \mbox{a.e.}\ (t,\omega)\in [0,T]\times\Omega.$$
By the dominated convergence theorem, $h_{\eps}\to h$ in $L_{\mmf}^{4}(\Omega;L^{4}(0,T;\mrm))$ as $\eps\to 0^+$. This proves that $h\in \TT_{\mmu}(\bar u,v)$.

Let $\nu_{0} \in \T_{K}(\bx_{0})\cap \{P_1(0)\}^\perp$ and $\varpi_{0}\in \TT_{K}(\bx_{0},\nu_{0})$.

Define $u^{\eps}=\bar u+\eps v+\eps^2 h_{\eps}$ and let
$x_{0}^{\eps}$, $\delta x^{\eps}$ and $\delta u^{\eps}$  be
defined as above. Denote
$\tilde{f}_{xx}^{\eps}(t):=\int_{0}^{1}(1-\theta)f_{xx}(t,\bx(t)
+ \theta\delta x^{\eps}(t),\bu(t)+\theta\delta
u^{\eps}(t))d\theta$. Mappings  $\tilde{f}_{xu}^{\eps}(t)$,
$\tilde{f}_{uu}^{\eps}(t)$ and $\tilde{g}_{xx}^{\eps}(T)$ are
defined in a similar way.

Expanding the cost functional $J$ at $\bu$, we get
\begin{eqnarray*}
& & \frac{J(u^{\eps})-J(\bu)}{\eps^2}\\
&=&\frac{1}{\eps^2}\me\int_{0}^{T}\Big(\inner{f_x(t)}{\delta x^{\eps}(t)}+
\inner{f_{u}(t)}{\delta u^{\eps}(t)}
+\inner{\tilde{f}_{xx}^{\eps}(t)\delta x^{\eps}(t)}{\delta x^{\eps}(t)}
\\
& &
+2\inner{\tilde{f}_{xu}^{\eps}(t)\delta x^{\eps}(t)}{\delta u^{\eps}(t)}
+\inner{\tilde{f}_{uu}^{\eps}(t)\delta u^{\eps}(t)}{\delta u^{\eps}(t)}\Big)dt\\
& &
+\frac{1}{\eps^2}\me \Big(\inner{g_{x}(\bx(T))}{\delta x^{\eps}(T)}+\inner{\tilde{g}_{xx}^{\eps}(\bx(T))\delta x^{\eps}(T)}{\delta x^{\eps}(T)}
\Big)\\
&=& \me\int_{0}^{T}\Big[\frac{1}{\eps}\inner{f_{x}(t)}{ y_{1}(t)}
+\frac{1}{2}\inner{f_{x}(t)}{y_{2}(t)}
+\frac{1}{\eps}\inner{f_{u}(t)}{v(t)}+\inner{f_{u}(t)}{h(t)}\\
& &+\frac{1}{2}\Big(\inner{f_{xx}(t)y_{1}(t)}{y_{1}(t)}
+2\inner{f_{xu}(t)y_{1}(t)}{v(t)}
+\inner{f_{uu}(t)v(t)}{v(t)}\Big)\Big]dt\\
& &
+\me \Big(\frac{1}{\eps} \inner{g_{x}(\bar{x}(T))}{y_{1}(T)}
+\frac{1}{2} \inner{g_{x}(\bx(T))}{y_{2}(T)} \\
& &+\frac{1}{2}\inner{g_{xx}(\bx(T))y_{1}(T)}{y_{1}(T)}\Big)
+ \rho_{2}^{\eps},
\end{eqnarray*}
where
\begin{eqnarray*}
\rho_{2}^{\eps}
&=&\me \int_{0}^{T}\Big( \inner{f_{x}(t)}{r_{2}^{\eps}(t)}+               \inner{f_{u}(t)}{h_{\eps}(t)-h(t)}\Big)dt+\me\inner{g_{x}(\bx(T))}{r_{2}^{\eps}(T)}\\
&&
+\me \int_{0}^{T}\Big[\Big(\inner{\tilde{f}_{xx}^{\eps}(t)\frac{\delta x^{\eps}(t)}{\eps}}{\frac{\delta x^{\eps}(t)}{\eps}}-\frac{1}{2}\inner{f_{xx}(t)y_{1}(t)}{y_{1}(t)}\Big)\\
&&\qquad\quad
+\Big(2\inner{\tilde{f}_{xu}^{\eps}(t)\frac{\delta x^{\eps}(t)}{\eps}}{\frac{\delta u^{\eps}(t)}{\eps}}-\inner{f_{xu}(t)y_{1}(t)}{v(t)}\Big)\\
&&\qquad\quad
+\Big(\inner{\tilde{f}_{uu}^{\eps}(t)\frac{\delta u^{\eps}(t)}{\eps}}{\frac{\delta u^{\eps}(t)}{\eps}}-\frac{1}{2}\inner{f_{uu}(t)v(t)}{v(t)}\Big)\Big]dt\\
&&+\me \Big(\inner{\tilde{g}_{xx}^{\eps}(\bx(T))\frac{\delta x^{\eps}(T)}{\eps}}{\frac{\delta x^{\eps}(T)}{\eps}}-\frac{1}{2}\inner{g_{xx}(\bx(T))y_{1}(T)}{y_{1}(T)}\Big).
\end{eqnarray*}

In the same way as in the proof of Lemma \ref{second order vari equ}, we  find that $\lim_{\eps\to 0^+}\rho_{2}^{\eps}=0$. On the other hand,
by (\ref{duality between y1 p1}) and, recalling that $v\in \ups$,   $\nu_{0} \in \{P_1(0)\}^\perp$, we have
\begin{eqnarray*}
& &\frac{1}{\eps}\me\int_{0}^{T}\Big(\inner{f_{x}(t)}{ y_{1}(t)}
+\inner{f_{u}(t)}{v(t)}\Big)dt
+\frac{1}{\eps}\me\inner{g_{x}(\bar{x}(T))}{y_{1}(T)}\nonumber\\
&=&-\frac{1}{\eps}\inner{P_{1}(0)}{\nu_{0}}-\frac{1}{\eps}\me\int_{0}^{T}\inner{H_{u}(t)}{v(t)}dt=0.
\end{eqnarray*}
Therefore,
\begin{eqnarray}\label{taylorexpconvex}
0&\le&\lim_{\eps\to 0^+} \frac{J(u^{\eps}(\cdot))-J(\bu(\cdot))}{\eps^2}\nonumber\\
&=& \me\int_{0}^{T}\Big[\frac{1}{2}\inner{f_{x}(t)}{y_{2}(t)}
+\inner{f_{u}(t)}{h(t)}\nonumber\\
& &+\frac{1}{2}\Big(\inner{f_{xx}(t)y_{1}(t)}{y_{1}(t)}
+2\inner{f_{xu}(t)y_{1}(t)}{v(t)}
+\inner{f_{uu}(t)v(t)}{v(t)}\Big)\Big]dt\nonumber\\
& &
+\frac{1}{2}\me\Big(\inner{g_{x}(\bx(T))}{y_{2}(T)}
+\inner{g_{xx}(\bx(T))y_{1}(T)}{y_{1}(T)}\Big).\nonumber\\
\end{eqnarray}

By It\^{o}'s formula,
\begin{eqnarray}\label{hxty2}
& &\me ~\inner{g_{x}(\bar{x}(T))}{y_{2}(T)}
= -\me ~\inner{P_{1}(T)}{y_{2}(T)}\\
&=&-2\inner{P_{1}(0)}{\varpi_{0}}-\me\int_{0}^{T}\Big(2\inner{P_{1}(t)}{b_{u}(t)h(t)}
+\inner{P_{1}(t)}{y_{1}(t)^{\top}b_{xx}(t)y_{1}(t)}
\nonumber\\
& &
+2\inner{P_{1}(t)}{v(t)^{\top}b_{xu}(t)y_{1}(t)}+\inner{P_{1}(t)}{v(t)^{\top}b_{uu}(t)v(t)}
+2\inner{Q_{1}(t)}{\sigma_{u}(t) h(t)}
\nonumber\\
& &
+\inner{Q_{1}(t)}{y_{1}(t)^{\top}\sigma_{xx}(t)y_{1}(t)}+2\inner{Q_{1}(t)}{v(t)^{\top}\sigma_{xu}(t)y_{1}(t)}
\nonumber\\
& &
+\inner{Q_{1}(t)}{v(t)^{\top}\sigma_{uu}(t)v(t)}+\inner{f_{x}(t)}{y_{2}(t)}\Big)dt,\nonumber
\end{eqnarray}
and
\begin{eqnarray}\label{hxxty12}
& &\me ~\inner{ g_{xx}(\bar{x}(T))y_{1}(T)}{y_{1}(T)}
= -\me ~\inner{P_{2}(T)y_{1}(T)}{y_{1}(T)}\\
&=&-\inner{P_{2}(0)\nu_{0}}{\nu_{0}}-\me\int_{0}^{T}\Big(2\inner{P_{2}(t)y_{1}(t)}{b_{u}(t)v(t)}
+2\inner{P_{2}(t)\sigma_{x}(t)y_{1}(t)}{\sigma_{u}(t)v(t)}
\nonumber\\
& &
+\inner{P_{2}(t)\sigma_{u}(t)v(t)}{\sigma_{u}(t)v(t)}
+2\inner{Q_{2}(t)\sigma_{u}(t)v(s)}{y_{1}(t)}
-\inner{H_{xx}(t)y_{1}(t)}{y_{1}(t)}
\Big)dt.\nonumber
\end{eqnarray}
Substituting (\ref{hxty2}) and (\ref{hxxty12}) into
(\ref{taylorexpconvex})  yields
\begin{eqnarray*}
0&\ge&\inner{P_{1}(0)}{\varpi_{0}}+\frac{1}{2}\inner{P_{2}(0)\nu_{0}}{\nu_{0}}
\nonumber\\
& &+\me\int_{0}^{T}\Big[\Big(
\inner{P_{1}(t)}{b_{u}(t)h(t)}
+\inner{Q_{1}(t)}{\sigma_{u}(t)h(t)}
-\inner{f_{u}(t)}{h(t)}\Big)\nonumber\\
& &\qquad\quad
+\frac{1}{2}\Big(\inner{P_{1}(t)}{v(t)^{\top}b_{uu}(t)v(t)}
+\inner{Q_{1}(t)}{v(t)^{\top}\sigma_{uu}(t)v(t)}
-\inner{f_{uu}(t)v(t)}{v(t)}\Big)\nonumber\\
& &\qquad\quad
+\frac{1}{2}\inner{P_{2}(t)
\sigma_{u}(t)v(t)}{\sigma_{u}(t)v(t)}
+\Big(
\inner{P_{1}(t)}{v(t)^{\top}b_{xu}(t)y_{1}(t)}\nonumber\\
& &\qquad\quad
+\inner{Q_{1}(t)}{v(t)^{\top}\sigma_{xu}(t)y_{1}(t)}
-\inner{f_{xu}(t)y_{1}(t)}{v(t)}
+\inner{b_{u}(t)^{\top}P_{2}(t)y_{1}(t)}{v(t)}\nonumber\\
& &\qquad\quad
+\inner{\sigma_{u}(t)^{\top}P_{2}(t)\sigma_{x}(t)y_{1}(t)}{v(t)}
+\inner{\sigma_{u}(t)^{\top}Q_{2}(t)y_{1}(t)}{v(t)}\Big)
\Big]dt\\
&=&\inner{P_{1}(0)}{\varpi_{0}}+\frac{1}{2}\inner{P_{2}(0)\nu_{0}}{\nu_{0}}+
\frac{1}{2}\me\int_{0}^{T}\Big(
2\inner{H_{u}(t)}{h(t)}\nonumber\\
&&+\inner{H_{uu}(t)v(t)}{v(t)}
+\inner{P_{2}(t)\sigma_{u}(t)v(t)}{\sigma_{u}(t)v(t)}
+2\inner{\ms(t)y_{1}(t)}{v(t)}\Big)dt.
\end{eqnarray*}
Then, letting $v(\cdot)=h(\cdot)=0$ we obtain (\ref{second order
trans})    and letting $\nu_{0}=\varpi_{0}=0$, we obtain
(\ref{second order integral condition}), for any $(v,h)\in\ma^*$
satisfying $v\in \ups$.

To prove  (\ref{second order integral condition}) for any
$(v,h)\in\ma$ satisfying $v\in \ups$, define
$$E_{i}:=\{(t,\omega)\in [0,T]\times\Omega\ |\ dist(\bar u(t,\omega)+\eps v(t,\omega), U)\le \eps^2\ell(t,\omega),\ \forall\ \eps\in (0,\frac{1}{i}] \}.$$
It can be proved that $E_{i}$ is $\mb([0,T])\otimes\mf$-measurable, the family $\{E_{i}\}_{i=1}^{\infty}$ is nondecreasing and $\bigcup_{i=1}^{\infty}E_{i}$ is of full measure in $[0,T]\times \Omega$. For any $i\in \mn$ and $(v,h)\in \ma$ satisfying $v\in \ups$, define
$$
v^{i}(t,\omega):=\left\{
\begin{array}{l}
v(t,\omega),\ \ \ (t,\omega)\in E_{i},\\[+0.6em]
0,      \qquad\quad \      \text{otherwise},
\end{array}\right.\quad
h^{i}(t,\omega):=\left\{
\begin{array}{l}
h(t,\omega),\ \ \  (t,\omega)\in E_{i},\\[+0.6em]
0,    \qquad\quad \        \text{otherwise}.
\end{array}\right.
$$
Then, $(v^{i}, h^{i})\in\ma^*$ and $v^i\in \ups$. Hence,
\begin{eqnarray}\label{second order condition approx}
&&\me\int_{0}^{T}\Big(
2\inner{H_{u}(t)}{h^{i}(t)}
+\inner{H_{uu}(t)v^{i}(t)}{v^{i}(t)}\nonumber\\
&&\qquad\qquad+\inner{P_{2}(t)\sigma_{u}(t)v^{i}(t)}{\sigma_{u}(t)v^{i}(t)}
+2\inner{\ms(t)y_{1}^{i}(t)}{v^{i}(t)}\Big)dt\le 0,
\end{eqnarray}
where $y_{1}^{i}$ is the solution to the first order variational
equation   (\ref{first vari equ}) with $v$ replaced by $v^{i}$.
Since $v^{i}\to v $, $h^{i}\to h$ in
$L_{\mmf}^{4}(\Omega;\!L^{4}(0,T;\mrm))$ as $i\to\infty$, we have
$y^{i}_{1}\to y_{1}$ in $L^{4}_{\mmf}(\Omega;\!C([0,T];\mrn))$.
Passing to the limit in inequality (\ref{second order condition
approx}), we finally obtain (\ref{second order integral
condition}). This completes the proof of Theorem \ref{TH second
order integral condition}.
\end{proof}

In what follows, we shall give a consequence of Theorem \ref{TH second
order integral condition} for the case when $U$ is represented by finitely many mixed constraints, i.e.,
\begin{equation*}
U= \big \{u \in \mr^m\,\big|\, \varphi_{i}(u)=0,\; \forall\, i=1,...,p, \;  \psi_j(u) \le 0, \;\forall\, j=1,...,r\big\},
\end{equation*}
where    $\varphi_1,...,\varphi_p \colon \mr^n \to \mr$ and $\psi_1,\dots, \psi_r\colon \mr^n \to \mr$ (for some $p,r\in \mn$) are  twice continuously  differentiable functions and for any $u\in U$,
\begin{equation}\label{strong LICQ}
\{ \nabla \varphi_1(u),\cdots,\nabla \varphi_p(u)\}\bigcup \{ \nabla \psi_j(u)\,|\, j \in I(u)\} \;\;\mbox{\rm are linearly independent}.
\end{equation}
Moreover, there exist two constants $L \geq 0$ and $ \rho >0$ such that for every $u \in U$,
$$|\varphi_i ''(u)|\le L, \quad i=1,...,p,$$ 
$$|\psi_j ''(u)|\le L,\quad j \in I(u), $$
\begin{equation}\label{upper bound of inverse of Au}
\rho B_{Im (\Gamma_{u})}\subset \Gamma_{u}B_{\mr^{p+k}},
\end{equation}
where $I(u)$ is the set of all active indices at $u$, $\Gamma_{u}:=(\nabla \varphi_1(u),..., \nabla \varphi_p(u),\nabla \psi_{i_1}(u),...,\psi_{i_k}(u) )$ with 
 $i_{1},..., i_{k}\in I(u)$ being all active indices for some $k\le r$, and $B_{Im (\Gamma_{u})}$ and $B_{\mr^{p+k}}$ are respectively the unit balls in  the image space of $\Gamma_{u}$ and $\mr^{p+k}$.

We observe that (\ref{strong LICQ}) implies (\ref{upper bound of inverse of Au})  with a $\rho$  depending on $u$.  In the above we required $\rho$ to be independent of $u$ to obtain the following result.

\begin{corollary}\label{corollary for equ and inequ constr} Let  $U$ be as above,  (C2)--(C3) hold and   $(\bar x, \bar u, \bar x_{0})$ be a local minimizer for the problem (\ref{minimum J}) with $\bar u\in L^{4}_{\mmf}(\Omega;L^{4}$ $(0,$$T;\mrm))$.
Then there exist $\mu_{i}(\cdot)\in L^{2}_{\mmf}(\Omega;L^{2}(0,T;\mr))$, $i=1,...,p$ and $\lambda_{j}(\cdot)\in L^{2}_{\mmf}(\Omega;L^{2}(0,T;$ $\mr_{+}))$, $j=1,...,r$ such that for any $v(\cdot)\in \Upsilon_{\bar u}\cap L^{4}_{\mmf}(\Omega;L^{4}(0,T;\mrm))$ satisfying $v(t,\omega)\in \T_{U}(\bar u(t,\omega))$, a.e. $(t,\omega)\in [0,T]\times\Omega$ and  the corresponding solution $y_{1}$ of equation (\ref{first vari equ}) we have
\begin{eqnarray}\label{2nd order condition for equ and inequ}
&&\me\int_{0}^{T}\Big(
\inner{H_{uu}(t)v(t)}{v(t)}+\inner{P_{2}(t)\sigma_{u}(t)v(t)}{\sigma_{u}(t)v(t)}
+2\inner{\ms(t)y_{1}(t)}{v(t)}\nonumber\\
&&\qquad\quad
-\sum_{i=1}^p\mu_{i}(t)\inner{ \varphi_{i} '' (\bar u(t))v(t)}{v(t)}-\!\!\!\sum_{j \in I_{v}(\bar u(t))}\!\!\!\lambda_{j}(t)\inner{ \psi_{j} '' (\bar u(t))v(t)}{v(t)}
\Big)dt\le 0,\qquad
\end{eqnarray}
where
$$I_{v}(\bar u(t,\omega))=\{j\in I(\bar u(t,\omega))\ |\ \inner{\nabla \psi_j (\bar u(t,\omega))}{v(t,\omega)}=0\}.$$
\end{corollary}
\begin{proof}
The proof of this result is similar to that of \cite[Theorem 3]{Frankowska15}.
Obviously, condition (\ref{strong LICQ}) implies the
Mangasarian-Fromowitz constraint qualification.
By Example \ref{example for 1st and 2nd tangent set}, for any $(t,\omega)$,
$$\N_{U}(\bar u(t,\omega))=\sum_{i=1}^p \mr \nabla \varphi_i(\bar u(t,\omega)) + \sum_{j \in I(\bar u(t,\omega))}\mr_+ \nabla \psi_j (\bar u(t,\omega)).$$

Then, by the first order condition (\ref{first order pointwise condition}), we have
$$H_{u}(t,\omega)\in \sum_{i=1}^p \mr \nabla \varphi_i(\bar u(t,\omega)) +\!\!\! \sum_{j \in I(\bar u(t,\omega))}\!\!\!\mr_+ \nabla \psi_j (\bar u(t,\omega)),\; a.e.\; (t,\omega)\in [0,T]\times\Omega.$$
Define
$$\Gamma(t,\omega)=\{(\mu_{1},...,\mu_{p},\lambda_{1},...,\lambda_{r})\in \mr^{p+r}\ |\,
\lambda_{j}\ge 0,\; j=1,...,r,\; \lambda_{j}\psi_{j}(\bar u(t,\omega))=0\},$$
and
$$G(t,\omega, \Gamma(t,\omega))=\sum_{i=1}^p \mr \nabla \varphi_i(\bar u(t,\omega)) + \sum_{j \in I(\bar u(t,\omega))}\mr_+ \nabla \psi_j (\bar u(t,\omega)).$$
By Filippov's theorem (see \cite[Theorem 8.2.10]{Aubin90}), there exists a $\mg^*$-measurable selection
$$\gamma^*(t,\omega)=(\mu_{1}^*(t,\omega),...,\mu_{p}^*(t,\omega),
\lambda_{1}^*(t,\omega),...,\lambda_{r}^*(t,\omega))\in \Gamma(t,\omega),\; a.e.\; (t,\omega)\in [0,T]\times\Omega$$
such that
$$H_{u}(t,\omega)=\sum_{i=1}^p \mu_{i}^*(t,\omega)\nabla \varphi_i(\bar u(t,\omega)) +\!\!\! \sum_{j \in I(\bar u(t,\omega))}\!\!\!\lambda_{j}^*(t,\omega) \nabla \psi_j (\bar u(t,\omega)),\; a.e.\; (t,\omega)\in [0,T]\times\Omega$$
where $\mg^*$ is the completion of $\mg$ and $\mg$ is defined by (\ref{adapted sigma field}).  By assumption (\ref{strong LICQ}) the process $\gamma^* (\cdot)$ is uniquely determined (up to a set of measure zero). Since $\mr^m$ is separable, there exists a $\mg$-measurable modification of $\gamma^*(\cdot)$:
$$\gamma (\cdot)=(\mu_{1} (\cdot),...,\mu_{p} (\cdot),\lambda_{1} (\cdot),...,\lambda_{r} (\cdot)).$$
By Lemma \ref{lemma adapted sigma field}, $\gamma (\cdot)$ is $\mb([0,T])\otimes\mf$-measurable and $\mmf$-adapted and
\begin{equation}\label{first order condition for equ and inequ constr}
H_{u}(t,\omega)=\sum_{i=1}^p \mu_{i}(t,\omega)\nabla \varphi_i(\bar u(t,\omega)) +\!\!\! \sum_{j \in I(\bar u(t,\omega))}\!\!\!\lambda_{j}(t,\omega) \nabla \psi_j (\bar u(t,\omega)),\; a.e.\; (t,\omega)\in [0,T]\times\Omega
\end{equation}

By \cite[Theorem 2.1]{Frankowska1990} and assumption (\ref{upper bound of inverse of Au}), for a.e.  $(t,\omega)\in [0,T]\times\Omega$
\begin{equation}\label{estimates for Lagrange mulitplier}
|\mu_{i}(t,\omega)|\le \frac{1}{\rho} |H_{u}(t,\omega)|,\; \forall\, i=1,...,p,\; \lambda_{j}(t,\omega) \le\frac{1}{\rho} |H_{u}(t,\omega)|,\; \forall\, j\in I(\bar u(t,\omega)).
\end{equation}
On the other hand, when $j\notin I(\bar u(t,\omega))$, $\lambda_{j}(t,\omega)=0$  and therefore also $\lambda_{j}(t,\omega)  \le\frac{1}{\rho} |H_{u}(t,\omega)|$.
Since $H_{u}(\cdot)\in L^{2}_{\mmf}(\Omega;L^{2}(0,T;\mrm))$, we deduce that  $\mu_{i}(\cdot)\in L^{2}_{\mmf}(\Omega;L^{2}(0,T;\mr))$, $i=1,...,p$, and, $\lambda_{j}(\cdot)\in L^{2}_{\mmf}(\Omega;L^{2}(0,T;\mr_{+}))$, $j=1,...,r$.

Let $v(\cdot)\in \Upsilon_{\bar u}\cap L^{4}_{\mmf}(\Omega;L^{4}(0,T;\mrm))$ satisfy $v(t,\omega)\in \T_{U}(\bar u(t,\omega))$, a.e. $(t,\omega)\in [0,T]\times\Omega$. Then
\begin{equation}\label{degeneration direction}
\inner{H_{u}(t,\omega)}{v(t,\omega)}=0,\quad a.e.\; (t,\omega)\in [0,T]\times\Omega.
\end{equation}
Combining (\ref{degeneration direction}) with
(\ref{first order condition for equ and inequ constr}), one has, for a.e. $(t,\omega)\in [0,T]\times\Omega$,
$$\sum_{j \in I(\bar u(t,\omega))}\lambda_{j}(t,\omega)\inner{ \nabla \psi_j (\bar u(t,\omega))}{v(t,\omega)}=0.$$
Therefore, for a.e. $(t,\omega)\in [0,T]\times\Omega$ and for any $j\notin I_{v}(\bar u(t,\omega))$, $\lambda_{j}(t,\omega)=0$. Consequently,
\begin{equation}\label{refined 1order condition for equ and inequ}
H_{u}(t,\omega)=\sum_{i=1}^p \mu_{i}(t,\omega)\nabla \varphi_i(\bar u(t,\omega)) +\!\!\! \sum_{j \in I_{v}(\bar u(t,\omega))}\!\!\!\lambda_{j}(t,\omega) \nabla \psi_j (\bar u(t,\omega)),\; a.e.\; (t,\omega)\in [0,T]\times\Omega.
\end{equation}

On the other hand, for any $(t,\omega)\in [0,T]\times\Omega$, by Example \ref{example for 1st and 2nd tangent set},
\begin{equation}\label{second order variational vector}
\begin{array}{ll}\displaystyle\quad  \emptyset \neq \TT_{U}(\bar u(t,\omega),v(t,\omega))\\[+0.5em]
\displaystyle= \left\{h\in \mrm\;\left|\;
 \langle \nabla \varphi_i(\bar u(t,\omega)),
h\rangle+\frac{1}{2}\langle \varphi_i '' (\bar u(t,\omega))v(t,\omega),v(t,\omega)\rangle = 0, \;\forall \, i=1,\cdots,p
, \right.\right.\\
\displaystyle\qquad \hbox{and }\left. \langle \nabla \psi_j(\bar u(t,\omega)),
h\rangle+\frac{1}{2}\langle \psi_j '' (\bar u(t,\omega))v(t,\omega),v(t,\omega)\rangle\le0, \;
\forall \, j \in I_{v}(\bar u(t,\omega))\right \}.
\end{array}
\end{equation}
By the maximum condition (\ref{maximum principle}), it follows that, for any $h\in \TT_{U}(\bar u(t,\omega),v(t,\omega))$ and a.e. $(t,\omega)\in [0,T]\times\Omega$,
\begin{equation}
\inner{ H_{u}(t,\omega)}{h}
+\frac{1}{2}\inner{\big(H_{uu}(t,\omega)+
\sigma_{u}(t,\omega)^{\top} {P}_{2}(t,\omega)
\sigma_{u}(t,\omega)\big) v(t,\omega)}{v(t,\omega)}\le 0,
\end{equation}
which implies that
$$\sup_{h\in \TT_{U}(\bar u(t,\omega),v(t,\omega))}\inner{H_{u}(t,\omega)}{h}<\infty,\;a.e.\; (t,\omega)\in [0,T]\times\Omega.$$

By (\ref{second order variational vector}), $\TT_{U}(\bar u(t,\omega),v(t,\omega))$ is a polyhedral set, cf. \cite[p. 43]{RW}. By \cite[Corollary 3.53]{RW} the supremum in the above is attained.

By \cite[Theorems 8.2.11 and 8.2.9]{Aubin90} (making a completion argumentation if necessary), there exists a $\mb([0,T])\otimes\mf$-measurable and $\mmf$-adapted  process $\tilde h(\cdot)$ such that  $\tilde h(t,\omega)  \in \TT_{U}(\bar u(t,\omega),v(t,\omega))$  a.e. in $0,T]\times\Omega$  and
$$\inner{H_{u}(t,\omega)}{\tilde h(t,\omega)}=\sup_{h\in \TT_{U}(\bar u(t,\omega),v(t,\omega))}\inner{H_{u}(t,\omega)}{h} ,\;a.e.\; (t,\omega)\in [0,T]\times\Omega.$$

Then, for a.e. $(t,\omega)\in [0,T]\times\Omega$
\begin{equation}\label{second adjacent vect for equ}
\mu_{i}(t,\omega)\inner{ \nabla \varphi_{i}(\bar u(t,\omega))}{\tilde h(t,\omega)}= -\frac{\mu_{i}(t,\omega)}{2}\inner{ \varphi_{i} '' (\bar u(t,\omega))v(t,\omega)}{v(t,\omega)},\; \forall\; i=1,...,p,
\end{equation}
and,
$$\lambda_{j}(t,\omega)\!\inner{ \nabla \psi_{j}(\bar u(t,\omega))}{\tilde h(t,\omega)}\!\le\! - \frac{\lambda_{j}(t,\omega)}{2}\inner{ \psi_{j} '' (\bar u(t,\omega))v(t,\omega)}{v(t,\omega)},\; \forall\; j\in I_{v}(\bar u(t,\omega)).$$
Applying the same argument as at the end of Example 2.1 we show,
using  (\ref{refined 1order condition for equ and inequ}),  that
\begin{equation}\label{second adjacent vect for inequ}
\lambda_{j}(t,\omega)\!\inner{ \nabla \psi_{j}(\bar u(t,\omega))}{\tilde h(t,\omega)}\!= -  \frac{\lambda_{j}(t,\omega)}{2}\inner{ \psi_{j} '' (\bar u(t,\omega))v(t,\omega)}{v(t,\omega)},\; \forall\; j\in I_{v}(\bar u(t,\omega)),
\end{equation}
Combining (\ref{refined 1order condition for equ and inequ}), (\ref{second adjacent vect for equ}) with (\ref{second adjacent vect for inequ}), one obtains that, for a.e. $(t,\omega)\in [0,T]\times\Omega$,
\begin{equation}\label{repersentaion of second order vari}
\begin{array}{ll}\displaystyle\quad  \inner{H_{u}(t,\omega)}{\tilde h(t,\omega)}
=-\frac{1}{2}\sum_{i=1}^p\mu_{i}(t,\omega)\inner{ \varphi_{i} '' (\bar u(t,\omega))v(t,\omega)}{v(t,\omega)}\\
\displaystyle\qquad\qquad\qquad\qquad\quad\;\;\;-\frac{1}{2}\sum_{j \in I_{v}(\bar u(t,\omega))}\!\!\!\lambda_{j}(t,\omega)\inner{ \psi_{j} '' (\bar u(t,\omega))v(t,\omega)}{v(t,\omega)}.
\end{array}
\end{equation}

Now, for any $i\in \mn$, define
$$
v^{i}(t,\omega):=\left\{
\begin{array}{l}
v(t,\omega),\ \ \    \text{if }  |\tilde h(t,\omega)|\le i,\\[+0.6em]
0,      \qquad\quad \      \text{otherwise},
\end{array}\right.\quad
h^{i}(t,\omega):=\left\{
\begin{array}{l}
\tilde h(t,\omega),\ \ \    \text{if }   |\tilde h(t,\omega)|\le i,\\[+0.6em]
0,    \qquad\quad \        \text{otherwise},
\end{array}\right.
$$
we have $(v^{i}(\cdot),h^{i}(\cdot))\in \ma$ and $v^{i}(\cdot)\in \Upsilon_{\bar u}$. Let $y_{1}^{i}$ be the solution to the first order variational equation (\ref{first vari equ}) corresponding to $v^{i}(\cdot)$, then by (\ref{repersentaion of second order vari}) and condition (\ref{second order integral condition}), we obtain that
\begin{eqnarray}\label{prox 2nd order condition for equ and inequ}
&&\me\int_{0}^{T}\Big(
\inner{H_{uu}(t)v^{i}(t)}{v^{i}(t)}+\inner{P_{2}(t)\sigma_{u}(t)v^{i}(t)}{\sigma_{u}(t)v^{i}(t)}
+2\inner{\ms(t)y_{1}^{i}(t)}{v^{i}(t)}\nonumber\\
&&\qquad
- \sum_{i=1}^p\mu_{i}(t)\inner{ \varphi_{i} '' (\bar u(t))v^{i}(t)}{v^{i}(t)}- \!\!\!\sum_{j \in I_{v}(\bar u(t))}\!\!\!\lambda_{j}(t)\inner{ \psi_{j} '' (\bar u(t))v^{i}(t)}{v^{i}(t)}
\Big)dt\le 0.\qquad\;\;
\end{eqnarray}
Passing to the limit in inequality (\ref{prox 2nd order condition for equ and inequ}), we finally obtain condition (\ref{2nd order condition for equ and inequ}).
This completes the proof of Corollary \ref{corollary for equ and inequ constr}.
\end{proof}

In \cite{Bonnans12}, in the special case of $K=\{x_{0}\}$, the
authors obtained the following integral-type first  and second
order necessary conditions for stochastic optimal controls:
\begin{theorem}
Let (C2)--(C3) hold.
If $U$ is  closed and convex and $\bar{u}$ is an optimal control, then
\begin{equation}\label{bonnans 1order condition}
\me\int^{T}_{0}\inner{H_{u}(t)}{v(t)}dt\le0,\qquad \forall
\; v\in cl_{2,2}\big(\mathcal{R}_{\mmu}(\bar{u})\cap L^{4}_{\mmf}(\Omega;L^{4}(0,T;\mrm))\big).
\end{equation}
Furthermore, for any  $v(\cdot)\in cl_{4,4}\big(\mathcal{R}_{\mmu}(\bar{u})\cap L^{\infty}([0,T]\times \Omega;\mrm)\cap \ups\big)$ the following second order necessary condition holds:
\begin{eqnarray}\label{bonnans 2order condition}
& &\me\int^{T}_{0}\Big(\inner{H_{xx}(t)y_{1}(t)}{y_{1}(t)}
+2\inner{H_{xu}(t)y_{1}(t)}{v(t)}\nonumber\\
& &\quad\quad +\inner{H_{uu}(t)v(t)}{v(t)}\Big)dt
 +\me\inner{g_{xx}(\bar{x}(T))y_{1}(T)}{y_{1}(T)}\le0, ~~\qquad
\end{eqnarray}
where,
$$\mathcal{R}_{\mmu}(\bar{u})
:=\big\{\alpha u-\alpha\bar{u}\ \big|\
u\in \mmu, \alpha\ge0\big\},$$
and $cl_{2,2}(A)$ and $cl_{4,4}(A)$ are respectively the closures of a set $A$ under the norms $\|\cdot\|_{2,2}$ and $\|\cdot\|_{4,4}$.
\end{theorem}

\begin{remark}
There are three main differences between  (\ref{second order integral condition}) and  (\ref{bonnans 2order condition}): First, the control region is allowed to be nonconvex in (\ref{second order integral condition}).
Second, the solutions to  two adjoint equations (\ref{first ajoint equ}) and (\ref{second ajoint equ}) are used  in  (\ref{second order integral condition}), and consequently, the second order term involving  $y_{1}$ (the solution to the first order variational equation (\ref{first vari equ})) is absent  in this condition.Third, the condition (\ref{second order integral condition}) contains the second order adjacent vector $h$, while in (\ref{bonnans 2order condition}) it is equal to zero, cf. Remark \ref{rem}. Our condition (\ref{second order integral condition}) is more effective in distinguishing optimal controls from other admissible controls than (\ref{bonnans 2order condition}), even if the diffusion term $\sigma=0$, see \cite{Hoehener12}. See also the examples (especially Example \ref{0712example1}) that we shall give below.
\end{remark}

\begin{example}\label{comparion example with Bonnans}
Let $U$ be equal to the intersection of two closed balls in $\mr^2$ of radii $1$ and centers at respectively $(1,0)$  and $(\frac{-1}{\sqrt{2}}, \frac{1}{\sqrt{2}})$, $T=1$, $A\in \mr^{2\times 2}$, $F=(F^{1},F^{2}):\mr^2\to \mr_+ \times \mr$ be a given function satisfying $F(0)=0$, $F_{x}(0)=0$, $F_{xx}(0)=0$, and for some $L>0$,
$$ |F_{x}(x)|+|F_{xx}(x)|\le L,\quad \forall\, x\in \mr^2.$$

Consider the stochastic control system
\begin{equation*}
\left\{
\begin{array}{l}
dx(t)=\big[F(x(t))+u(t)\big]dt+Au(t)dW(t),\ \ \ t\in[0,1],\\
x(0)=0,
\end{array}\right.
\end{equation*}
with the cost functional
\begin{equation*}
J(u(\cdot))=\me\big[x_{1}(1)- \cos (x_{2}(1))^2\big].
\end{equation*}

For this optimal control problem, the Hamiltonian is defined as
$$H(t,x,u, p,q,\omega)
:=\inner{p}{F(x)+u}+\inner{q}{Au},$$
where $(t,x,u,p,q,\omega)\in [0,1]\times\mr^2\times \mr^2\times\mr^2\times\mr^2\times\Omega.$

Define $\bar u(t)\equiv (0,0)$. Then, the corresponding state $\bx(t)\equiv(0,0)$. Since $F^{1}(x)\ge 0$ for any $x\in \mr^2$ and $U \subset \mr_+ \times \mr$,  we deduce that  $\me(x_{1}(1))\ge0$ for any solution $x=(x_1,x_2)$ of the above stochastic system. Therefore $\bar u$ is the global minimizer. Furthermore, the first and the second order adjoint equations are
\begin{equation}\label{first ajoint equ example add}
 \left\{
\begin{array}{l}
dP_{1}(t)=Q_{1}(t)dW(t), \quad  t\in[0,1], \\
P_{1}(1)=(-1,0)
\end{array}\right.
\end{equation}
and
\begin{equation}\label{second ajoint equ example add}
 \left\{
\begin{array}{l}
dP_{2}(t)=Q_{2}(t)dW(t), \quad  t\in[0,1], \\
P_{2}(1)=0.
\end{array}\right.
\end{equation}

It is easy to see that the solution to equations (\ref{first ajoint equ example add}) and (\ref{second ajoint equ example add}) are $P_{1}(t)\equiv (-1,0)$, $Q_{1}(t)\equiv0$ and $(P_{2}(t),Q_{2}(t))\equiv (0,0)$, respectively. Then, 
$$H_{u}(t)=P_{1}(t)+A^{\top}Q_{1}(t)\equiv (-1,0),\  H_{uu}(t)+ \sigma_{u}^{\top}(t)P_{2}(t)\sigma_{u}(t)\equiv0,\ \mbox{and}\ \ms(t)\equiv0.$$

By the definition of $U$,
$\T_{U}((0,0))$ is the closed convex cone generated by $\{(0,1),(1,1)\}$. Moreover $(\frac12,0) \in \TT_{U}((0,0), (0,1))$.

Then the first order necessary condition
$$\inner{H_{u}(t,\omega)}{v} \le 0, \quad \forall\ v\in\T_{U}((0,0)) $$
(which corresponds to the first condition in (\ref{first order pointwise condition})) is satisfied and
$$H_{xx}(t)\equiv 0, \ \ H_{xu}(t)\equiv 0, \ \ H_{uu}(t)\equiv 0, \ \ \mbox{and}\ \ \ g_{xx}(\bx(1))\equiv 0.$$
Therefore, the second order necessary condition (\ref{bonnans 2order condition}) is satisfied trivially in this case and does not contain any additional information with respect to the first order necessary condition (\ref{bonnans 1order condition}).

Comparatively, our second order necessary condition (\ref{second order integral condition}) provides more information about the control $\bu$. For example, let $\tilde v(t)\equiv(0,1)$ and $\tilde h(t)\equiv (\frac{1}{2},0)$. Obviously $\tilde v\in \Upsilon_{\bu}$,  $(\tilde v, \tilde h)\in \ma$, and condition (\ref{second order integral condition}) becomes
$$2\me \int_{0}^{1}\inner{H_{u}(t)}{\tilde h(t)}dt=-1\le 0.$$
Noting that $(\frac{1}{2},0)\notin \T_{U}((0,0))$, the last inequality is different from the first order necessary condition (\ref{first order integraltype condition}) and from the second order necessary condition (\ref{bonnans 2order condition}).
\end{example}

\begin{example}\label{0712example1}
Let $n=m=2$, $T=1$, and
 $$U=\{(u_{1},u_{2})\in \mr^2\ |\ |u_{1}+1|^2+|u_{2}|^2 = 1\}\cup \{(u_{1},u_{2})\in \mr^2\ |\ |u_{1}-1|^2+|u_{2}|^2= 1\}.
  $$
Clearly, this $U$ is neither a finite set nor convex in
$\mr^2$. One can easily check that
$$
\T_{U}((0,0))=\{0\}\times \mr,\qquad \TT_{U}((0,0),(0,1))\ni (
\frac{1}{2}, 0) .$$ Consider the control system
\begin{equation}\label{controlsys example1-1}
\left\{
\begin{array}{l}
dx_{1}(t)=(x_{2}(t)-\frac{1}{2})dt+dW(t),\ \ \ t\in[0,1],\\
dx_{2}(t)=u_{1}(t)dt+|u_{2}(t)|^4dW(t),\ \ \ t\in[0,1],\\
x_{1}(0)=0, x_{2}(0)=0
\end{array}\right.
\end{equation}
with the cost functional
\begin{equation}\label{20160712e1z}
J(u)= \me \Big[\frac{1}{2} |x_{1}(1)-W(1)|^2 +\int_0^1
|u_{2}(t)|^4 dt \Big] .
\end{equation}
 Obviously, the only difference between
(\ref{controlsys example1}) and (\ref{controlsys example1-1}) is
that the coefficient ``$u_{2}(t)$" in the first system is replaced
by ``$\,|u_{2}(t)|^4$" in the second one and, since $U$ is a bounded set, the assumptions (C2)--(C3) are fulfilled.

The Hamiltonian of this optimal control problem is given by
 $$
H(t,(x_{1},x_{2}),(u_{1},u_{2}), (p_{1}^{1},p_{1}^{2}),(q_{1}^{1},q_{1}^{2}),\omega)
=p_{1}^{1}(x_{2}-\frac{1}{2})+p_{1}^{2}u_{1}+q_{1}^{1}+q_{1}^{2}|u_{2}|^4-|u_{2}|^4,
 $$
for all $(t,(x_{1},x_{2}),(u_{1},u_{2}),
(p_{1}^{1},p_{1}^{2}),(q_{1}^{1},q_{1}^{2}),\omega)\in
[0,1]\times\mr^2\times \mr^2\times\mr^2\times\mr^2\times\Omega$. In
what follows, we show that the admissible control
$(u_{1}(t),u_{2}(t))\equiv(0,0)$ is not locally optimal.

The corresponding solution to the control system (\ref{controlsys
example1-1}) is still given by  (\ref{nding solution1}), and the
first  order adjoint equation is the same as in (\ref{first order
adjequ example1}). Therefore $(P_{1}^{1}(t),Q_{1}^{1}(t)) $ and $
(P_{1}^{2}(t),Q_{1}^{2}(t))$ are as in  (\ref{20160712e3}).

For the present problem,
\begin{equation}\label{20160712e0}
H_{u}(t)= (P_{1}^{2}(t), 4Q_1^2(u_2(t))^3-4(u_2(t))^3)=(\frac{1-t}{2},0).
\end{equation}
Hence, the first order condition in (\ref{first order pointwise condition}),
$$\inner{H_{u}(t)}{v}=P_{1}^{2}(t)v_{1}+4(Q_1^2-1)(u_2(t))^3v_2= 0,\quad \forall \ v=(v_{1},v_{2})\in \T_{U}((0,0))$$
is trivially satisfied, and therefore we need to check the second order condition (\ref{second order integral condition}). For this, we observe that
 \begin{equation}\label{20160712e11}
H_{uu}(t)= \left[ \begin{array}{cc}
          0 & 0 \\
          0 & 0 \\
\end{array} \right],\quad b_x(t)=\left[ \begin{array}{cc}
          0 & 1 \\
          0 & 0 \\
\end{array} \right],\quad b_u(t)=\left[ \begin{array}{cc}
          0 & 0 \\
          1 & 0 \\
\end{array} \right], \quad\sigma_x(t)=\sigma_u(t)=\left[ \begin{array}{cc}
          0 & 0 \\
          0 & 0 \\
\end{array} \right].
\end{equation}
We now choose a direction $v=(v_1,v_2)=(0,1)$ and $\nu_{0}=(0,0)$. Then, the first order variational equation (\ref{first vari equ}) becomes
\begin{equation}\label{first-vari-equ}
\left\{
\begin{array}{l}
\frac{dy_{1}(t)}{dt}=b_{x}(t) y_{1}(t),\quad  t\in[0,1], \\
y_{1}(0)=(0,0),
\end{array}\right.
\end{equation}
and hence $y_1(t)\equiv (0,0)$. This, combined with (\ref{20160712e11}), shows that  the second condition in (\ref{second order integral condition}) is specified as
\begin{eqnarray}\label{secontegral condition}
\me\int_{0}^{1}\inner{H_{u}(t)}{h}
dt\le 0,\qquad\forall\;h\in \TT_{U}((0,0),(0,1)).
\end{eqnarray}
We now choose $h=(\frac{1}{2}, 0)$ in (\ref{secontegral condition}). By (\ref{20160712e0}), we obtain that
$$
\me\int_{0}^{1}\inner{H_{u}(t)}{h}
dt=\frac{1}{8}>0,
$$
which is a contradiction. Therefore, $(u_{1}(t),u_{2}(t))\equiv(0,0)$ is not locally optimal.
\end{example}

\subsection{Pointwise second order necessary conditions}

In this subsection, under some further assumptions, we shall
deduce from  the integral-type second order necessary condition
(\ref{second order integral condition}) a pointwise one. First, we
introduce the following notion.
\begin{definition}\label{20160710e2}
We call $\tilde{u}\in \mmu$ partially singular in the classical
sense if $\tilde{u}$ satisfies
\begin{equation}\label{singularcontrol Hu Huu euql zero}
\left\{
\begin{array}{l}
\widetilde H_{u}(t)=0, \quad a.e.\ t\in [0,T], \  a.s.,\\[+0.6em]
\inner{\big(\widetilde H_{uu}(t)+\tilde \sigma_{u}(t)^{\top}
\widetilde{P}_{2}(t)\tilde \sigma_{u}(t)\big) v}{v}=0, \quad \forall\
v\in \T_{U}(\tilde{u}(t)),\ a.e.\ t\in [0,T], \  a.s.
\end{array}\right.
\end{equation}
where $\tilde{x}$ is the state corresponding to $\tilde{u}$,
$\widetilde H_{u}(t)=H_{u}(t,\tilde{x}(t), \tilde{u}(t)
,\widetilde{P}_{1}(t),\widetilde{Q}_{1}(t))$, and similarly for $\widetilde
H_{uu}(t)$ and  $\tilde \sigma_{u}(t)$.
$(\widetilde{P}_{1},\widetilde{Q}_{1})$ and
$(\widetilde{P}_{2},\widetilde{Q}_{2})$ are the adjoint processes given
respectively by (\ref{first ajoint equ}) and (\ref{second ajoint
equ}) with $(\bar x,\bar u, \bx_{0})$ replaced by
$(\tilde{x},\tilde{u},\tilde{x}_{0})$. When $(\bx,\bu,\bx_{0})$ is
a   local minimizer for the problem (\ref{minimum J}) and $\bu$ is
singular, we call $(\bx,\bu,\bx_{0})$ a singular  local minimizer
(for the problem (\ref{minimum J})).
\end{definition}

\begin{remark}
The definition of the singular control in (\ref{singularcontrol Hu
Huu euql zero}) is much more  general than that in
\cite[Definition 3.3]{zhangH14a}.
 More precisely, by the maximality condition (\ref{maximum principle}), if the control $\tilde u$ is optimal, the first and second necessary
 conditions
  in optimization theory immediately imply that, for a.e. $(t,\omega)\in [0,T]\times\Omega$,
$$\inner{\widetilde H_{u}(t,\omega)}{v}\le 0, \quad \forall\; v\in \T_{U}(\tilde{u}(t,\omega)).$$
Further, if $\inner{\widetilde H_{u}(t,\omega)}{v_0}=0$ for some $v_0 \in
\T_{U}(\tilde{u}(t,\omega)) $, then for any $h\in
\TT_{U}(\tilde{u}(t,\omega),v_0)$,
\begin{equation}\label{Legendre condition}
\inner{\widetilde H_{u}(t,\omega)}{h}
+\frac{1}{2}\inner{\big(\widetilde H_{uu}(t,\omega)+\tilde
\sigma_{u}(t,\omega)^{\top} \widetilde{P}_{2}(t,\omega)\tilde
\sigma_{u}(t,\omega)\big) v_0}{v_0}\le 0.
\end{equation}
Both Definition
\ref{20160710e2} and \cite[Definition 3.3]{zhangH14a} imply that
the corresponding singular controls satisfy the above first and
second order necessary condition trivially, but in Definition
\ref{20160710e2}, $\widetilde H_{uu}(t)+\tilde \sigma_{u}(t)^{\top}
\widetilde{P}_{2}(t)\tilde \sigma_{u}(t)$ is only assumed to be
degenerated, for a.e. $[0,T]\times\Omega$, in the directions from
$\T_{U}(\tilde{u}(t))$. We shall see in Example \ref{example2}
below that for partially singular controls,  $\widetilde
H_{uu}(t)+\tilde \sigma_{u}(t)^{\top}\widetilde{P}_{2}(t)\tilde
\sigma_{u}(t)$ may be different from 0 on a subset of
$[0,T]\times\Omega$ having positive measure.
\end{remark}

By Theorem \ref{TH second order integral condition}, it is easy to
verify the following second order integral-type necessary
condition for the problem (\ref{minimum J}).

\begin{theorem}\label{TH 2order integral condition for sing}
Let (C1)--(C3) hold. If $(\bx,\bu,\bx_{0})$ is a singular  local
minimizer for the problem (\ref{minimum J}) and $\bu\in
L_{\mmf}^{4}(\Omega;L^{4}(0,T;\mrm))$, then
\begin{equation}\label{2order integral condition for sing}
\me\int^{T}_{0}\inner{\ms(t)y_{1}(t)}{v(t)}dt\le 0, \qquad
\forall\; v\in\ma^{1}.
\end{equation}
\end{theorem}

As underlined in \cite{zhangH14a}, there are some essential
difficulties to deduce from the above integral type second order
necessary condition  a pointwise one. The main reason for it  is
that the spike variations have to be used to get the  pointwise
second order necessary condition from  (\ref{2order integral
condition for sing}). Substituting the explicit expression for
$y_{1}$ into (\ref{2order integral condition for sing}), the
It\^{o} integral will appear in this condition. Thus there will be
a ``bad" term making impossible  using  the Lebesgue
differentiation theorem   to derive the pointwise condition (see
Subsection 3.2 in \cite{zhangH14a} for more details). However,
when $\ms$ and $v$ are regular enough, a  method similar to the
one  proposed in \cite{zhangH14a} can be used to establish the
following pointwise second-order necessary condition for
stochastic singular optimal controls for the problem (\ref{minimum
J}).

\begin{theorem}\label{TH partial pointwise 2order condition}
Let (C1)--(C3) hold. If  $(\bx,\bu,\bx_{0})$ is a singular  local
minimizer for the problem (\ref{minimum J}), $\bu\in
L_{\mmf}^{4}(\Omega;L^{4}(0,T;\mrm))$ and $\ms\in
\ml_{2,\mmf}^{1,2}(\mr^{m\times n})\cap
L^{\infty}([0,T]\times\Omega;\mr^{m\times n})$, then in addition
to the second order transversality condition (\ref{second order
trans}), for any $v\in\ml_{2,\mmf}^{1,2}(\mrm)\cap
L^{\infty}([0,T]\times\Omega;\mrm)\cap \ma^{1}$, the following
pointwise second order necessary condition holds:
\begin{eqnarray}\label{partial pointwise 2order condition}
& &\inner{\ms(\t)b_{u}(\t)v(\t)}{v(\t)}
 + \inner{\nabla \ms(\t)\sigma_{u}(\t)v(\t)}
{v(\t)}\\
& &\quad+ \inner{\ms(\t)\sigma_{u}(\t)v(\t)} {\nabla v(\t)} \le 0,
\quad a.e.\ \t\in[0,T], \ a.s.\nonumber
\end{eqnarray}

\begin{proof}
The proof is similar to the one of \cite[Theorem 3.13]{zhangH14a}.
Let $\t\in[0,T)$, $\theta\in (0,T-\t)$, $E_{\theta}=[\t,
\t+\theta)$ and choose $A\in \mf_{\t}$. For any
$v(\cdot)\in\ml_{2,\mmf}^{1,2}(\mrm)\cap
L^{\infty}([0,T]\times\Omega;\mrm)\cap \ma^{1}$, define
$$
v^{\theta,A}(t,\omega)=\left\{
\begin{array}{l}
v(t,\omega), \qquad (t,\omega)\in E_{\theta}\times A,\\
0, \qquad \quad\quad\, (t,\omega)\in \big([0,T]\times \Omega\big) \setminus  \big(E_{\theta}\times A\big).\\
\end{array}\right.
$$
Clearly, $v^{\theta,A}(\cdot)\in \ma^{1}$. Denote by
$y_{1}^{\theta,A}(\cdot)$ the solution to the first order
variational equation (\ref{first vari equ}) with $v(\cdot)$
replaced by $v^{\theta,A}(\cdot)$. By \cite[Theorem 1.6.14,
p.47]{Yong99}, $y_{1}^{\theta,A}(\cdot)$ enjoys an explicit
representation:
\begin{eqnarray}\label{y1(theta A)(t)for v(theta A)(t)}
 y_{1}^{\theta,A}(t)&=&\Phi(t)\int_{0}^{t}\Phi(s)^{-1}
\big(b_{u}(s)-\sigma_{x}(s)\sigma_{u}(s)
\big)v^{\theta,A}(s)ds\nonumber\\
&
&+\Phi(t)\int_{0}^{t}\Phi(s)^{-1}\sigma_{u}(s)v^{\theta,A}(s)dW(s),
\end{eqnarray}
where $\Phi(\cdot)$ solves the following matrix-valued stochastic
differential equation
\begin{equation}\label{Phi}
\left\{
\begin{array}{l}
d\Phi(t)= b_{x}(t)\Phi(t)dt+\sigma_{x}(t)\Phi(t)dW(t),
\qquad \ \ \ t\in[0,T], \qquad \\
\Phi(0)=I,
\end{array}\right.
\end{equation}
and $I$ stands for the identity matrix of dimension $n$.

From Theorem \ref{TH 2order integral condition for sing}, it
follows that
\begin{eqnarray}\label{limit of S(t)y1(t)}
\qquad 0&\ge&\frac{1}{\theta^2}\me\int_{\t}^{\t+\theta}
\inner{\ms(t)y_{1}^{\theta,A}(t)}{v(t)}\chi_{A}dt\nonumber\\
&=&\frac{1}{\theta^2}\me\int_{\t}^{\t+\theta} \inner{
\ms(t)\Phi(t)\int_{\t}^{t}\Phi(s)^{-1}
\big(b_{u}(s)-\sigma_{x}(s)\sigma_{u}(s) \big)v(s)\chi_{A}ds}{
v(t)}\chi_{A}dt\nonumber\\
& &+\frac{1}{\theta^2} \me\int_{\t}^{\t+\theta}
\inner{\ms(t)\Phi(t)\int_{\t}^{t}
\Phi(s)^{-1}\sigma_{u}(s)v(s)\chi_{A}dW(s)}{ v(t)}\chi_{A} dt.
\end{eqnarray}

By the Lebesgue differentiation theorem, it is immediate that for
a.e. $\t\in [0,T)$,
\begin{eqnarray}\label{limit of S(t)y1(theta,A)(t)part1}
\qquad& &\lim_{\theta\to
0^+}\frac{1}{\theta^2}\me\int_{\t}^{\t+\theta} \inner{
\ms(t)\Phi(t)\int_{\t}^{t}\Phi(s)^{-1}
\big(b_{u}(s)-\sigma_{x}(s)\sigma_{u}(s) \big)v(s)\chi_{A}ds}{
v(t)}\chi_{A}dt\nonumber\\
&=&
\frac{1}{2}\me~\Big(\inner{\ms(\t)\big(b_{u}(\t)-\sigma_{x}(\t)\sigma_{u}(\t)
\big)v(\t)} {v(\t)}\chi_{A}\Big).
\end{eqnarray}

On the other hand, by (\ref{Phi})
\begin{eqnarray}\label{limit of S(t)y1(theta,A)(t)part2}
\qquad& &\frac{1}{\theta^2} \me\int_{\t}^{\t+\theta}
\inner{\ms(t)\Phi(t)\int_{\t}^{t}
\Phi(s)^{-1}\sigma_{u}(s)v(s)\chi_{A}dW(s)}{
v(t)}\chi_{A} dt\\
&=&\frac{1}{\theta^2} \me\int_{\t}^{\t+\theta}
\inner{\ms(t)\Phi(\t)\int_{\t}^{t}
\Phi(s)^{-1}\sigma_{u}(s)v(s)\chi_{A}dW(s)}{
v(t)}\chi_{A} dt\nonumber\\
& &+\frac{1}{\theta^2} \me\int_{\t}^{\t+\theta}
\inner{\ms(t)\int_{\t}^{t}b_{x}(s)\Phi(s)ds\int_{\t}^{t}
\Phi(s)^{-1}\sigma_{u}(s)v(s)\chi_{A}dW(s)}{
v(t)}\chi_{A} dt\nonumber\\
&&+\frac{1}{\theta^2} \me\int_{\t}^{\t+\theta}
\inner{\ms(t)\int_{\t}^{t}\sigma_{x}(s)\Phi(s)dW(s)\int_{\t}^{t}
\Phi(s)^{-1}\sigma_{u}(s)v(s)\chi_{A}dW(s)}{ v(t)}\chi_{A}
dt.\nonumber
\end{eqnarray}

By the properties of the It\^{o} integral and the Lebesgue
differentiation theorem, it can be proved that
\begin{eqnarray}\label{limit of S(t)y1(theta,A)(t)part2equ1}
& &\lim_{\theta\to 0^+}\frac{1}{\theta^2} \me\int_{\t}^{\t+\theta}
\inner{\ms(t)\int_{\t}^{t}b_{x}(s)\Phi(s)ds\int_{\t}^{t}
\Phi(s)^{-1}\sigma_{u}(s)v(s)\chi_{A}dW(s)}{
v(t)}\chi_{A} dt\nonumber\\
&=& 0,  \ \ \ \ a.e.\ \ \t\in[0,T),
\end{eqnarray}
and
\begin{eqnarray}\label{limit of S(t)y1(theta,A)(t)part2equ2}
& &\lim_{\theta\to 0^+}\frac{1}{\theta^2} \me\int_{\t}^{\t+\theta}
\inner{\ms(t)\int_{\t}^{t}\sigma_{x}(s)\Phi(s)dW(s)\int_{\t}^{t}
\Phi(s)^{-1}\sigma_{u}(s)v(s)\chi_{A}dW(s)}{
v(t)}\chi_{A} dt\nonumber\\
&=&\frac{1}{2}\me~\Big(\inner{\ms(\t)\sigma_{x}(\t)
\sigma_{u}(\t)v(\t)} {v(\t)}\chi_{A}\Big), \ \  a.e.\ \t\in[0,T).
\end{eqnarray}

Next, the assumptions on $\ms$ and $v$ yield
$$\ms(\cdot)^{\top}v(\cdot)\in\ml^{1,2}_{\mmf}(\mrn)\cap L^{\infty}([0,T]\times\Omega;\mr^{n}).$$
Hence, by the Clark-Ocone formula, for a.e. $t\in [0,T)$,
\begin{equation}\label{expS(t)v(t)}
\ms(t)^{\top}v(t)
=\me~\big(\ms(t)^{\top}v(t)\big)+\int_{0}^{t}\me~\Big(\dd_{s}
\big(\ms(t)^{\top}v(t)\big)\ \Big|\ \mf_{s}\Big)dW(s).
\end{equation}
Substituting (\ref{expS(t)v(t)}) into the first term of the right
hand of (\ref{limit of S(t)y1(theta,A)(t)part2}), it follows that
\begin{eqnarray}\label{limit of S(t)y1(theta,A)(t)part2equ3}
& &\frac{1}{\theta^2} \me\int_{\t}^{\t+\theta}
\inner{\ms(t)\Phi(\t)\int_{\t}^{t}
\Phi(s)^{-1}\sigma_{u}(s)v(s)\chi_{A}dW(s)}{
v(t)}\chi_{A} dt\\
\!\!&=&\!\!\frac{1}{\theta^2}\int_{\t}^{\t+\theta}
\me~\inner{\int_{\t}^{t}\Phi(\t)\Phi(s)^{-1}\sigma_{u}(s)
v(s)\chi_{A}dW(s)}{\me~\big(\ms(t)^{\top}v(t)\big)
}\chi_{A}dt\nonumber\\
& &\!\!+\frac{1}{\theta^2}\int_{\t}^{\t+\theta}\!\!
\me\inner{\int_{\t}^{t}\Phi(\t)\Phi(s)^{-1}\sigma_{u}(s)
v(s)\chi_{A}dW(s)}{ \int_{0}^{t}\me\Big(\dd_{s}
\big(\ms(t)^{\top}v(t)\big)\Big| \mf_{s}\Big)dW(s)
}\chi_{A}dt\nonumber\\
\!\!&=&\!\!\frac{1}{\theta^2}\int_{\t}^{\t+\theta}\int_{\t}^{t}
\me~\inner{\Phi(\t)\Phi(s)^{-1}\sigma_{u}(s) v(s)}{
\dd_{s}\big(\ms(t)^{\top}v(t)\big) }\chi_{A}dsdt.\nonumber
\end{eqnarray}
Note that
$$\dd_{s}\big(\ms(t)^{\top}v(t)\big)
=\big(\dd_{s} \ms(t)^{\top}\big)v(t) + \ms(t)^{\top}\dd_{s}v(t).$$
Using the same argument as that in \cite[Theorem 3.13]{zhangH14a},
we conclude that there exists a sequence
$\{\theta_\ell\}_{\ell=1}^\infty$ of positive numbers such that
$\lim_{\ell\to\infty}\theta_\ell=0$, and
\begin{eqnarray}\label{limit of S(t)y1(theta,A)(t)part2equ4}
& &\lim_{\ell\to\infty}\frac{1}{\theta_{\ell}^2}
\me\int_{\t}^{\t+\theta_{\ell}} \inner{\ms(t)\Phi(\t)\int_{\t}^{t}
\Phi(s)^{-1}\sigma_{u}(s)v(s)\chi_{A}dW(s)}{
v(t)}\chi_{A} dt\\
\!\!&=&\!\! \!\!\frac{1}{2}
\me\Big(\inner{\nabla\ms(\t)^{\top}v(\t)}
{\sigma_{u}(\t)v(\t)}\chi_{A}\Big)
\!+\!\frac{1}{2} \me\Big(\inner{\ms(\t)^{\top}\nabla v(\t)}
{\sigma_{u}(\t)v(\t)}\chi_{A}\Big),\   a.e.\ \t\in[0,T).\nonumber
\end{eqnarray}

Then, by (\ref{limit of S(t)y1(theta,A)(t)part2}), (\ref{limit of
S(t)y1(theta,A)(t)part2equ4}), one concludes that
\begin{eqnarray}\label{limit of S(t)y1(theta,A)(t)part2final}
& &\lim_{\ell\to \infty}\frac{1}{\theta_\ell^{2}}
\me\int_{\t}^{\t+\theta_{\ell}} \inner{\ms(t)\Phi(t)\int_{\t}^{t}
\Phi(s)^{-1}\sigma_{u}(s)v(s)\chi_{A}dW(s)}{v(t)}\chi_{A} dt\\
&=&
\frac{1}{2}\me~\Big(\inner{\ms(\t)\sigma_{x}(\t)\sigma_{u}(\t)v(\t)}
{v(\t)}\chi_{A}\Big)
+\frac{1}{2} \me~\Big(\inner{\nabla \ms(\t)^{\top}
v(\t)}{\sigma_{u}(\t)v(\t)}\chi_{A}\Big)
\nonumber\\
& & +\frac{1}{2} \me~\Big(\inner{\ms(\t)^{\top}\nabla v(\t)}
{\sigma_{u}(\t)v(\t)}\chi_{A}\Big),
\  \ a.e. \ \t\in [0,T).\nonumber
\end{eqnarray}
Combining (\ref{limit of S(t)y1(t)}), (\ref{limit of
S(t)y1(theta,A)(t)part1}) and (\ref{limit of
S(t)y1(theta,A)(t)part2final}), one has
\begin{eqnarray*}
\qquad 0&\ge&\me~\Big(
\inner{\ms(\t)b_{u}(\t)v(\t))}{v(\t)}\chi_{A}\Big)
+\me~\Big(\inner{\nabla \ms(\t)^{\top}
v(\t)}{\sigma_{u}(\t)v(\t)}\chi_{A}\Big)\nonumber\\
& & +\me~\Big(\inner{\ms(\t)^{\top}\nabla v(\t)}
{\sigma_{u}(\t)v(\t)}\chi_{A}\Big),\quad a.e. \ \t\in
[0,T).\nonumber
\end{eqnarray*}
Finally, by the arbitrariness of $A\in \mf_{\t}$, we deduce that
the desired second order necessary condition (\ref{partial
pointwise 2order condition}) holds. This completes the proof of
Theorem \ref{TH partial pointwise 2order condition}.
\end{proof}
\end{theorem}

If $\bar{u}\in \ml_{2,\mmf}^{1,2}(\mrm)$, $U$ is a bounded closed
convex set in $\mrm$, $v-\bar{u}(\cdot)\in
\ml_{2,\mmf}^{1,2}(\mrm)\cap
L^{\infty}([0,T]\times\Omega;\mr^{m})\cap \ma^{1}$ holds true for
any $v\in U$. Then, by  Theorem \ref{TH partial pointwise 2order
condition} and the separability of $U$, one has
\begin{eqnarray}\label{2order pointwise condition for convex}
& &\inner{\ms(\t)b_{u}(\t)(v -\bar{u}(\t))}{v-\bar{u}(\t)}
+ \inner{\nabla \ms(\t)\sigma_{u}(\t)(v-\bar{u}(\t))}
{v-\bar{u}(\t)}\nonumber\\
& &\quad-\inner{\ms(\t)\sigma_{u}(\t)(v-\bar{u}(\t))}
{\nabla\bar{u}(\t)}\le 0, \qquad \forall\; v\in U, \ a.e.\
\t\in[0,T],\ a.s.,
\end{eqnarray}
which coincides with \cite[Theorem 3.13]{zhangH14a}. However, when
the control set $U$ is nonconvex, some more assumptions as follows
are required to establish a pointwise condition similar to
(\ref{2order pointwise condition for convex}).

\begin{enumerate}
\item [{\bf (C4)}] For any $u\in \partial U$ and $v\in \T_{U}(u)$,
$\TT_{U}(u,v)\neq \emptyset$.
\end{enumerate}

When the control set $U$ has a $C^2$ boundary, the assumption (C4)
holds, see \cite{Frankowska13}.

From the proof of Theorem \ref{TH partial pointwise 2order
condition}, we deduce the following result.

\begin{corollary}\label{TH step optim control}
Let (C1)--(C4) hold, $(\bx,\bu,\bx_{0})$ be a singular local
minimizer for the problem (\ref{minimum J}). If
$\ms\in\ml_{2,\mmf}^{1,2}(\mr^{m\times n})$, and the optimal
control $\bu$ is a step function as below
\begin{equation}\label{20160302e1}
\bar{u}(t,\omega)=\sum_{i=1}^{k}\sum_{j=1}^{l_{i}}
u_{ij}\chi_{A_{ij}}\chi_{[t_{i},t_{i+1})}(t,\omega),\ a.e.\
(t,\omega)\in [0,T]\times \Omega,
\end{equation}
where $k\in\mathbb{N}$, $ 0=t_{1}<\cdots<t_{k+1}=T$, $l_{i}\in
\mathbb{N}$, $u_{ij}\in U$ and $A_{ij}\in \mathcal{F}_{t_{i}}$ for
$i=1,\cdots, k$ and $j=1,\cdots, l_{i}$, then, in addition to the
second order  transversality  condition (\ref{second order
trans}), the following pointwise second order necessary condition
holds:
\begin{equation*}\label{2order pointwise condi for step optim control}
\inner{\ms(\t,\omega)b_{u}(\t,\omega)v}{v}
+ \inner{\nabla \ms(\t,\omega)\sigma_{u}(\t,\omega)v} {v} \le 0,
\quad \forall\; v\in \T_{U}(\bar{u}(\t,\omega)), \ a.e.\
\t\in[0,T],\ a.s.
\end{equation*}

\begin{proof}
When $\bar u(t,\omega)$ is given as in (\ref{20160302e1}), for any
fixed $i$ and $j$, $\bu(t,\omega)$ has constant value $u_{ij}$ on
$[t_{i},t_{i+1})\times A_{ij}$. Then, on $[t_{i},t_{i+1})\times
A_{ij}$, let $v_{ij}\in \T_{U}(u_{ij})$, $h_{ij}\in
\TT_{U}(u_{ij},v_{ij})$, $\t\in[t_{i},t_{i+1})$, $\theta\in
(0,t_{i+1}-\t)$, $E_{\theta}=[\t, \t+\theta)$ and choose $A\in
\mf_{t_{i}}$. Define
$$
v^{\theta,A}(t,\omega)=\left\{
\begin{array}{l}
v_{ij}, \ \ (t,\omega)\in E_{\theta}\times (A\cap A_{ij}),\\[+0.6em]
0, \ \ \ \ \, \text{otherwise},\\
\end{array}\right.\
h^{\theta,A}(t,\omega)=\left\{
\begin{array}{l}
h_{ij},\ \  (t,\omega)\in E_{\theta}\times (A\cap A_{ij}),\\[+0.6em]
0,   \ \ \ \ \,     \text{otherwise}.
\end{array}\right.
$$
It is clear that $(v^{\theta,A},h^{\theta,A})\in \ma$. Then, by
similar arguments as in the proof of Theorem \ref{TH partial
pointwise 2order condition} and noting that the Malliavin
derivative of the constant-valued process $v_{ij}$ is equal to 0,
we obtain that
\begin{equation*}
\inner{\ms(\t,\omega)b_{u}(\t,\omega)v_{ij}}{v_{ij}}
+ \inner{\nabla \ms(\t,\omega)\sigma_{u}(\t,\omega)v_{ij}}
{v_{ij}} \le 0,\quad \ a.e.\ (\t,\omega)\in[t_{i},t_{i+1})\times
A_{ij}.
\end{equation*}
By the closedness of the adjacent cone, the separability of
$\mrm$, the arbitrariness of $i$, $j$ and $v_{ij}$ it follows that
\begin{equation*}
\inner{\ms(\t,\omega)b_{u}(\t,\omega)v}{v}
+ \inner{\nabla \ms(\t,\omega)\sigma_{u}(\t,\omega)v} {v} \le 0,\
\forall\; v\in \T_{U}(\bar{u}(\t,\omega)), \ a.e.\ \t\in[0,T],\
a.s.
\end{equation*}
This completes the proof of Corollary \ref{TH step optim control}.
\end{proof}
\end{corollary}

\begin{example}\label{example2}
Let the optimal control problem be the one stated in Example
\ref{example1}. We have shown that $\bar u(t)=(\bar u_{1}(t),\bar
u_{2}(t))\equiv(1,0)$ is the optimal control. In the following we
will prove that this optimal control is partially singular and
satisfies the second order
 necessary condition (\ref{2order pointwise condi for step optim control}).

In Example \ref{example1} we obtained that the corresponding state
$(\bar x_{1}(t),\bar x_{2}(t))$ is as in (\ref{20160710e3}) and the
first order adjoint process $(P_{1} (t),Q_{1} (t))$ is as in
(\ref{20160710e4}). In addition, it is easy to see that the second
order adjoint equation is
\begin{equation}\nonumber
\quad \left\{
\begin{array}{l}
\!d\left[ \begin{array}{cc}
          P_{2}^{1}(t) & P_{2}^{2}(t) \\
          P_{2}^{3}(t) &P_{2}^{4}(t) \\
\end{array} \right]
\!=\!\left[\! \begin{array}{cc}
          0 & -P_{2}^{1}(t) \\
          -P_{2}^{1}(t) &\!-\!P_{2}^{2}(t)\!-\! P_{2}^{3}(t) \\
\end{array} \!\right]\! dt
\!+\!\left[\! \begin{array}{cc}
          Q_{2}^{1}(t) & Q_{2}^{2}(t) \\
          Q_{2}^{3}(t) & Q_{2}^{4}(t) \\
        \end{array} \!\right]\!dW(t),
\   t\!\in\![0,1], \qquad\\[+1.3em]
\!\left[ \begin{array}{cc}
          P_{2}^{1}(1) & P_{2}^{2}(1) \\
          P_{2}^{3}(1) &P_{2}^{4}(1) \\
\end{array} \right]
\!=\!\left[ \begin{array}{cc}
          -1 & 0 \\
          0 &  0 \\
\end{array} \right]
\end{array}\right.
\end{equation}
and its solution is
$$
\Bigg(\left[ \begin{array}{cc}
          P_{2}^{1}(t) & P_{2}^{2}(t) \\
          P_{2}^{3}(t) &P_{2}^{4}(t) \\
\end{array} \right], \left[ \begin{array}{cc}
          Q_{2}^{1}(t) & Q_{2}^{2}(t) \\
          Q_{2}^{3}(t) & Q_{2}^{4}(t) \\
\end{array} \right]\Bigg)
= \Bigg(\left[ \begin{array}{cc}
          -1 & t-1 \\
          t-1 & -t^2+2t-1 \\
\end{array} \right], \left[ \begin{array}{cc}
          0 & 0 \\
          0 & 0 \\
\end{array} \right]\Bigg).
$$
A direct calculation shows that
$$\T_{U}((1,0))=\{(v_{1},0)\in \mr^2\ |\ v_{1}\le 0\}.$$

Then, we have
$$H_{u}(t)
=0,\ H_{uu}(t) =0,\quad \forall \ (t,\omega)\in
[0,1]\times\Omega,$$
$$ \sigma_{u}(t)^{\top}
P_{2}(t)\sigma_{u}(t) = \left[ \begin{array}{cc}
          0 & 0 \\
          0 & 1 \\
\end{array} \right]
\left[ \begin{array}{cc}
          -1 & t-1 \\
          t-1 & -t^2+2t-1 \\
\end{array} \right]
\left[ \begin{array}{cc}
          0 & 0 \\
          0 & 1 \\
\end{array} \right]
=\left[ \begin{array}{cc}
          0 & 0 \\
          0 & -t^2+2t-1 \\
\end{array} \right]
$$
and therefore
$$\inner{\big(H_{uu}(t)+ \sigma_{u}(t)^{\top}P_{2}(t)\sigma_{u}(t)\big) v}{v}=0,
\quad \forall\ v\in \T_{U}(\bar{u}(t)),\ a.e.\ t\in [0,T], \  a.s.
$$ This means that $\bar u(t)=(\bar u_{1}(t),\bar
u_{2}(t))\equiv(1,0)$ is partially singular. Next, we prove that
$\bar u(t)=(\bar u_{1}(t),\bar u_{2}(t))\equiv(1,0)$ satisfies the
second order necessary condition in Corollary \ref{TH step optim
control}. It is clear that
$$\ms (t)=\left[ \begin{array}{cc}
          0 & 1 \\
          0 & 0 \\
\end{array} \right]
\left[ \begin{array}{cc}
          -1 & t-1 \\
          t-1 & -t^2+2t-1 \\
\end{array} \right]
=\left[ \begin{array}{cc}
          t-1 & -t^2+2t-1 \\
          0 & 0\\
\end{array} \right]
$$
Then, $\nabla\ms(t)\equiv 0$, and
\begin{eqnarray*}
\inner{\ms(t)b_{u}(t)v}{v}+\inner{\nabla\ms(t)\sigma_{u}(t)v}{v}
\!\!&=&\!\!\left[\begin{array}{cc}
 v_{1} & 0
\end{array}\right]
\left[ \begin{array}{cc}
           -t^2+2t-1 & 0\\
          0 & 0\\
\end{array} \right]
\left[ \begin{array}{c}
           v_{1}\\
            0\\
\end{array} \right]\\
\!\!&=&\!\!-(t-1)^2v_{1}^2\le 0, \ \forall\ v\in
\T_{U}(\bar{u}(t)), \  a.e.\ t\in [0,T], \  a.s.
\end{eqnarray*}

\end{example}

\section{Appendix}
\appendix
In this section, we prove the two technical  Lemmas \ref{estimate
one of varie qu} and  \ref{estimate two of varie qu}. The
fundamental idea comes from the classical calculus, see also the
related results in \cite{Bensoussan81, Bonnans12} for the optimal
control problems with convex control constraints, and
\cite{Peng90, Yong99} for the general control constraints.

\section{Proof of Lemma \ref{estimate one of varie qu}}
\begin{proof}
From  (\ref{first vari equ}) and  Lemma \ref{estimatelinearsde} we
deduce that
\begin{eqnarray*}
\me\Big(\sup_{t\in[0,T]}|y_{1}(t)|^{\beta}\Big) &\le &
C\me~\Big[|\nu_{0}|^{\beta}+
\Big(\int_{0}^{T}|b_{u}(t)v(t)|dt\Big)^{\beta}
+\Big(\int_{0}^{T}|\sigma_{u}(t)v(t)|^{2}dt\Big)^\frac{\beta}{2}\Big]\\
&\le
&C\me~\Big[|\nu_{0}|^{\beta}+\Big(\int_{0}^{T}|v(t)|^2dt\Big)^\frac{\beta}{2}\Big].
\end{eqnarray*}
Since $v_{\eps}(\cdot)$ converges to $v(\cdot)$ in
$L^\beta_{\mmf}(\Omega;L^2(0,T;\mrm))$, and $\nu_{0}^{\eps}\to
\nu_{0}$ in $\mrn$  as $\eps\to 0^+$, we deduce from
(\ref{estimateof delta x}) that
\begin{equation*}
\me\Big(\sup_{t\in[0,T]}|\delta x^{\eps}(t)|^{\beta}\Big) \le
C\me~\Big[\eps^{\beta}|\nu^{\eps}_{0}|^{\beta}+ \Big(
\int_{0}^{T}|\eps v_{\eps}(t)|^{2}ds\Big)^\frac{\beta}{2}\Big]
=O(\eps^\beta).
\end{equation*}
Consequently, by the H\"older inequality,
\begin{equation}\label{13-11}
\me\Big(\sup_{t\in[0,T]}|\delta x^{\eps}(t)|\Big) \le \Big[
\me~\Big(\sup_{t\in[0,T]}|\delta
x^{\eps}(t)|^\beta\Big)\Big]^{1/\beta} =O(\varepsilon )
\end{equation}
and \begin{equation}\label{13-11a} \me\int_{0}^{T} |v_\eps(t)
-v(t)|dt \le C\Big[\me\Big(\int_{0}^{T} |v_\eps(t)
-v(t)|^{2}dt\Big)^{\frac{\beta}{2}}\Big]^{\frac{1}{\beta}}  \to 0,
\;\; \eps \to 0^+ .
\end{equation}

Denote $\tilde{b}_{x}^{\eps}(t):=\int_{0}^{1}b_{x}(t,\bx(t)
+\theta\delta x^{\eps}(t),\bu(t)+\theta\eps v_{\eps}(t))d\theta$.
Mappings $\tilde{b}_{u}^{\eps}(t)$, $\tilde{\sigma}_{x}^{\eps}(t)$
and $\tilde{\sigma}_{u}^{\eps}(t)$ are defined in a similar way.
Then, $\delta x^{\eps} (\cdot)$ is the solution to the following
stochastic differential equation
\begin{equation*}
\left\{
\begin{array}{l}
d\delta x^{\eps}(t)= \big(\tilde{b}_{x}^{\eps}(t)\delta
x^{\eps}(t)
+\eps\tilde{b}_{u}^{\eps}(t) v_{\eps}(t)\big)dt\\[+0.6em]
\qquad\qquad+\big(\tilde{\sigma}_{x}^{\eps}(t)\delta x^{\eps}(t)
+\eps\tilde{\sigma}_{u}^{\eps}(t) v_{\eps}(t)\big)dW(t),\ \  t\in[0,T],\\[+0.6em]
\delta x^{\eps}(0)=\eps \nu_{0}^{\eps},
\end{array}\right.
\end{equation*}
and   $r_{1}^{\eps}(\cdot) $ satisfies the following stochastic
differential equation
\begin{equation}\label{frac detax-y1 eps}
\left\{
\begin{array}{l}
dr_{1}^{\eps}(t)= \Big[\tilde{b}_{x}^{\eps}(t)r_{1}^{\eps}(t)
+\big(\tilde{b}_{x}^{\eps}(t)-b_{x}(t)\big)y_{1}(t)
+\tilde{b}_{u}^{\eps}(t)\big(v_{\eps}(t)-v(t)\big)\\[+0.3em]
\qquad\qquad
+\big(\tilde{b}_{u}^{\eps}(t)-b_{u}(t)\big)v(t)\Big]dt
+\Big[\tilde{\sigma}_{x}^{\eps}(t)r_{1}^{\eps}(t)
+\big(\tilde{\sigma}_{x}^{\eps}(t)-\sigma_{x}(t)\big)y_{1}(t)\\[+0.3em]
\qquad\qquad
+\tilde{\sigma}_{u}^{\eps}(t)\big(v_{\eps}(t)-v(t)\big)
+\big(\tilde{\sigma}_{u}^{\eps}(t)-\sigma_{u}(t)\big)v(t)\Big]dW(t),\ \  t\in[0,T],\\[+0.3em]
r_{1}^{\eps}(0)=\nu^{\eps}_{0}-\nu_{0}.
\end{array}\right.
\end{equation}

For any sequence $\{\eps_{j}\}_{j=1}^\infty$ of positive numbers
converging to $0$ as $j\to \infty$, we can find a subsequence
$\{j_k\}_{k=1}^\infty\subset\mn$ such that
$\sup_{t\in[0,T]}|\delta x^{\eps_{j_k}}(t)| \to 0$ a.s.  and
$\eps_{j_k} v_{\eps_{j_k}}(t) \to 0 $  a.s. for a.e. $t \in
[0,T]$, as $k\to\infty$. The assumption (C2) yields,
$\big|(\tilde{b}_{x}^{\eps_{j_k}}(t)-b_{x}(t))y_{1}(t) \big| \to
0$ a.s. for a.e. $t \in [0,T]$, as $k\to\infty$. Hence,
$$  \big|(\tilde{b}_{x}^{\eps_j}(\cdot)-b_{x}(\cdot))y_{1}(\cdot) \big| \to 0\; \mbox{ in  measure, as }j\to\infty. $$
  Then,  using  Lebesgue's dominated convergence theorem, we conclude that
\begin{equation}\label{Lemma3_2 equ for control conver1}
\me\Big(\int_{0}^{T}|\big(\tilde{b}_{x}^{\eps_{j}}(t)
-b_{x}(t)\big)y_{1}(t)|^2dt\Big)^{\frac{\beta}{2}}\to 0,\quad j\to
\infty.
\end{equation}
A slight modification of the above discussion shows that
\begin{eqnarray}\label{Lemma3_2 equ for control conver2}
& &\me\Big(\int_{0}^{T}|\big(\tilde{b}_{u}^{\eps_j}(t)
-b_{u}(t)\big)v(t)|^2dt\Big)^{\frac{\beta}{2}}
+\me\Big(\int_{0}^{T}|\big(\tilde{\sigma}_{x}^{\eps_j}(t)
-\sigma_{x}(t)\big)y_{1}(t)|^2dt\Big)^{\frac{\beta}{2}}\nonumber\\
& & +\me\Big(\int_{0}^{T}|\big(\tilde{\sigma}_{u}^{\eps_j}(t)
-\sigma_{u}(t)\big)v(t)|^2dt\Big)^{\frac{\beta}{2}}\to 0, \quad j
\to \infty.
\end{eqnarray}
On the other hand
\begin{eqnarray*}
&&\me\Big(\int_{0}^{T}|\tilde{b}_{u}^{\eps_j}(t)
\big(v_{\eps_j}(t)-v(t)\big)|^2dt\Big)^{\frac{\beta}{2}}
+\me\Big(\int_{0}^{T}|\tilde{\sigma}_{u}^{\eps_j}(t)
\big(v_{\eps_j}(t)-v(t)\big)|^2dt\Big)^{\frac{\beta}{2}}\\
&\le& C\me\Big(\int_{0}^{T}
|v_{\eps_j}(t)-v(t)|^2dt\Big)^{\frac{\beta}{2}} \to 0,\quad j\to
\infty,
\end{eqnarray*}
Therefore, by Lemma \ref{estimatelinearsde}, we finally obtain
that
\begin{eqnarray*}
\me\Big(\sup_{t\in[0,T]}|r_{1}^{\eps_j}(t)|^{\beta}\Big) &\le &
C\me\Big[|\nu^{\eps_j}_{0}-\nu_{0}|^{\beta}+\Big(\int_{0}^{T}\big|\big(\tilde{b}_{x}^{\eps_j}(t)-b_{x}(t)\big)y_{1}(t)
+\tilde{b}_{u}^{\eps_j}(t)\big(v_{\eps_j}(t)-v(t)\big)\\[+0.3em]
&+ &
\big(\tilde{b}_{u}^{\eps_j}(t)-b_{u}(t)\big)v(t)\big|dt\Big)^{\beta}
+\Big(\int_{0}^{T}\big|\big(\tilde{\sigma}_{x}^{\eps_j}(t)
-\sigma_{x}(t)\big)y_{1}(t)\\[+0.3em]
&+ & \tilde{\sigma}_{u}^{\eps_j}(t)\big(v_{\eps_j}(t)-v(t)\big)
+\big(\tilde{\sigma}_{u}^{\eps_j}(t)-\sigma_{u}(t)\big)v(t)
\big|^2dt\Big)^{\frac{\beta}{2}}\Big]\to 0,\quad j\to \infty.
\end{eqnarray*}
The sequence $\varepsilon_j \to 0^+$ being arbitrary, the proof is
complete.
\end{proof}

\section{Proof of Lemma \ref{estimate two of varie qu}}

\begin{proof}
By Lemma \ref{estimate one of varie qu} (with $\beta$  replaced by
$2\beta$), we obtain that
\begin{equation}\label{estimate of y1 in estmate2}
\me\Big(\sup_{t\in[0,T]}|y_{1}(t)|^{2\beta}\Big) \le C\me~\Big[
|\nu_{0}|^{2\beta} +\Big(\int_{0}^{T}|v(t)|^2dt\Big)^{\beta}\Big].
\end{equation}
Then, by  (\ref{second order vari equ}), Lemma
\ref{estimatelinearsde} and the H\"older inequality, it follows
that
\begin{eqnarray*}
&&\me\Big(\sup_{t\in[0,T]}|y_{2}(t)|^{\beta}\Big)\\
&\le & C\me~\Big[|\varpi_{0}|^{\beta}+
\Big(\int_{0}^{T}|2b_{u}(t)h(t)
+y_{1}(t)^{\top}b_{xx}(t)y_{1}(t)+2v(t)^{\top}b_{xu}(t)y_{1}(t)
\\
& &\qquad
+v(t)^{\top}b_{uu}(t)v(t)|dt\Big)^{\beta}+\Big(\int_{0}^{T}|2\sigma_{u}(t)h(t) +y_{1}(t)^{\top}\sigma_{xx}(t)y_{1}(t)\\
& &\qquad +2v(t)^{\top}\sigma_{xu}(t)y_{1}(t)
+v(t)^{\top}\sigma_{uu}(t)v(t)|^{2}dt\Big)^{\frac{\beta}{2}}\Big]\\
&\le &C\me~\Big[|\varpi_{0}|^{\beta}+
\Big(\int_{0}^{T}|h(t)|^2dt\Big)^{\frac{\beta}{2}}
+\sup_{t\in [0,T]}|y_{1}(t)|^{2\beta}\\
& &\qquad\qquad\qquad\ +\sup_{t\in [0,T]}|y_{1}(t)|^{\beta}\cdot
\Big(\int_{0}^{T}|v(t)|^2dt\Big)^{\frac{\beta}{2}}
+\Big(\int_{0}^{T}|v(t)|^4dt\Big)^{\frac{\beta}{2}}\Big]\\
&\le& C\me~\Big[|\varpi_{0}|^{\beta}+|\nu_{0}|^{2\beta}
+\Big(\int_{0}^{T}|h(t)|^2dt\Big)^{\frac{\beta}{2}}
+\Big(\int_{0}^{T}|v(t)|^4dt\Big)^{\frac{\beta}{2}}\Big].
\end{eqnarray*}

Denote
$\tilde{b}_{xx}^{\eps}(t):=\int_{0}^{1}(1-\theta)b_{xx}(t,\bx(t)
+\theta\delta x^{\eps}(t),\bu(t)+\theta\delta
u^{\eps}(t))d\theta$. Mappings $\tilde{b}_{xu}^{\eps}(t)$,
$\tilde{b}_{uu}^{\eps}(t)$, $\tilde{\sigma}_{xx}^{\eps}(t)$,
$\tilde{\sigma}_{xu}^{\eps}(t)$ and
$\tilde{\sigma}_{uu}^{\eps}(t)$ are defined in a similar way.
Then, $\delta x^{\eps}$ satisfies the following stochastic
differential equation:
\begin{displaymath}\label{exp deltax order two}
\left\{
\begin{array}{l}
d\delta x^{\eps}(t)= \Big(b_x(t)\delta x^{\eps}(t) +b_{u}(t)\delta
u^{\eps}(t)
+\delta x^{\eps}(t)^{\top}\tilde{b}_{xx}^{\eps}(t)\delta x^{\eps}(t)\\[+0.4em]
\qquad\qquad\qquad\ +2\delta
x^{\eps}(t)^{\top}\tilde{b}_{xu}^{\eps}(t)\delta u^{\eps}(t)
+\delta u^{\eps}(t)^{\top}\tilde{b}_{uu}^{\eps}(t)\delta
u^{\eps}(t)
\Big)dt\\[+0.4em]
\qquad\qquad\quad\ +\Big(\sigma_x(t)\delta x^{\eps}(t)
+\sigma_{u}(t)\delta u^{\eps}(t)
+\delta x^{\eps}(t)^{\top}\tilde{\sigma}_{xx}^{\eps}(t)
\delta x^{\eps}(t)\\[+0.4em]
\qquad\qquad\qquad\ +2\delta
x^{\eps}(t)^{\top}\tilde{\sigma}_{xu}^{\eps}(t) \delta u^{\eps}(t)
+\delta u^{\eps}(t)^{\top}\tilde{\sigma}_{uu}^{\eps}(t)
\delta u^{\eps}(t)\Big)dW(t),\quad t\in [0,T],\\[+0.4em]
\delta x^{\eps}(0)=\eps \nu_{0}+\eps^2\varpi_{0}^{\eps}.
\end{array}\right.
\end{displaymath}
Therefore  $r_{2}^{\eps}$ solves  the following stochastic
differential equation:
\begin{equation}\label{frac detax-eps y1-eps2y2 eps2}
\left\{
\begin{array}{l}
dr_{2}^{\eps}(t)= \Big\{b_{x}(t)r_{2}^{\eps}(t)
+b_{u}(t)\big(h_{\eps}(t)-h(t)\big)\\[+0.5em]
\qquad\qquad
+\big[\big(\frac{\delta x^{\eps}(t)}{\eps}\big)^{\top}\tilde{b}_{xx}^{\eps}(t)\big(\frac{\delta x^{\eps}(t)}{\eps}\big)-\frac{1}{2}y_{1}(t)^{\top}b_{xx}(t)y_{1}(t)\big]\\[+0.5em]
\qquad\qquad
+\big[2\big(\frac{\delta x^{\eps}(t)}{\eps}\big)^{\top}\tilde{b}_{xu}^{\eps}(t)\big(\frac{\delta u^{\eps}(t)}{\eps}\big)-y_{1}(t)^{\top}b_{xu}(t)v(t)\big]\\[+0.5em]
\qquad\qquad
+\big[\big(\frac{\delta u^{\eps}(t)}{\eps}\big)^{\top}\tilde{b}_{uu}^{\eps}(t)\big(\frac{\delta u^{\eps}(t)}{\eps}\big)-\frac{1}{2}v(t)^{\top}b_{uu}(t)v(t)\big]\Big\}dt\\[+0.5em]
\qquad\qquad +\Big\{\sigma_{x}(t)r_{2}^{\eps}(t)
+\sigma_{u}(t)\big(h_{\eps}(t)-h(t)\big)\\[+0.5em]
\qquad\qquad +\big[\big(\frac{\delta
x^{\eps}(t)}{\eps}\big)^{\top}
\tilde{\sigma}_{xx}^{\eps}(t)\big(\frac{\delta
x^{\eps}(t)}{\eps}\big)
-\frac{1}{2}y_{1}(t)^{\top}\sigma_{xx}(t)y_{1}(t)\big]\\[+0.5em]
\qquad\qquad +\big[2\big(\frac{\delta
x^{\eps}(t)}{\eps}\big)^{\top}
\tilde{\sigma}_{xu}^{\eps}(t)\big(\frac{\delta u^{\eps}(t)}{\eps}\big)-y_{1}(t)^{\top}\sigma_{xu}(t)v(t)\big]\\[+0.5em]
\qquad\qquad +\big[\big(\frac{\delta
u^{\eps}(t)}{\eps}\big)^{\top}
\tilde{\sigma}_{uu}^{\eps}(t)\big(\frac{\delta u^{\eps}(t)}{\eps}\big)-\frac{1}{2}v(t)^{\top}\sigma_{uu}(t)v(t)\big]\Big\}dW(t),\ \  t\in[0,T],\\[+0.5em]
r_{2}^{\eps}(0)=\varpi_{0}^{\eps}-\varpi_{0}.
\end{array}\right.
\end{equation}
Since $h_{\eps}(\cdot)$ converges to $h(\cdot)$ in
$L^{2\beta}_{\mmf}(\Omega;L^{4}(0,T;\mrm))$,

\begin{equation}\label{equ1add in estimate 2}
\me\Big(\int_{0}^{T}\Big|b_{u}(t)\big(h_{\eps}(t)-h(t)\big)
\Big|dt\Big)^{\beta} +
\me\Big(\int_{0}^{T}\Big|\sigma_{u}(t)\big(h_{\eps}(t)-h(t)\big)
\Big|^2dt\Big)^{\frac{\beta}{2}} \to 0,\quad \eps\to 0^+.
\end{equation}
On the other hand, by the H\"older inequality,
\begin{eqnarray}\label{equ1 in estimate 2}& &\me\Big(\int_{0}^{T}\Big|\big(\frac{\delta x^{\eps}(t)}{\eps}\big)^{\top}
\tilde{b}_{xx}^{\eps}(t)\big(\frac{\delta x^{\eps}(t)}{\eps}\big)
-\frac{1}{2}y_{1}(t)^{\top}b_{xx}(t)y_{1}(t)\Big|dt\Big)^{\beta}
\nonumber\\
&\le & C \me \Big(\int_{0}^{T}\Big|\big(\frac{\delta
x^{\eps}(t)}{\eps}\big)^{\top}
\tilde{b}_{xx}^{\eps}(t)\big(\frac{\delta x^{\eps}(t)}{\eps}\big)
-\frac{1}{2}y_{1}(t)^{\top}b_{xx}(t)y_{1}(t)\Big|^2dt\Big)^{\frac{\beta}{2}}
\nonumber\\
&\le&C\me\Big[\int_{0}^{T}\Big|\big(\frac{\delta
x^{\eps}(t)}{\eps}\big)^{\top}\big(\tilde{b}_{xx}^{\eps}(t)
-\frac{1}{2}b_{xx}(t)\big)\big(\frac{\delta x^{\eps}(t)}{\eps}\big)\Big|^2dt\Big]^{\frac{\beta}{2}}\nonumber\\
& &+C\me\Big[\sup_{t \in [0,T]}\Big|\frac{\delta
x^{\eps}(t)}{\eps} -y_1(t)\Big|^\beta \Big( \sup_{t \in
[0,T]}\Big|\frac{\delta x^{\eps}(t)}{\eps}\Big|^\beta +
\sup_{t \in [0,T]}|y_1(t)|^\beta\Big)\Big]  \nonumber\\
&\le&C \Big[\me \Big(\sup_{t \in [0,T]} \Big|\frac{\delta
x^{\eps}(t)}{\eps}\Big|^{2\beta} \Big) \Big]^{1/2} \Big[\me
\Big(\int_{0}^{T}\Big|\tilde{b}_{xx}^{\eps}(t)
-\frac{1}{2}b_{xx}(t)\Big|^4\cdot \Big|\frac{\delta x^{\eps}(t)}{\eps}\Big|^{4}dt\Big)^{\frac{\beta}{2}} \Big]^{1/2}\nonumber\\
&&+ C\Big[\me \Big(\sup_{t \in [0,T]}\Big|\frac{\delta
x^{\eps}(t)}{\eps} -y_1(t)\Big|^{2\beta}\Big)\Big]^\frac12  \Big[
\me \Big(\sup_{t \in [0,T]}\Big|\frac{\delta
x^{\eps}(t)}{\eps}\Big|^{2\beta} + \sup_{t \in
[0,T]}|y_1(t)|^{2\beta}\Big)\Big]^\frac12 .
\end{eqnarray}
Since $h_{\eps}$ converges to $h$ in
$L^{2\beta}_{\mmf}(\Omega;L^4(0,T;\mrm))$ and $\varpi_{0}^{\eps}$
converges to $\varpi_{0}$ in $\mrm$ as $\eps\to 0^+$, by Lemma
\ref{estimatelinearsde},
\begin{equation*}
\me~\Big(\sup_{t\in[0,T]}|\delta x^{\eps}(t)|^{2\beta}\Big) \le
C\me~\Big[|\eps \nu_{0}+\eps^2\varpi_{0}^{\eps}|^{2\beta}+
\Big(\int_{0}^{T}|\eps v(t)+\eps^2
h_{\eps}(t)|^{2}dt\Big)^{\beta}\Big] =O(\eps^{2\beta}).
\end{equation*}
As in the proof of (\ref{r1 to 0}) in Lemma \ref{estimate one of
varie qu}, we obtain that
\begin{equation*}
\me\Big(\sup_{t\in[0,T]}\Big|\frac{\delta
x^{\eps}(t)}{\eps}-y_{1}(t)\Big|^{2\beta}\Big)\to 0,\quad \eps\to
0^+.
\end{equation*}
For any sequence $\{\eps_{j}\}_{j=1}^\infty$ of positive numbers
converging to $0$ as $j\to \infty$, one can show that
$$b_{xx}(\cdot,\bx(\cdot)
+\theta\delta x^{\eps_j}(\cdot),\bu(\cdot)+\theta\delta
u^{\eps_j}(\cdot)) -b_{xx}(\cdot) \to 0,\;\;\;  \mbox{ in
measure, as }j\to\infty.
 $$
Since
$$\tilde{b}_{xx}^{\eps_j}(t)
-\frac{1}{2}b_{xx}(t) = \int_{0}^{1}(1-\theta)\big(b_{xx}(t,\bx(t)
+\theta\delta x^{\eps_j}(t),\bu(t)+\theta\delta u^{\eps_j}(t))
-b_{xx}(t)\big)d\theta,
$$
from(C3),  (\ref{equ1 in estimate 2})   and the Lebesgue dominated
convergence theorem, we obtain that
\begin{equation}\label{equ4 in estimate 2}
\me\Big(\int_{0}^{T}\Big|\big(\frac{\delta
x^{\eps_{j}}(t)}{\eps_{j}}\big)^{\top}
\tilde{b}_{xx}^{\eps_{j}}(t)\big(\frac{\delta
x^{\eps_{j}}(t)}{\eps_{j}}\big)
-\frac{1}{2}y_{1}(t)^{\top}b_{xx}(t)y_{1}(t)\Big|dt\Big)^{\beta}
\to 0,\quad \hbox{as }j\to \infty.
\end{equation}
Similarly,
\begin{eqnarray*}
& &\me\Big(\int_{0}^{T}\Big|2\big(\frac{\delta
x^{\eps_{j}}(t)}{\eps_{j}}\big)^{\top}
\tilde{b}_{xu}^{\eps_{j}}(t)\big(\frac{\delta
u^{\eps_{j}}(t)}{\eps_{j}}\big)
-y_{1}(t)^{\top}b_{xu}(t)v(t)\Big|dt\Big)^{\beta}\nonumber\\
& \leq &C\me\Big(\int_{0}^{T}\Big|2\big(\frac{\delta
x^{\eps_{j}}(t)}{\eps_{j}}\big)^{\top}
\tilde{b}_{xu}^{\eps_{j}}(t)\big(\frac{\delta
u^{\eps_{j}}(t)}{\eps_{j}}\big)
-y_{1}(t)^{\top}b_{xu}(t)v(t)\Big|^2dt\Big)^{\beta/2}\nonumber\\
&\le&C \Big[\me  \Big(  \sup_{t \in [0,T]}\Big|\frac{\delta
x^{\eps_{j}}(t)}{\eps_{j}} \Big|^{2\beta} \Big) \Big]^{\frac12}
\Big[\me  \Big( \int_{0}^{T} \big|\tilde{b}_{xu}^{\eps_{j}}(t)
-\frac{1}{2}b_{xu}(t)\big|^4\Big|\frac{\delta u^{\eps_{j}}(t)}{\eps_{j}}\Big|^4dt\Big)^{\frac{\beta}{2}}\Big]^{\frac12}\nonumber\\
&& + C\Big[\me \Big( \sup_{t \in [0,T]} \Big|\frac{\delta x^{\eps_{j}}(t)}{\eps_{j}} -y_1(t)\Big|\Big)^{2\beta} \Big]^{\frac12} \Big[\me\Big(\int_{0}^{T}\big|\frac{\delta u^{\eps_{j}}(t)}{\eps_{j}}\big|^4 dt \Big)^{\frac{\beta}{2}} \Big]^{\frac12} \nonumber\\
&&+   C \Big[ \me\Big(\sup_{t \in [0,T]}|y_1(t)|^{2 \beta}
\Big)\Big]^{\frac12} \Big[ \me\Big( \int_{0}^{T}\big|\frac{\delta
u^{\eps_{j}}(t)}{\eps_{j}} -v(t)\big|^4 dt
\Big)^{\frac{\beta}{2}}\Big]^{\frac12},
\end{eqnarray*}
which implies that
\begin{equation}\label{equ5 in estimate 2}
\me\Big(\int_{0}^{T}\Big|2\big(\frac{\delta
x^{\eps_{j}}(t)}{\eps_{j}}\big)^{\top}
\tilde{b}_{xu}^{\eps_{j}}(t)\big(\frac{\delta
u^{\eps_{j}}(t)}{\eps_{j}}\big)
-y_{1}(t)^{\top}b_{xu}(t)v(t)\Big|dt\Big)^{\beta}
\to 0,\quad j\to \infty.
\end{equation}
In a similar way, we have
\begin{eqnarray}\label{equ6 in estimate 2}
& &\me\Big(\int_{0}^{T}\Big|\big(\frac{\delta
u^{\eps_{j}}(t)}{\eps_{j}}\big)^{\top}
\tilde{b}_{uu}^{\eps_{j}}(t)\big(\frac{\delta
u^{\eps_{j}}(t)}{\eps_{j}}\big)
-\frac{1}{2}v^{\top}b_{uu}(t)v(t)\Big|dt\Big)^{\beta}\nonumber\\
&\le&C\me\Big(\int_{0}^{T}\Big|\big(\frac{\delta
u^{\eps_{j}}(t)}{\eps_{j}}\big)^{\top}\big(\tilde{b}_{uu}^{\eps_{j}}(t)
-\frac{1}{2}b_{uu}(t)\big)\big(\frac{\delta u^{\eps_{j}}(t)}{\eps_{j}}\big)\Big|^2dt\Big)^{\frac{\beta}{2}}\nonumber\\
& &+C\me\Big[\int_{0}^{T}\Big|\frac{\delta
u^{\eps_{j}}(t)}{\eps_{j}}
-v(t)\Big|^{2}\cdot\Big(\Big|\frac{\delta
u^{\eps_{j}}(t)}{\eps_{j}}\Big|
^2+|v(t)|^2\Big)dt\Big]^{\frac{\beta}{2}}\nonumber\\
&\le&C\me\Big(\int_{0}^{T}\Big|\frac{\delta
u^{\eps_{j}}(t)}{\eps_{j}}\big|^{4}\big|\tilde{b}_{uu}^{\eps_{j}}(t)
-\frac{1}{2}b_{uu}(t)\big|^2dt\Big)^{\frac{\beta}{2}}\nonumber\\
& &+C\me\Big[\int_{0}^{T}\Big|\eps_{j}
h_{\eps_{j}}(t)\Big|^{2}\cdot\Big(\Big| v(t)+ \eps_{j}
h_{\eps_{j}}(t)\Big|^{2}+| v(t)|^2\Big)dt\Big]^{\frac{\beta}{2}}
\to 0,\quad j\to \infty.
\end{eqnarray}
Applying the above method to the diffusion coefficient $\sigma$,
we conclude that
\begin{equation}\label{equ7 in estimate 2}
\me\Big(\int_{0}^{T}\Big|\big(\frac{\delta
x^{\eps_{j}}(t)}{\eps_{j}}\big)^{\top}
\tilde{\sigma}_{xx}^{\eps_{j}}(t)\big(\frac{\delta
x^{\eps_{j}}(t)}{\eps_{j}}\big)
-\frac{1}{2}y_{1}(t)^{\top}\sigma_{xx}(t)y_{1}(t)\Big|^2dt\Big)^{\frac{\beta}{2}}
\to 0,\quad j\to \infty,
\end{equation}
\begin{equation}\label{equ8 in estimate 2}
\me\Big(\int_{0}^{T}\Big|2\big(\frac{\delta
x^{\eps_{j}}(t)}{\eps_{j}}\big)^{\top}
\tilde{\sigma}_{xu}^{\eps_{j}}(t)\big(\frac{\delta
u^{\eps_{j}}(t)}{\eps_{j}}\big)
-y_{1}(t)^{\top}\sigma_{xu}(t)v(t)\Big|^2dt\Big)^{\frac{\beta}{2}}
\to 0,\quad j\to \infty.
\end{equation}
and
\begin{equation}\label{equ9 in estimate 2}
\me\Big(\int_{0}^{T}\Big|\big(\frac{\delta
u^{\eps_{j}}(t)}{\eps_{j}}\big)^{\top}
\tilde{\sigma}_{uu}^{\eps_{j}}(t)\big(\frac{\delta
u^{\eps_{j}}(t)}{\eps_{j}}\big)
-\frac{1}{2}v(t)^{\top}\sigma_{uu}(t)v(t)\Big|^2dt\Big)^{\frac{\beta}{2}}
\to 0,\quad j\to \infty.
\end{equation}
By Lemma \ref{estimatelinearsde}, and using (\ref{frac detax-eps
y1-eps2y2 eps2}), (\ref{equ1add in estimate 2}) and (\ref{equ4 in
estimate 2})--(\ref{equ9 in estimate 2}), we obtain that
\begin{eqnarray*}
& &\me\Big[\sup_{t\in [0,T]}|r_{2}^{\eps_{j}}(t)|^{\beta}\Big] \\
&\leq & C|\varpi_{0}^{\eps_{j}}-\varpi_{0}|^{\beta}
+C\me\Big(\int_{0}^{T}\Big|b_{u}(t)\big(h_{\eps_{j}}(t)-h(t)\big) \Big|dt\Big)^{\beta}\\
& & +C \me\Big(\int_{0}^{T}\Big|\sigma_{u}(t)\big(h_{\eps_{j}}(t)-h(t)\big) \Big|^2dt\Big)^{\frac{\beta}{2}}\\
& & + C\me\Big(\int_{0}^{T}\Big|\big(\frac{\delta
x^{\eps_{j}}(t)}{\eps_{j}}\big)^{\top}
\tilde{b}_{xx}^{\eps_{j}}(t)\big(\frac{\delta
x^{\eps_{j}}(t)}{\eps_{j}}\big)
-\frac{1}{2}y_{1}(t)^{\top}b_{xx}(t)y_{1}(t)\Big|dt\Big)^{\beta}\\
& &+C\me\Big(\int_{0}^{T}\Big|2\big(\frac{\delta
x^{\eps_{j}}(t)}{\eps_{j}}\big)^{\top}
\tilde{b}_{xu}^{\eps_{j}}(t)\big(\frac{\delta
u^{\eps_{j}}(t)}{\eps_{j}}\big)
-y_{1}(t)^{\top}b_{xu}(t)v(t)\Big|dt\Big)^{\beta}\\
& &+C\me\Big(\int_{0}^{T}\Big|\big(\frac{\delta
u^{\eps_{j}}(t)}{\eps_{j}}\big)^{\top}
\tilde{b}_{uu}^{\eps_{j}}(t)\big(\frac{\delta
u^{\eps_{j}}(t)}{\eps_{j}}\big)
-\frac{1}{2}v(t)^{\top}b_{uu}(t)v(t)\Big|dt\Big)^{\beta}\\
& &+C\me\Big(\int_{0}^{T}\Big|\big(\frac{\delta
x^{\eps_{j}}(t)}{\eps_{j}}\big)^{\top}
\tilde{\sigma}_{xx}^{\eps_{j}}(t)\big(\frac{\delta
x^{\eps_{j}}(t)}{\eps_{j}}\big)
-\frac{1}{2}y_{1}(t)^{\top}\sigma_{xx}(t)y_{1}(t)\Big|^2dt\Big)^{\frac{\beta}{2}}\\
& &+C\me\Big(\int_{0}^{T}\Big|2\big(\frac{\delta
x^{\eps_{j}}(t)}{\eps_{j}}\big)^{\top}
\tilde{\sigma}_{xu}^{\eps_{j}}(t)\big(\frac{\delta
u^{\eps_{j}}(t)}{\eps_{j}}\big)
-y_{1}(t)^{\top}\sigma_{xu}(t)v(t)\Big|^2dt\Big)^{\frac{\beta}{2}}\\
& &+C\me\Big(\int_{0}^{T}\Big|\big(\frac{\delta
u^{\eps_{j}}(t)}{\eps_{j}}\big)^{\top}
\tilde{\sigma}_{uu}^{\eps_{j}}(t)\big(\frac{\delta
u^{\eps_{j}}(t)}{\eps_{j}}\big)
-\frac{1}{2}v(t)^{\top}\sigma_{uu}(t)v(t)\Big|^2dt\Big)^{\frac{\beta}{2}}\\
& &\to 0,\quad j\to \infty.
\end{eqnarray*}
This proves (\ref{r2 to 0}).  The sequence $\varepsilon_j \to 0^+$
being arbitrary, the proof is complete.

\end{proof}

\section*{Acknowledgment}
The authors highly appreciate the constructive comments of two
anonymous referees  which led to several improvements of the
original version.

\end{document}